\documentclass[11pt]{article}
\usepackage{impostazioni}
\usepackage[top=1in, bottom=1.25in, left=1.25in, right=1.25in]{geometry}
\usepackage{amsmath}

\usepackage[inline]{enumitem}
\makeatother

\begin{document}

\title{{Polytopic Discontinuous Galerkin methods for the numerical modelling of flow in porous media with networks
of intersecting fractures \footnote{Paola F. Antonietti and Chiara Facciol\`a have been supported by
SIR Project n. RBSI14VT0S \lq \lq PolyPDEs: Non-conforming polyhedral finite
element methods for the approximation of partial differential equations\rq\rq ~ funded by MIUR. 
 Marco Verani has been partially supported by the Italian research grant  {\sl Prin 2012}  2012HBLYE4  \lq\lq Metodologie innovative nella modellistica differenziale numerica\rq\rq ~ and by INdAM-GNCS.}}}

\author{Paola F. Antonietti$^{\#}$, Chiara Facciol\`a$^{\#}$ and Marco Verani$^{\#}$}

\maketitle

\begin{center}
{\small
$^{\#}$ MOX- Laboratory for Modeling and Scientific Computing\\ Dipartimento di Matematica \\ Politecnico di Milano\\ Piazza Leonardo da Vinci 32, 20133 Milano, Italy\\
\vskip 0.1cm
{\tt paola.antonietti@polimi.it, chiara.facciola@polimi.it, marco.verani@polimi.it}
}
\end{center}
%


\pagestyle{myheadings}
\thispagestyle{plain}
\markboth{}
        {}

\begin{abstract}
We present a numerical approximation of Darcy's flow through a porous medium that incorporates networks of fractures with non empty intersection. Our scheme employs PolyDG methods, i.e. discontinuous Galerkin methods on general polygonal and polyhedral (polytopic, for short) grids, featuring elements with edges/faces that may be in arbitrary number
(potentially unlimited) and whose measure may be arbitrarily small. Our approach is
then very well suited to tame the geometrical complexity featured by most of applications
in the computational geoscience field. From the modelling point of view, we adopt a reduction strategy that treats
fractures as manifolds of codimension one and we employ the primal version of Darcy's law to describe the flow in both the bulk and in the fracture network. In addition, some physically consistent conditions couple the two problems, allowing for jump of pressure at their interface, and they as well prescribe the behaviour of the fluid along the intersections, imposing pressure continuity and flux conservation.
Both the bulk and fracture discretizations are obtained employing the Symmetric
Interior Penalty DG method
extended to the polytopic setting. The key instrument to obtain a polyDG approximation
of the problem in the fracture network is the generalization of the concepts of
jump and average at the intersection, so that the contribution from all the fractures
is taken into account. We prove
the well-posedness of the discrete formulation and perform an error analysis obtaining a
priori $hp$-error estimates. All our theoretical results are validated performing preliminary numerical
tests with known analytical solution.

\end{abstract}

\section{Introduction}
This work is concerned with the simulation of Darcean flows through porous media that incorporate networks of fractures with non empty intersection. The focus is on the development and analysis of a numerical approximation that employs PolyDG methods, i.e. discontinuous Galerkin methods on general polygonal and polyhedral (polytopic, for short) grids. 
In the past decades, increased attention has be given to the efficient implementation of numerical methods for fractured reservoir simulations. The analysis and prediction of the flow is indeed fundamental in many environmental and energy engineering applications, which include petroleum extraction, CO2 storage in depleted oil fields, isolation of radioactive waste and geothermal energy production, for example. In all the aforementioned applications, fractures severely affect the flow,
since they can act as barriers for the fluid (when they are filled with low permeable material), or as conduits (when they have higher permeability than the surrounding medium). Moreover, in many cases the geometry of the fault system can be highly intricate, featuring thousands of fractures, which may also intersect with small angles or be nearly coincident \cite{formaggiascottisottocasa}.

A first step towards a reduction in the complexity of the simulation is usually taken in the conceptual modelling of the flow:
since fractures usually present a small width-to-length ratio, as well as a small relative size with respect to the domain, a popular choice consists in treating them as as manifolds of codimension one. The development of this kind of reduced models has been addressed for single-phase flows in several works, for example in \cite{Alboin, Alboin2, martinjaffreroberts, Frih}. Our main reference will be the model described in \cite{martinjaffreroberts}, which considers for simplicity the case of a single fracture non-immersed in the bulk domain. In the model, the flow in the bulk is governed by Darcy's law, whereas a suitable dimensionally reduced version of the law is formulated on the surface modelling the fracture. Moreover, the exchange of fluid between the fracture and the porous medium is described via some physically consistent coupling conditions. We also remark that both low and large permeable fractures can be handled.  

Although the use of dimensionally reduced models avoids the requirement of extremely refined grids inside the fracture domains, realistic simulations still call for high mesh resolution in these areas, so that all the small geometrical features can be captured without resorting to low-quality elements in the classical sense. The mesh generation process within a classical finite elements approach can then represent a bottleneck in the whole simulation, especially in 3D, as only computational grids composed by tetrahedral/hexahedral/prismatic elements are supported. The same issue can be encountered in many application areas, ranging from fluid-structure interaction, to wave propagation problems, to name a few. This has motivated a huge effort in the past years in the design of numerical methods supporting meshes made of general
polytopic elements. A huge reduction on the computational cost may be achieved by resorting to hybrid mesh generation techniques, for example: first a (possibly structured) grid is generated independently of geometric features (e.g., fractures), secondly the elements are cut according to the required pattern. It follows that the final mesh contains arbitrarily shaped elements in the surrounding of such features and is regular far from them.
In addition to the simplicity of the procedure described, polytopic meshes present, on average, a much lower number of elements, even on relatively simple geometries, without committing a variational crime \cite{antonietti2013hp,antonietti2014domain}.

Within this framework, many numerical methods have been developed on top of polytopic meshes in the context of flows in fractured porous media. In particular, we mention \cite{AnFoScVeNi2016,formaggiascottisottocasa}, where a mixed approximation based on Mimetic Finite Differences was applied; the works \cite{BerroneVEM1,BerroneVEM2}, where virtual elements were employed to deal with flows in Discrete Fracture Networks, and  \cite{florent}, which uses the Hybrid High-Order method.
We also remark that an important alternative is given by the use of non-conforming discretizations, in which the bulk grid can be chosen fairly regular independently from the fractures, since these are considered immersed in the geometry. We refer to \cite{dangeloscotti,fumagalliscotti,reviewXFEMfratture} for the use of the eXtended Finite Element Method and to \cite{cutFem} for the Cut Finite Element Method.

In \cite{mioUnafrattura} we presented an approximation of the coupled bulk-fracture problem that employs PolyDG methods. The inherited flexibility of DG methods in handling arbitrarily shaped, non-necessarily matching, grids and elementwise variable
polynomial orders represents, in fact, the ideal setting to handle such kind of problems
that typically feature a high-level of geometrical complexity. In particular, since they employ local polynomial spaces defined
elementwise without any continuity constraint, DG methods feature a high-level of
intrinsic parallelism.  Furthermore, the lack of continuity between
neighbouring elements allows for the employment of extremely broad families of meshes containing elements with edges/faces that may be in arbitrary number (potentially unlimited) and whose measure may be arbitrarily small; cf. \cite{AntoniettiFacciolaHoustonMazzieriPennesiVerani_2020} for a comprehensive review on PolyDG methods on polyhedral grids for geophysical applications, including seismic wave propagation and fractured reservoir simulations.
The geometric flexibility highlighted so far is not the only motivation to employ such techniques in the context of fractured porous media. A more physically motivated argument is given by the discontinuous nature of the solution at the matrix-fractures interface, which can be intrinsically captured in the choice of the discrete spaces. In addition, the bulk and fractures coupling can be easily reformulated by means of the jump and average operators, which are a fundamental tool in the development of DG formulations, and thus naturally incorporated into the variational formulation. Finally, the abstract setting proposed in \cite{Arnoldbrezzicockburnmarini2002}, based on the flux-formulation, allows for the introduction of a unified framework where primal or mixed formulations can be chosen independently for both bulk and fractures, depending on the application at hand and on the quantities of interest. We refer to \cite{mioUnified} for further details on the unified analysis and to \cite{mioUnafrattura} for a focus on the primal-primal framework. In both cases, our analysis was carried on in the simplified setting of a single, non-immersed fracture. The purpose of the present work is to extend our formulation to networks of \emph{intersecting} fractures.
For simplicity, we consider, as in \cite{mioUnafrattura}, the primal-primal setting, so that we can mainly focus on handling the intersections. To this aim, we supplement the mathematical model \cite{martinjaffreroberts} with some suitable physical conditions at the intersections, prescribing the behaviour of the fluid. Following \cite{formaggiascottisottocasa,brennerhennicker, BerroneVEM2}, we impose that:
\begin{itemize}
\item pressure between fractures is continuous along the intersections;
\item flux is conserved, so that no exchange of fluid between bulk and fracture network takes place along the intersections.
\end{itemize}
We mention that more general conditions, where the angle between fractures is taken into account and jumps of
pressure across the intersection are allowed, may be imposed. Some examples can be found in \cite{flemischWohlmuth,ruffo,boon2018robust,chernyshenko2019unfitted,fumagallikeiscialo}. We also mention that the analysis of the mixed-mixed setting in the case of a totally immersed network of fractures has been addressed in
 \cite{formaggiascottisottocasa}.

From the DG-discretization point of view, the key instrument for dealing with intersections is the generalization of the concepts of jump and average. If we assume that the fracture network may be approximated by the union of $N_\Gamma$ fractures $\gamma_k$, each of which is a one co-dimensional planar manifold, i.e. $\Gamma = \bigcup_{k=1}^{N_\Gamma} \gamma_k$, the intersections correspond to lines when $d=3$ and to points when $d=2$. Let us focus for simplicity on the case $d=3$, see Figure~\ref{intro:fig: fracture network 3d} for an example. Here, the intersection line is denoted by $\mathcal{I}_\cap$ and each fracture $\gamma_k$, $k=1,2,3,4$, is characterised by the outward normal vector $\boldsymbol{\tau}_k$ at the intersection, which belongs to the plane containing the fracture.
\begin{figure}
\centering
\includegraphics[scale=0.65]{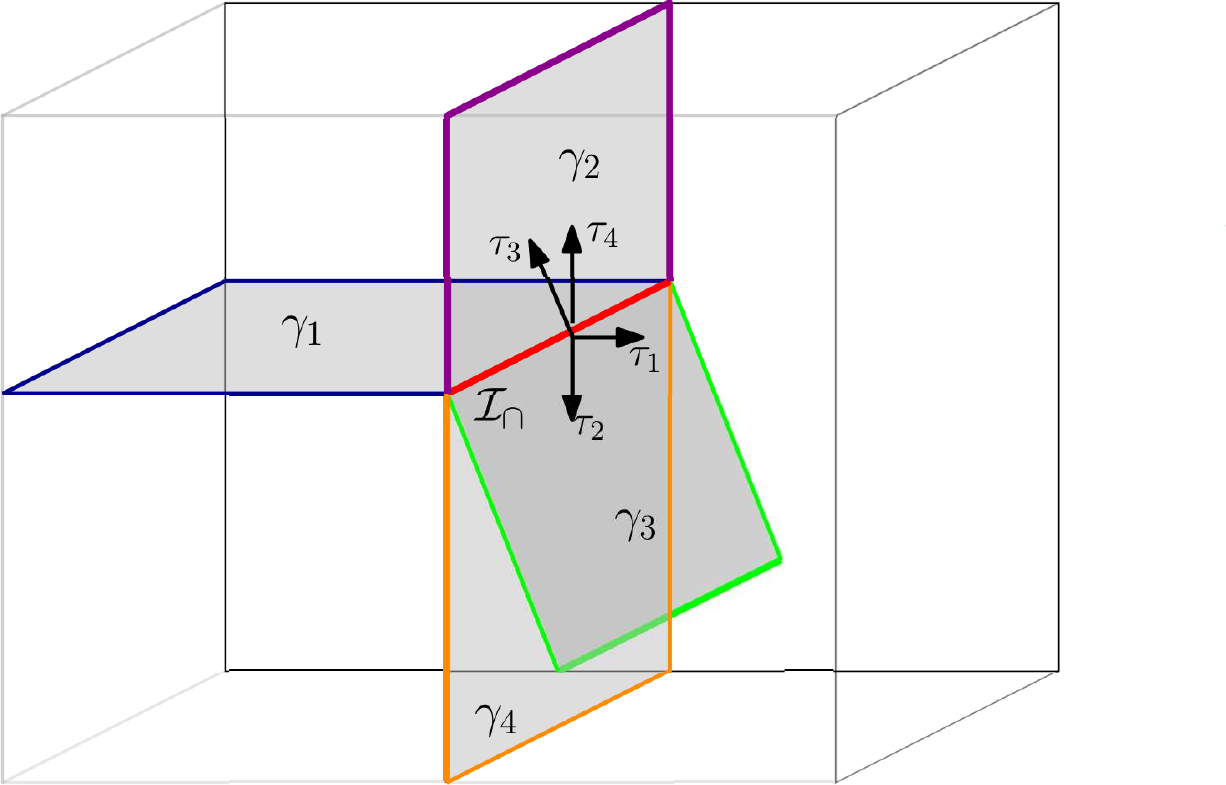}
\caption{Example of network of intersecting fractures and corresponding normal vectors for $d=3$.}
\label{intro:fig: fracture network 3d}
\end{figure}
In order to describe the pressure field in all the network, we employ the global variable $p_\Gamma=(p_\Gamma^1, \dots, p_\Gamma^{N_\Gamma})$, defined in a suitable product space of all the local fracture spaces. Our aim is to introduce some operators that are able to capture the behaviour of the function $p_\Gamma$ across the intersection line, taking into account the contribution from all the fractures, similarly to how classic jump and average operators \cite{Arnoldbrezzicockburnmarini2002}
 describe the discontinuity of a piecewise-continuous function across elemental interfaces. The main difference with respect to the standard case is that the normal vectors, contained into the definition of the operators, are not aligned. This is related to the linear DG approximation of elliptic PDEs on surfaces presented in \cite{DGcurvilineari}, then extended to high order in \cite{DGcurvihighorder}. Here, the surface is approximated by a piecewise linear surface composed of planar triangles, so that a new definition of jump and average operators is needed, to take into account the fact that the outward normal vectors of two neighbouring triangles are not, in general, opposite. Our definition is a further generalization, since it considers the intersection of an arbitrary number of planar surfaces.
Using the newly defined jump and average operators we are able to define a DG approximation for the problem in the bulk combined with a DG approximation for the problem in the fracture network, where the conditions at the intersection are imposed \lq\lq in the spirit of DG methods \rq\rq. In particular, this means that continuity is enforced \emph{penalizing the jump} of the pressure (after a suitable definition of the penalization coefficient at the intersection), while balance of fluxes is imposed \emph{naturally}, similarly to how homogeneous Neumann boundary conditions are usually enforced.
Both the bulk and fracture discretizations are obtained employing the SIPDG method extended to the polytopic setting. 

The rest of the paper is structured as follows. In Section~\ref{ntw:sec:math model}, we present the mathematical model. In Section~\ref{ntw:sec:weak form}, we introduce the weak formulation of the problem and prove its well-posedness. Section~\ref{ntw:sec:DG discr} contains the polyDG discretization of the coupled system based on the new definition of jump and average operators at intersections, which is also introduced in Section \ref{ntw:sec:DG discr}. Finally, Sections~\ref{ntw:sec:well posed} and \ref{ntw:sec: error} enclose the stability and error analysis of the discrete method. We conclude with Section~\ref{ntw:sec:numerical exp}, where we present some preliminary numerical experiments with known analytical solution, so that we are able to verify the obtained convergence rates.

\section{Mathematical model}\label{ntw:sec:math model}
We consider the domain $\Omega \subset \R^d$, with $d=2,3$, representing the porous medium. We assume that the fracture network may be approximated by a collection of one co-dimensional planar manifolds $\Gamma \subset \R^{d-1}$,  
 adopting the reduced model introduced in \cite{martinjaffreroberts} and extended to fracture networks in \cite{formaggiascottisottocasa,brennerhennicker,BerroneVEM2}. 
In particular, we consider $\Gamma$ to be the union of $N_\Gamma$ fractures $\gamma_k$,
\begin{equation}
\Gamma = \bigcup_{k=1}^{N_\Gamma} \gamma_k,
\end{equation}
with every $\gamma_k$ being an open, bounded, connected, planar $(d-1)$-dimensional orientable manifold. Each $\gamma_k$ is, in fact, the approximation of the actual fracture $\tilde{\gamma}_k$, which we assume may be characterized by
\begin{equation}
\tilde{\gamma}_k = \{ \textbf{x}+ d\textbf{n}_k, \, \mbox{for} \, \textbf{x} \in \gamma_k, \, d \in (- \frac{\ell_k(\textbf{x)}}{2}, \frac{\ell_k(\textbf{x})}{2} ) \},
\end{equation}
where $\textbf{n}_k$ is a unit normal vector to $\gamma_k$, whose precise definition is given below, and $\ell_k(\textbf{x)}$ is a $\mathcal{C}^1$ function that describes the fracture aperture. For all $k=1, \dots, N_\Gamma$, we assume there is a constant $\ell_*>0$ such that $\ell_k>\ell_*$. Finally, we denote by $\ell_\Gamma$ the aperture of the whole fracture network, meaning that $\ell_\Gamma|_{\gamma_k} =\ell_k$.

Without loss of generality for the analysis (see Remark \ref{generalizz}), we can assume that:
 \begin{enumerate}
 \item[(i)] the fracture network is connected; 
 \item[(ii)] all the fractures intersect in one point if $d=2$ or line if $d=3$;
 \item[(iii)] for each fracture, the intersection point corresponds to one of its endpoints if $d=2$ or to part of one of its facets if $d=3$.
 \end{enumerate}
We denote by $\mathcal{I}_\cap$ the intersection point/line, i.e.,
\begin{equation}
\mathcal{I}_\cap= \bigcap_{k=1}^{N_\Gamma} \bar{\gamma}_k.
\end{equation}
We assume that the angle between intersecting fractures is bounded from below, as well as
the angles between fractures and $\partial \Omega$, whenever a fracture touches the boundary. This implies, in particular,
that the number of fractures joining at the intersection is bounded.

We assume that the boundary of the bulk domain may be subdivided into two measurable subsets for the imposition of boundary conditions on the pressure and on the Darcy's velocity, that is $\partial \Omega = \partial \Omega_D \cup \partial \Omega_{N}$, with $|\partial \Omega_D|>0$. 
 This induces a subdivision of the boundary of each fracture into four different sets, some of which may be empty: $\partial \gamma_k^D = \partial \gamma_k \cap \partial \Omega_D$, $\partial \gamma_k^{N} = \partial \gamma_k \cap \partial \Omega_{N}$, the intersection tips $\partial \gamma_k ^\cap = \bigcup_{\substack{j=1 \\ j \neq k}}^{N_\Gamma}(\partial \gamma_k \cap \partial \gamma_j)$ and finally $\partial \gamma_k^F = \partial \gamma_k \setminus (\partial \gamma_k^D \cup \partial \gamma_k^{N} \cup \partial \gamma_k^\cap)$, which corresponds to the set of immersed tips. We also introduce the corresponding definitions for the network $\partial \Gamma_D =  \bigcup_{k=1}^{N_\Gamma} \partial \gamma_k^D$, $\partial \Gamma_N =  \bigcup_{k=1}^{N_\Gamma} \partial \gamma_k^N$, $\partial \Gamma_\cap =  \bigcup_{k=1}^{N_\Gamma} \partial \gamma_k^\cap$ and $\partial \Gamma_F =  \bigcup_{k=1}^{N_\Gamma} \partial \gamma_k^F$.
Some of these sets may as well be empty, and also the case of totally immersed network, i.e., $\partial \Gamma_D \cup \partial \Gamma_{N} = \emptyset$ is admitted. See Figure~\ref{fig:notazione insiemi bordo fratture 2d}-\ref{fig:notazione insiemi bordo fratture 3d} for an explicative example of the notation.
\begin{figure} 
\subfigure[]{
\label{fig:notazione insiemi bordo fratture 2d}
   \includegraphics[scale=0.65]{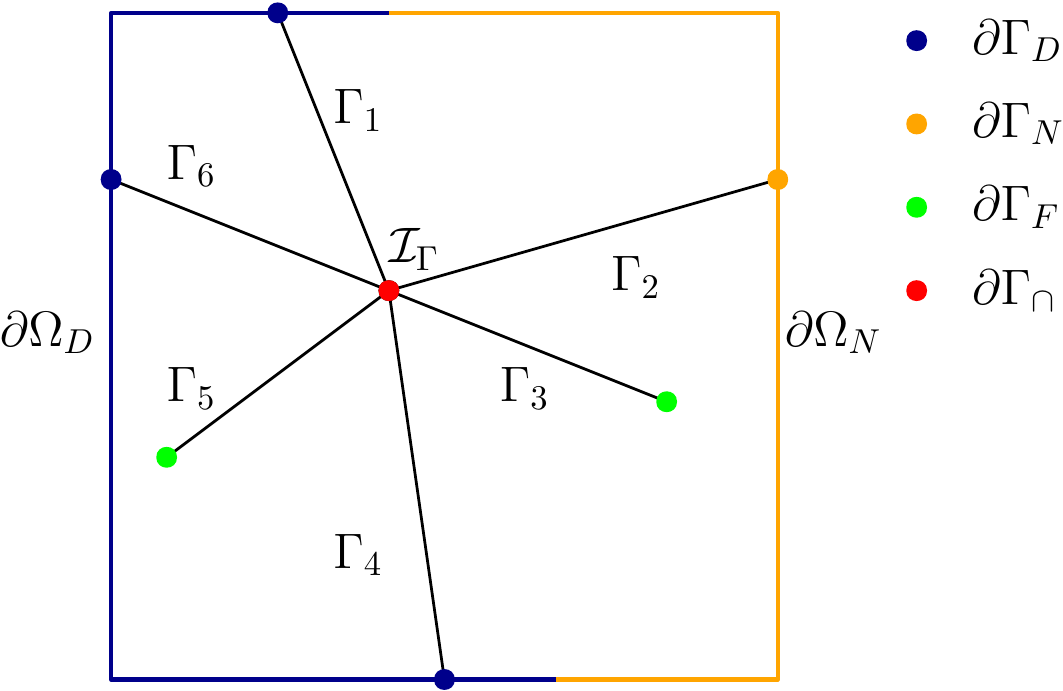}
}
\subfigure[]{
\label{fig:notazione insiemi bordo fratture 3d}
 \includegraphics[scale=0.55]{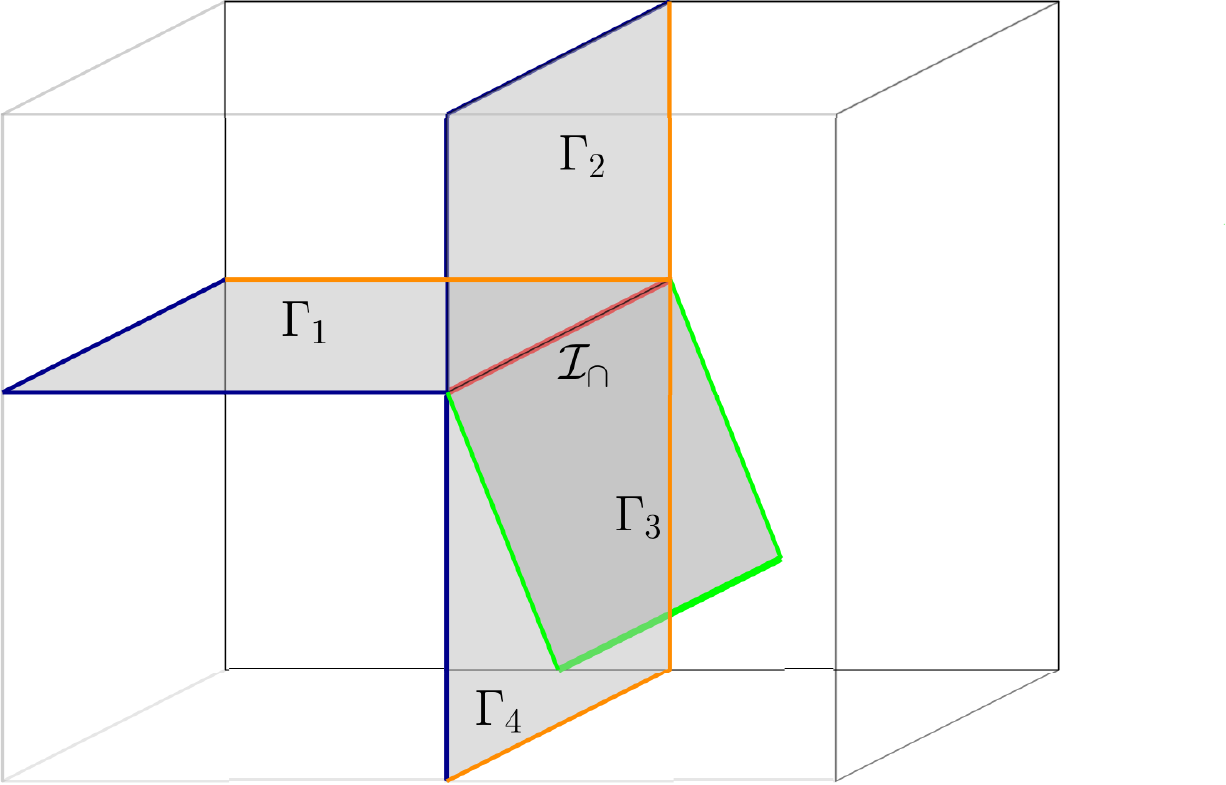} 
}
\caption{Example of fracture network satisfying the geometrical assumptions with subdivision of the boundary into sets for $d=2$ (left) and $d=3$ (right). }
\end{figure}
\begin{figure}
\subfigure[]{
\includegraphics[scale=0.6]{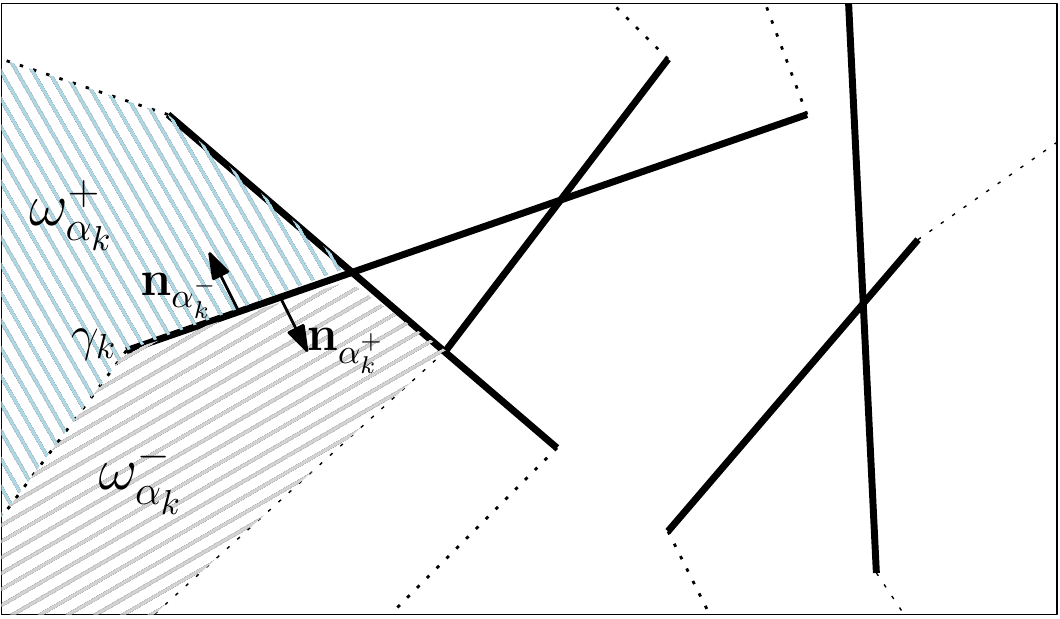}
\label{fig:partizione omega}
}
\subfigure[]{
\includegraphics[scale=0.6]{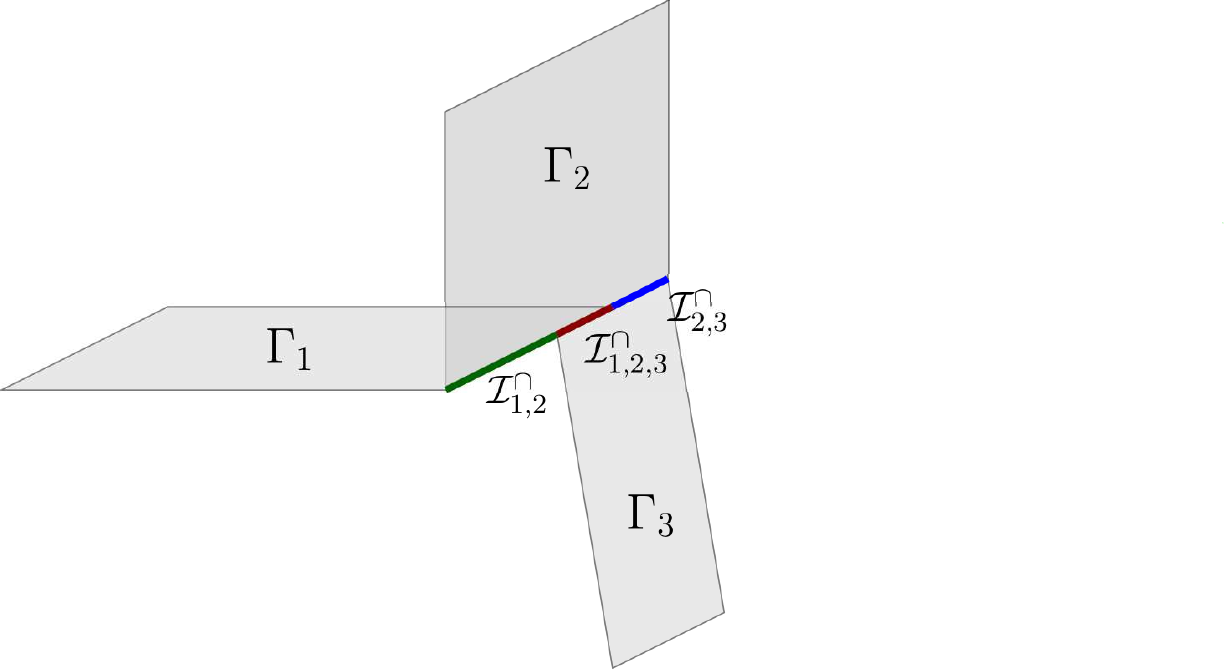}
\label{fig:multiple intersection 3d}
}
\caption{Figure~\ref{fig:partizione omega}: Partition of the domain $\Omega$ into subdomains $\omega_j$ induced by the prolongation of fractures. 
Figure~\ref{fig:multiple intersection 3d}: Example of multiple intersections for $d=3$, where an intersection is defined as a segment shared by a fixed subset of fractures. Here, we can define 3 intersections,  $\mathcal{I}_{1,2}^\cap = \partial \Gamma_1 \cap \partial \Gamma_2$, $\mathcal{I}_{1,2,3}^\cap = \partial \Gamma_1 \cap \partial \Gamma_2 \cap \partial \Gamma_3$ and $\mathcal{I}_{2,3}^\cap= \partial\Gamma_2 \cap \partial \Gamma_3$}
\end{figure}

Following the same strategy as in \cite{Angot,formaggiascottisottocasa,brennerhennicker}, we assume that the fractures can be suitably extended so that the domain $\Omega$ is partitioned into a collection of Lipschitz subdomains $\omega_j$, with $j=1,\dots ,N_\omega$, i.e., $\Omega= \bigcup_{j=1}^{N_\omega} \omega_j$, cf Figure~\ref{fig:partizione omega}.
By construction, for each fracture $\gamma_k$ we have exactly two subdomains, $\omega_{\alpha^+}$ and $\omega_{\alpha^-}$, such that $\gamma_k \subset \partial \omega_{\alpha^+}\cap \partial \omega_{\alpha^-}$. 
This implies that we can identify for each fracture $\gamma_k$ the normal $\textbf{n}_k$ defined as $\textbf{n}_k = \textbf{n}_{\alpha^+} = -\textbf{n}_{\alpha^-}$, where $\textbf{n}_\alpha$ is the unit normal vector pointing outward of the subdomain $\omega_\alpha$. Moreover, we denote by $\textbf{n}_\Gamma$ the normal to the whole fracture network, meaning that $\textbf{n}_\Gamma= \textbf{n}_k$ on $\gamma_k$.

Then, for a regular-enough scalar-valued function $q$ defined on $\Omega$, we can introduce \emph{jump} and \emph{average} across the fracture $\gamma_k \subset \partial \omega_{\alpha_k^+}\cap \partial \omega_{\alpha_k^-}$ in a standard way as
\begin{equation} \label{salto media scalar}
\llbracket q \rrbracket_{\gamma_k} = q_{\alpha_k^+} \textbf{n}_{\alpha_k^+} + q_{\alpha_k^-}  \textbf{n}_{\alpha_k^-}, \quad \quad \{q\}_{\gamma_k}= \frac{1}{2}(q_{\alpha_k^+} + q_{\alpha_k^-}), 
\end{equation} 
where $q_{\alpha_k^+}$ and $q_{\alpha_k^-}$ are the restriction to $\gamma_k$ of the traces of $q$ on $\partial \omega_{\alpha_k^+}$ and $\partial \omega_{\alpha_k^-}$, respectively. We refer to \cite{brennerhennicker} and to \cite{Angot} for a rigorous definition of the trace operators, also in the case of immersed tips. Similarly, for a regular-enough vector valued function $\textbf{v}$, we define
\begin{equation} \label{salto media vector}
\llbracket \textbf{v}  \rrbracket_{\gamma_k} = \textbf{v}_{\alpha_k^+} \cdot \textbf{n}_{\alpha_k^+} + \textbf{v}_{\alpha_k^-} \cdot \textbf{n}_{\alpha_k^-}, \quad \quad \{\textbf{v}\}_{\gamma_k}= \frac{1}{2}(\textbf{v}_{\alpha_k^+} + \textbf{v}_{\alpha_k^-}).
\end{equation}

Moreover, for given functions $f_k$ defined on $\gamma_k$, with $k=1, \dots, N_\Gamma$, we define the function $f_\Gamma$ on the network $\Gamma$, in the sense of product spaces, as  $f_\Gamma = \prod_{k=1}^{N_\Gamma} f_k$. 
We can then define the jump and average of a function $q$ across the fracture network as 
$\llbracket q  \rrbracket_\Gamma = \prod_{k=1}^{N_\Gamma} \llbracket q  \rrbracket_{\gamma_k}$ and $\{q\}_\Gamma= \prod_{k=1}^{N_\Gamma} \{ q  \}_{\gamma_k}$, respectively.

\begin{oss}\label{generalizz}
We remark that the geometric hypotheses on the fracture network were made only for the sake of simplicity and the analysis can be easily extended to more general configurations. 
More precisely, the case of a network featuring multiple connected components can be treated analogously, as long as the partition of $\Omega$ into subdomains $\omega_\alpha$ is aligned with all of them.
The case of multiple intersections is an easy extension when $d=2$, and the same holds true when $d=3$ if
we define an intersection as a segment shared by a fixed subset of fractures (see Figure~\ref{fig:multiple intersection 3d} for an explicative example). Note that we do not need to impose any condition at the point shared by two intersections, since we are assuming that no flux is present along the intersections.
\end{oss}

\subsection{Governing equations}
In what follows, we present the governing equations for our model. In accordance with our previous works, \cite{mioUnafrattura,mioUnified}, we take as a reference the model for single-phase flow derived in \cite{martinjaffreroberts}, where fractures are treated as $(d-1)$-dimensional interfaces between $d$-dimensional subdomains. In particular, we adopt the extension of the above model to fracture networks developed in \cite{formaggiascottisottocasa,brennerhennicker}. 

The flow of an incompressible fluid through a fractured $d$-dimensional porous medium, $d=2,3$, may be described by four elements:
\begin{enumerate}
\item \emph{Governing equations for the flow in the porous medium}:\\
We assume that the flow is governed by Darcy's law. We denote by $p$ the fluid pressure and by $\bnu$ the permeability tensor, which we assume to include also the dependency on the viscosity. Given a function $f \in L^2(\Omega)$ representing a source term and a function $g \in H^{-1/2}(\partial \Omega_D)$, the motion of the fluid in the bulk may be then described by the following equations 
\begin{align} \label{pb bulk}
- \nabla \cdot (\boldsymbol{\nu} \nabla p)  &= f &\mbox{in} \; \Omega \setminus \bar{\Gamma},  \\
p &= g_{D}  &\mbox{on} \; \partial \Omega_D,  \\
\boldsymbol{\nu} \nabla p \cdot \textbf{n} &=0 &\mbox{on} \; \partial \Omega_{N}, 
\end{align}
where $\textbf{n}$ is the unit normal vector pointing outward of $\Omega$. We also make some regularity assumptions on the tensor $\bnu = \bnu(\textbf{x}) \in \R^{d \times d}$, requiring that it is symmetric, positive definite, uniformly bounded from below and above and with entries that are bounded, piecewise continuous real-valued functions.

\item \emph{Governing equations for the flow in the fracture network}:\\
Darcy's law is used also for modelling the flow along the fractures. In order to obtain a reduced model, where fractures are $(d-1)$-dimensional objects immersed in a $d$-dimensional domain, a process of integration of the equations across the fracture aperture $\ell_\Gamma$ is carried on, see \cite{martinjaffreroberts}. Reduced variables for the average pressure $p_\Gamma =(p_{\Gamma}^1, \dots, p_\Gamma^{N_\Gamma})$ are then defined on each fracture. The flow is also characterized by the permeability tensor $\bnu_\Gamma=(\bnu_\Gamma^1, \dots, \boldsymbol{\nu}_{\Gamma}^{N_\Gamma})$, scaled by viscosity. It is assumed that, on each fracture, $\bnu_{\Gamma}^k$ has a block-diagonal structure of the form
\begin{equation} \label{K frattura}
\boldsymbol{\nu}_{\Gamma}^k = \begin{bmatrix} \nu_{\gamma_k}^n &  0 \\ 0 &
\boldsymbol{\nu}_{\gamma_k}^{\tau}
\end{bmatrix},
\end{equation}
when written in its normal and tangential components, $k=1, \dots, N_\Gamma$. Here, $\boldsymbol{\nu}_{\gamma_k}^{\tau} \in \R^{(d-1) \times (d-1)}$ is a positive definite, uniformly bounded tensor (it reduces to a positive number for $d=2$) representing the tangential component of the permeability of the fracture $\gamma_k$.  Given  a source term $f_{\Gamma}= (f_\Gamma^1, \dots, f_\Gamma^{N_\Gamma}) \in  \prod_{k=1}^{N_\Gamma} L^2(\gamma_k)$ and $g_{\Gamma} \in H^{1/2}(\partial \Gamma_D)$, the governing equations for the fracture flow read
\begin{align} \label{pb frattura}
 -\nabla_{\tau} \cdot (\nug \nabla_{\tau}p_{\Gamma}) &=  \ell_\Gamma f_{\Gamma} - \llbracket \bnu \nabla p  \rrbracket  &\mbox{in} \; \Gamma, \\
 p_{\Gamma} &= g_{\Gamma} &\mbox{on} \; \partial\Gamma_D, \\
(\boldsymbol{\nu}_{\Gamma}^{\tau} \ell_\Gamma \nabla_{\tau}p_{\Gamma}) \cdot \boldsymbol{\tau} &=0 &\mbox{on} \; \partial\Gamma_{N}, \\
  (\boldsymbol{\nu}_{\Gamma}^{\tau} \ell_\Gamma \nabla_{\tau}p_{\Gamma}) \cdot \boldsymbol{\tau} &=0 &\mbox{on} \; \partial\Gamma_F,
\end{align} 
Here, $\boldsymbol{\tau}=(\btau_1, \dots, \btau_{N_\Gamma})$ is defined on each fracture $\gamma_k$ as the vector in its tangent plane normal to $\partial \gamma_k$, while $\nabla_{\tau}$  and $\nabla_{\tau} \cdot$ denote the tangential gradient and divergence operators, respectively. Note that, when a certain operator is written on quantities defined on the whole network $\Gamma$, it should be interpreted as the product of the corresponding operators on each fracture $\gamma_k$.

For the condition on the immersed fracture tips, we have taken as a reference  \cite{Angot}, where the
model developed in \cite{martinjaffreroberts} has been extended to fully immersed fractures. In particular, we have imposed a homogeneous conditions for the flux, stating that the mass transfer across the immersed tips can be neglected in front of the transversal one.

\item \emph{Coupling conditions between bulk and fractures along their interfaces}:\\
Following \cite{martinjaffreroberts}, we provide the interface conditions to account for the exchange of 
fluid between the fractures and the porous medium:
\begin{align}\label{CC}
 - \{ \boldsymbol{\nu} \nabla p\} \cdot \textbf{n}_\Gamma  & = \beta_{\Gamma} \llbracket  p  \rrbracket \cdot \textbf{n}_\Gamma &\mbox{on} \; \Gamma, \\
 -\llbracket \boldsymbol{\nu} \nabla p \rrbracket  &= \alpha_{\Gamma}(\{p\}-p_{\Gamma})  &\mbox{on} \; \Gamma,
\end{align}
where we have introduced 
\begin{equation}\label{def alpha beta eta}
\beta_{\Gamma} = \frac{1}{2 \eta_{\Gamma}},\quad \quad
\alpha_{\Gamma}=\frac{2}{\eta_{\Gamma}(2 \xi-1)},\quad \quad
\eta_{\Gamma} = \frac{\ell_{\Gamma}}{\nu_{\Gamma}^n},
\end{equation}
with $\nu_{\Gamma}^n $ being the normal component of the fracture permeability tensor, see \eqref{K frattura}. Note that these conditions depend on the closure parameter $\xi \neq \frac{1}{2}$, which is related to the assumption made on the pressure profile across the fracture aperture when deriving the reduced model.

\item \emph{Conditions at the intersection}:\\
Finally, following \cite{formaggiascottisottocasa,brennerhennicker,BerroneVEM2}, at the fracture intersection $\mathcal{I}_\cap$ we enforce pressure continuity and flux conservation:

\begin{subequations} \label{eq intersezione}
\begin{align} 
p_\Gamma^1=p_\Gamma^2= \dots &=p_\Gamma^{N_\Gamma}  &\mbox{in} \; \mathcal{I}_\cap,\label{int pressure continuity}\\
\sum_{k=1}^{N_\Gamma} \bnu_{\gamma_k}^{\tau} \ell_k \nabla_{\tau}p_{\Gamma}^k \cdot \boldsymbol{\tau}_k &=0 &\mbox{in} \; \mathcal{I}_\cap.\label{int no flux}
\end{align}
\end{subequations}

We remark that other possible, more general, conditions might be imposed at the intersection. Some examples may be found in \cite{flemischWohlmuth,ruffo,boon2018robust,chernyshenko2019unfitted,fumagallikeiscialo}, where the angle between fractures at the intersection is included in the model and jumps of pressure across the intersection are allowed.
\end{enumerate}

\section{Weak formulation}\label{ntw:sec:weak form}
In this section we introduce the weak formulation of the model problem \eqref{pb bulk}-\eqref{pb frattura}-\eqref{CC}-\eqref{eq intersezione} and prove its well-posedness.

For the sake of simplicity we will assume that \emph{homogeneous} Dirichlet boundary conditions are imposed for both the bulk and fracture problems, i.e., $g_D=0$ and $g_{\Gamma}=0$. The extension to the general non-homogeneous case is straightforward. \\

First, we introduce the functional setting. We will employ the following notation. For an open, bounded domain $D \subset \R^d$, $d=2,3$, we will denote by $H^s(D)$ the standard Sobolev space of order $s$, for a real number $s \geq 0$. When $s=0$, we will write $L^2(D)$. The usual norm on $H^s(D)$ will be denoted by $|| \cdot||_{s,D}$ and the usual seminorm by $| \cdot|_{s,D}$.
  Given a decomposition of the domain into elements $\mathcal{T}_h$, we will denote by $H^s(\mathcal{T}_h)$ the standard \emph{broken} Sobolev space, equipped with the broken norm $|| \cdot||_{s, \mathcal{T}_h}$. Furthermore, we will denote by $\mathbb{P}_k(D)$ the space of polynomials of \emph{total} degree less than or equal to $k \geq 1$ on $D$.
We will employ the symbols $\lesssim $ and $\gtrsim$ meaning that the inequalities hold up to multiplicative constants that are independent of the discretization parameters, but might depend on the physical parameters.\\

Next, we introduce the functional spaces for our weak formulation. For the bulk and fracture pressure we define the spaces
\begin{align*}
Q^b &= \{ q \in H^1(\Omega \setminus \bar{\Gamma}): \; q=0 \; \mbox{on} \; \partial \Omega_D \}, \\
Q^\Gamma &= \{ q_\Gamma=(q_\Gamma^1, \dots, q_\Gamma^{N^\Gamma}) \in \prod_{k=1}^{N_\Gamma} H^1(\gamma_k): \; q_\Gamma^k = 0  \;\; \mbox{on} \;\; \partial \gamma_k^D \; \forall k=1, \dots, N_\Gamma \\  &  \quad \quad  \mbox{and} \;\; q_\Gamma^1=\dots=q_\Gamma^{N_\Gamma} \;\; \mbox{on} \; \mathcal{I}_\cap \},
\end{align*}
where the trace operators are understood. We remark that the functions in the fracture space $Q^\Gamma$ have continuous trace at the intersection. We equip the space $Q^b \times Q^\Gamma$ with the norm 
\begin{equation}
||(q,q_\Gamma)||^2 = || \boldsymbol{\nu}^{1/2} \nabla q||^2_{0,\Omega} + || (\boldsymbol{\nu}_{\Gamma}^{\tau} \ell_{\Gamma})^{1/2} \nabla_{\tau} q_{\Gamma}||_{0, \Gamma}^2 +  || \beta_{\Gamma}^{1/2} \llbracket q \rrbracket||^2_{0,\Gamma} +  || \alpha_{\Gamma}^{1/2} (\{ q\}- q_{\Gamma})||^2_{0,\Gamma},
\end{equation}
assuming from now on that $\alpha_\Gamma >0$, that is $\xi > \frac{1}{2}$, see \eqref{def alpha beta eta}. Moreover, we introduce the bilinear form $\mathcal{A}: (Q^b \times Q^\Gamma) \, \times \, (Q^b \times Q^\Gamma) \rightarrow \R$ and the linear functional $\mathcal{L}: Q^b \times Q^{\Gamma} \rightarrow \R$, defined as 
   \begin{align*}
 \mathcal{A}\left((p,p_{\Gamma}), (q, q_{\Gamma}) \right) &=  \int_{\Omega} \boldsymbol{\nu} \nabla p \cdot \nabla q   \quad + \int_{\Gamma} \boldsymbol{\nu}_{\Gamma}^{\tau} \ell_{\Gamma} \nabla_{\tau} p_{\Gamma} \cdot \nabla_{\tau} q_{\Gamma} \\
 & \quad \quad + \int_{\Gamma} \beta_{\Gamma}  \llbracket p \rrbracket \cdot  \llbracket q \rrbracket +   \int_{\Gamma} \alpha_{\Gamma} (\{ p\}- p_{\Gamma})(\{q\}- q_{\Gamma}) \\
 \mathcal{L}(q,q_{\Gamma}) &=\int_{\Omega} fq +  \int_{\Gamma}\ell_\Gamma f_{\Gamma} q_{\Gamma}.
   \end{align*}
With the above notation, the weak formulation of the model problem \eqref{pb bulk}-\eqref{pb frattura}-\eqref{CC}-\eqref{eq intersezione} reads as follows: Find $(p, p_{\Gamma}) \in Q^b \times Q^{\Gamma}$ such that, for all $(q,q_{\Gamma}) \in Q^b \times Q^{\Gamma} $ 
 \begin{equation} \label{pb continuo weak}
 \mathcal{A}\left((p,p_{\Gamma}), (q, q_{\Gamma}) \right) = \mathcal{L}(q,q_{\Gamma}).
 \end{equation}
We remark that the equivalence, in the distributional sense, of problem \eqref{pb continuo weak} to the model problem \eqref{pb bulk}-\eqref{pb frattura}-\eqref{CC}-\eqref{eq intersezione} can be proved using standard distributional arguments. In particular, if we focus on the problem on the fracture network, 

We can now prove the following well-posedness result.
\begin{teo} \label{well-posed weak}
Let $\xi>1/2$. Then, problem \eqref{pb continuo weak} is well-posed. 
\end{teo}

\begin{proof}
The statement is a direct consequence of Lax-Milgram Theorem and of the regularity of the forcing terms.
\end{proof}

We remark that the choice of considering a primal-primal setting for both the bulk and fracture problems is made here only for the sake of simplicity. 
We refer to \cite{formaggiascottisottocasa} for the analysis of the mixed-mixed formulation in the case of a \emph{totally immersed} network of fractures.

Next, we focus on the numerical discretization of the problem based on polyDG methods.

\section{PolyDG discretization}\label{ntw:sec:DG discr}
In this section we present a numerical discretization for the coupled bulk-network problem that is based on DG methods on polytopic grids. In particular, we discretize both the bulk and fracture network problems in primal form, employing the Simmetric Interior Penalty DG method \cite{Arnold82,Wheeler78}. The key idea to obtain a DG discretization will be the generalization of the concepts of jump and average at the intersection point/line, so that we will be able to impose the conditions at the intersection \eqref{eq intersezione} in the spirit of DG methods. In particular, pressure continuity will be enforced penalizing the jump at the intersection, while balance of fluxes will be imposed \lq\lq naturally\rq\rq. \\

We start with the introduction of the notation related to the polytopic discretization of the domains.
For the problem in the bulk, consider a family of meshes $\mathcal{T}_h$ made of disjoint open \emph{polygonal/polyhedral} elements which are aligned with the fracture network $\Gamma$ and also with the decomposition of $\Omega$ into subdomains $\omega_\alpha$, $\alpha =1, \dots, N_\omega$. In particular, any element $E \in \mathcal{T}_h$ cannot be cut by $\Gamma$, and, since the subdomains $\omega_\alpha$ are disjoint, each element $E$ belongs exactly to one these subdomains.

In order to admit hanging nodes, following \cite{poligoni1,poligoni2,antonietti2015review,libropoligoni}, we introduce the concept of mesh \emph{interfaces}, which are defined to be the intersection of the $(d-1)$-dimensional facets of neighbouring elements. When $d=3$, interfaces consists in general polygons and we assume that it is possible to subdivide each interface into a set of co-planar triangles. We denote the set of all these triangles, which we call \emph{faces}. When $d=2$, the interfaces of an element simply consists of line segments, so that the concepts of faces and interfaces coincide. We still denote by $\F$ the set of all faces. Following \cite{poligoni1,poligoni2,antonietti2015review,libropoligoni}, no limitation is imposed on either the number of faces of each polygon $E \in \mathcal{T}_h$ or on the relative size of element faces compared to its diameter. \\
Each mesh $\mathcal{T}_h$ induces a subdivision of each fracture in the network $\gamma_k$ into faces, which we will denote by $\gamma_{k,h}$, for $k=1, \dots, N_\Gamma$. The collection of all the fracture faces is denoted by $\Gamma_h$, i.e., $\Gamma_h = \cup_{k=1}^{N_\Gamma} \gamma_{k,h}$. This implies that the set of all the faces $\mathcal{F}_h$ may be decomposed into three subsets, namely, 
\[ \mathcal{F}_h = \mathcal{F}_h^I \cup \mathcal{F}_h^B \cup \Gamma_h,\]
where  $\mathcal{F}_h^B$ is the set of boundary faces, $\Gamma_h$ is the set of fracture faces defined above, and $\mathcal{F}_h^I$ is the set of interior faces not belonging to the fracture. In addition, we have $\mathcal{F}_h^B = \mathcal{F}_h^D \cup \mathcal{F}_h^N$, where $\mathcal{F}_h^D$ and $\mathcal{F}_h^N$ are the boundary faces contained in $\partial \Omega_D$ and $\partial \Omega_N$, respectively (we assume the decomposition to be matching with the partition of $\partial \Omega$ into $\partial \Omega_D$ and $\partial \Omega_N$). 

The induced discretization of the fractures $\Gamma_h$ contains the faces of the elements of $\T$ that share part of their boundary with one ore more fractures, so that $\Gamma_h$ is made up of line segments when $d=2$ and of triangles when $d=3$. We observe that, when $d=3$, the triangles are not necessarily shape-regular and they may present hanging nodes, due to the fact that the sub-triangulations of each elemental interface is chosen independently from the others. For this reason, we need to extend the concept of \emph{interface} also to the $(d-2)$-dimensional facets of elements in $\Gamma_h$, defined again as intersection of boundaries of two neighbouring elements. When $d=2$, the interfaces reduce to points, while when $d=3$ they consists of line segments. We denote by $\mathcal{E}_{\Gamma,h}$ the set of all the interfaces (edges) of the elements in $\Gamma_h$, and we write, accordingly to the previous notation, 
\[\mathcal{E}_{\Gamma,h} =  \mathcal{E}_{\Gamma,h}^I \cup \mathcal{E}_{\Gamma,h}^B \cup \mathcal{E}_{\Gamma,h}^F  \cup \mathcal{E}_{\Gamma,h}^\cap,\]
where:
\begin{itemize}
\item $\mathcal{E}_{\Gamma,h}^I$ is the set of interior edges;
\item $\mathcal{E}_{\Gamma,h}^B=\mathcal{E}_{\Gamma,h}^D \cup \mathcal{E}_{\Gamma,h}^N$ is the set of edges belonging to the boundaries of the fracture network $\partial \Gamma_D$ and $\partial \Gamma_N$, respectively;
\item $ \mathcal{E}_{\Gamma,h}^F$ is the set of edges belonging to the immersed tips of the network;
\item  $\mathcal{E}_{\Gamma,h}^\cap$ is the set of edges on the intersection of the fractures. Note that, since we are considering a network with one single intersection, when $d=2$ this set consists only of one single point.
\end{itemize}
We will also write $\mathcal{E}_{\gamma_k,h}^*$, with $* \in \{I,B,F,\cap \}$, to denote the restriction of each of these sets to the fracture $\gamma_k$.

For each element $E \in \mathcal{T}_h$, we denote by $|E|$ its measure, by $h_E$ its diameter and we set $h= \max_{E \in \mathcal{T}_h} h_E$. Moreover, given an element $E \in \mathcal{T}_h$, for any face/edge $F \subset \partial E$ we define $\textbf{n}_F$ as the unit normal vector on $F$ that points outward of $E$.
 
Next, we recall that, for scalar and vector-valued functions $q$ and $\textbf{v}$ that are piecewise smooth on $\mathcal{T}_h$, it holds on every $F \in \F \setminus \F^B$:
 \begin{equation} \label{prima magic}
\llbracket q \textbf{v}  \rrbracket = \llbracket \textbf{v}  \rrbracket \{q \} + \{ \textbf{v}\} \cdot \llbracket q  \rrbracket ,
 \end{equation}
where jump and average operators are defined similarly to \eqref{salto media scalar} and \eqref{salto media vector}. 
If we define on $\F^B$
\begin{equation} \label{media salto bordo}
 \llbracket q \rrbracket = q \textbf{n}_F, \quad \quad 
\{\textbf{v}\}=  \textbf{v},
\end{equation}
identity \eqref{prima magic} implies the following well-known formula \cite{Arnold82}:
\begin{equation}\label{magic formula classica}
\sum_{E \in \mathcal{T}_h} \int_{\partial E} q \textbf{v}\cdot \textbf{n}_E = \int_{\mathcal{F}_h} \{\textbf{v}\} \cdot \llbracket q  \rrbracket + \int_{\mathcal{F}_h \setminus \mathcal{F}_h^B}\llbracket \textbf{v}  \rrbracket \{ q\}, 
\end{equation} 
where we have used the compact notation $\int_{\mathcal{F}_h} = \sum_{F \in \mathcal{F}_h} \int_F$.

Analogous definitions may be also stated for the fractures. In particular, given an element $F \in \Gamma_h$, with measure $|F|$ and diameter $h_F$, for any edge $e \subset \partial F$, with $e \in \E$, we define $\textbf{n}_e$ as the unit normal vector on $e$ pointing outward of F (it reduces to $\pm 1$ when $d=2$). Finally, standard jump and average operators across edges $e \in \E^I \cup \E^B$ can be defined for (regular enough) scalar and vector-valued functions and an analogous version of formula \eqref{magic formula classica} can be stated, which we will generalize to intersection edges in Proposition \ref{prop:magic formula generalizzata} below.

\subsection{Discrete formulation}
For simplicity in the forthcoming analysis, we will suppose that the permeability tensors $\boldsymbol{\nu}$ and $\boldsymbol{\nu}_{\Gamma}$ are piecewise  \emph{constant} on mesh elements, i.e., $\boldsymbol{\nu}|_E \in [\mathbb{P}_0(E)]^{d \times d}$ for all $E \in \mathcal{T}_h$, and $\boldsymbol{\nu}_{\Gamma}|_F \in [\mathbb{P}_0(F)]^{(d-1) \times (d-1)}$ for all $F \in \Gamma_h$. \\

First, we introduce the finite-dimensional spaces where we will set our discrete problem. For the problem in the bulk we define the broken polynomial space
\begin{equation}
Q_h^b =\{ q \in L^2(\Omega): \; q|_E \in \mathbb{P}_{k_E} (E)\; \forall E\in \mathcal{T}_h\}, \quad \quad k_E \geq 1, \, \forall E \in \mathcal{T}_h.
\end{equation}
Similarly, on each fracture $\gamma_k$, for $k=1, \dots N_\Gamma$, we define the space
\begin{equation}
Q_h^{\gamma_k}=\{q_{\Gamma}^k \in  L^2(\gamma_k):  \; q_{\Gamma}^k|_F \in \mathbb{P}_{k_F} (F) \; \forall F \in \gamma_{h,k}\} \quad \quad k_F \geq 1, \, \forall F \in \gamma_{h,k},
\end{equation}
so that on the fracture network we can introduce the product space
\begin{equation}
Q_h^\Gamma = \prod_{k=1}^{N_\Gamma} Q_h^{\gamma_k}.
\end{equation}
For future use in the analysis, we also introduce the DG vector-valued spaces
\begin{align*}
\textbf{W}_h^b &= \{ \textbf{v} \in [L^2(\Omega)]^d: \; \textbf{v}|_E \in [\mathbb{P}_{k_E} (E)]^d\; \forall E\in \mathcal{T}_h\},  \; &k_E \geq 1, \, \forall E \in \mathcal{T}_h,\\
\textbf{W}_h^{\gamma_k} &= \{ \textbf{v}_{\Gamma}^k \in [L^2(\Gamma)]^{d-1}: \; \textbf{v}_{\Gamma}^k|_F \in [\mathbb{P}_{k_F} (F)]^{d-1}\, \forall F\in \gamma_{h,k}\},  &k_F \geq 1, \, \forall F \in \gamma_{h,k}, \\
\textbf{W}_h^\Gamma &= \prod_{k=1}^{N_\Gamma} \textbf{W}_h^{\gamma_k}.
\end{align*}

In order to derive a DG discrete formulation of problem \eqref{pb continuo weak}, we make the following regularity assumption.
\begin{ass}\label{salti grad nulli}
We assume that the exact solution $(p,p_\Gamma)$ of problem \eqref{pb continuo weak} is such that:
\begin{itemize}
\item[A1. ] $p \in Q^b \cap H^2(\T)$ and $p_\Gamma \in Q^\Gamma \cap H^2(\Gamma_h)$;
\item[A2. ] the normal components of the exact fluxes $\bnu \nabla p$  and $\ell_\Gamma \bnu_\Gamma^\tau \nabla p_\Gamma$ are continuous across mesh interfaces, that is $\llbracket \bnu \nabla p  \rrbracket   =0$ on  $ \F^I$ and $\llbracket \ell_\Gamma \bnu_\Gamma^\tau \nabla p_\Gamma  \rrbracket   =0$ on $\E^I$.
\end{itemize}
\end{ass}
Moreover, for the forthcoming analysis, we introduce the following extended continuous spaces 
\begin{align}
Q^b(h)&= Q_h^b \oplus \big( Q^b \cap H^2(\T) \big) \label{spazio esteso bulk}\\
Q^\Gamma(h)&= Q_h^\Gamma \oplus \big( Q^\Gamma \cap H^2(\Gamma_h) \big). \label{spazio esteso fracture}
\end{align}

In order to derive a DG formulation for the problem in the bulk, we proceed as in \cite{mioUnafrattura,mioUnified}. We obtain the following: Find $p_h \in Q_h^b$ such that for every test function $q \in Q_h^b$ it holds
\begin{multline}\label{formulazione dopo flussi bulk}
\int_{\mathcal{T}_h} \boldsymbol{\nu} \nabla p_h \cdot \nabla q -\int_{\mathcal{F}_h^I \cup \mathcal{F}_h^D} \{ \bnu \nabla p_h\} \cdot \llbracket q  \rrbracket  -\int_{\mathcal{F}_h^I \cup \mathcal{F}_h^D} \{ \bnu \nabla q\} \cdot \llbracket p_h  \rrbracket  +\int_{\mathcal{F}_h^I \cup \mathcal{F}_h^D} \sigma_F \llbracket p_h  \rrbracket \cdot \llbracket q  \rrbracket \\
+\int_{\Gamma_h} \beta_{\Gamma} \llbracket p_h  \rrbracket \cdot \llbracket q  \rrbracket   + \int_{\Gamma_h} \alpha_{\Gamma}(\{p_h\}-p_{\Gamma})\{q\} = \int_{\mathcal{T}_h} fq - \int_{\mathcal{F}_h^D} (\bnu \nabla q \cdot\textbf{n}_F -\sigma_F q)g_D,
\end{multline}
where we have introduced the discontinuity penalization parameter $\sigma$, which is a non-negative bounded function, i.e., $ \sigma \in L^{\infty}(\mathcal{F}_h^I \cup \mathcal{F}_h^D)$. Its precise definition will be given in Definition \ref{defi sigma bulk} below.\\

Next, we derive a DG discrete formulation for the problem on the fracture network. 
For generality, we will write our formulation referring to the case $d=3$. However, the expressions are valid also when $d=2$, provided that, when the domain of integration reduces to a point, the integrals are interpreted as evaluations.First, we focus on a single fracture $\gamma_k$. Given a face $F \in \gamma_{k,h}$, we multiply the first equation in \eqref{pb frattura} for a test function $q_\Gamma^k \in Q_h^{\gamma_k}$ and integrate over $F$. Summing over all 
$F \in \gamma_{h,k}$ and integrating by parts, we obtain
\begin{multline*}
\int_{\gamma_{h,k}} \boldsymbol{\nu}_{\gamma_k}^{\tau}  \ell_k \nabla_{\tau}p_{\Gamma}^k \cdot \nabla_{\tau}q_{\Gamma}^k - \sum_{F \in \gamma_{h,k}} \int_{\partial F} q_\Gamma^k \boldsymbol{\nu}_{\gamma_k}^{\tau}  \ell_k \nabla_{\tau}p_{\Gamma}^k \cdot \textbf{n}_F  \\
= \int_{\gamma_{h,k}} \ell_k f_\Gamma^k q_\Gamma^k - \int_{\gamma_{h,k}}  \alpha_{\Gamma}(\{p\}-p_{\Gamma}^k)q_\Gamma^k,
\end{multline*}
where we have used the second coupling condition in \eqref{CC} to rewrite $- \llbracket \bnu \nabla p \rrbracket= \alpha_{\Gamma}(\{p\}-p_{\Gamma}^k)$ in the source term.
If we sum over all the fractures $\gamma_k$ in the network and use identity \eqref{magic formula classica} on each fracture $\gamma_k$, we get
\begin{multline*}
\int_{\Gamma_h} \boldsymbol{\nu}_{\Gamma}^{\tau}  \ell_{\Gamma} \nabla_{\tau}p_{\Gamma} \cdot \nabla_{\tau}q_{\Gamma} -  \int_{\mathcal{E}_{\Gamma,h}^I} \llbracket  \boldsymbol{\nu}_{\Gamma}^{\tau}  \ell_{\Gamma} \nabla_{\tau}p_{\Gamma}  \rrbracket \{ q_\Gamma \}  -  \int_{\mathcal{E}_{\Gamma,h}^I \cup \mathcal{E}_{\Gamma,h}^B} \{ \boldsymbol{\nu}_{\Gamma}^{\tau}  \ell_{\Gamma} \nabla_{\tau}p_{\Gamma} \} \cdot \llbracket q_\Gamma  \rrbracket \\ 
- \sum_{k=1}^{N_\Gamma} \Big[ \int_{\mathcal{E}_{\gamma_k,h}^F}  q_\Gamma^k \boldsymbol{\nu}_{\gamma_k}^{\tau}  \ell_k \nabla_{\tau}p_{\Gamma}^k \cdot \boldsymbol{\tau}_k + \int_{\mathcal{E}_{\gamma_k,h}^\cap} q_\Gamma^k \boldsymbol{\nu}_{\gamma_k}^{\tau}  \ell_k \nabla_{\tau}p_{\Gamma}^k \cdot \boldsymbol{\tau}_k \Big]
\\=  \int_{\Gamma_h} \ell_\Gamma f_\Gamma q_\Gamma - \int_{\Gamma_h}  \alpha_{\Gamma}(\{p\}-p_{\Gamma})q_\Gamma,
\end{multline*}
where we recall that $\boldsymbol{\tau}_k$ is the vector tangent to the fracture $\gamma_k$, pointing outward of $\partial \gamma_k$, and $\llbracket \cdot  \rrbracket $ and $\{\cdot\}$ are the standard jump and average operators defined in \eqref{salto media scalar}, \eqref{salto media vector} and \eqref{media salto bordo}.
In order to treat the term defined on the intersection 
\begin{equation}\label{integrale intersezioni}
\sum_{k=1}^{N_\Gamma} \int_{\mathcal{E}_{\gamma_k,h}^\cap} q_\Gamma^k \boldsymbol{\nu}_{\gamma_k}^{\tau}  \ell_k \nabla_{\tau}p_{\Gamma}^k \cdot \boldsymbol{\tau}_k,
\end{equation}
we will now extend the definition of jump and average operators to the case when a number of planes intersect along one line ($d=3$) or when a number of segments intersect in one point ($d=2$).

\subsubsection{Jump and average operators at the intersections}\label{ntw:sec: salto media intersezione}
Let $\underline{b}=(b_1,b_2, \dots,b_{N_\Gamma})$ and $\underline{\textbf{a}}= (\textbf{a}_1, \textbf{a}_2, \dots , \textbf{a}_{N_\Gamma})$ be a scalar and vector-valued functions defined on the network $\Gamma$ (product space), such that for every $k=1, \dots, N_\Gamma$ the traces of $b_k$ and $\textbf{a}_k$ are well defined on the intersection $\mathcal{I}_\cap$.  Moreover, for $k=1, \dots, N_\Gamma$, let $\boldsymbol{\tau}_k$ be the vector tangent to the fracture $\gamma_k$, pointing outward of the intersection point/line $\mathcal{I}_\cap$. 
\begin{defi} \label{salto media intersezione}
We define \emph{jump} and \emph{average} operators for $\underline{\textbf{a}}$ and $\underline{b}$ at $\mathcal{I}_\cap$ as
\begin{align}
\{ \underline{b}\}_\cap &= \frac{1}{N_\Gamma}(b_1+b_2+\dots +b_{N_\Gamma}) \\
\llbracket \underline{b}  \rrbracket_\cap &=  \big(b_i - b_k\big)_{i,k \in \{1,2,\dots,N_\Gamma\}, \,  i<k} \\
\{ \underline{\textbf{a}} \}_\cap &= \frac{1}{N_\Gamma} \big( \textbf{a}_i \cdot \btau_i - \textbf{a}_k \cdot \btau_k \big)_{i,k \in \{1,2,\dots,N_\Gamma\}, \,  i<k}\\
\llbracket \underline{\textbf{a}} \rrbracket_\cap &= \textbf{a}_1 \cdot \btau_1+\textbf{a}_2 \cdot \btau_2 +\dots + \textbf{a}_{N_\Gamma} \cdot \btau_{N_\Gamma}, 
\end{align}
where trace operators on $\mathcal{I}_\cap$ are understood.
\end{defi}

We remark that $\{\underline{b}\}_\cap$ and $\llbracket \underline{\textbf{a}} \rrbracket_\cap$ are scalar-valued, while $\llbracket \underline{b}  \rrbracket_\cap$ and $\{ \underline{\textbf{a}}\}_\cap$ are vector-valued, taking values in $ \in \R^{\binom{N_\Gamma}{2}}$. In particular, for the definition of
 $\llbracket \underline{b}  \rrbracket_\cap$ and $\{ \underline{\textbf{a}}\}_\cap$ we take all the pairs of indices in $\{1,\dots,N_\Gamma\}$ such that the first index is smaller than the second one. This is just one possible way of indicating all the pairs of fractures. Accordingly, these vectors contain $\binom{N_\Gamma}{2}=\frac{N_\Gamma(N_\Gamma-1)}{2}$ elements. For example, for $N_\Gamma=4$, we have
\[\llbracket \underline{b}  \rrbracket_\cap = (b_1-b_2, b_1-b_2, b_1-b_4, b_2-b_3, b_2-b_4, b_3-b_4) \in \R^6, \]
while, for $N_\Gamma=5$, we have
\[ \llbracket \underline{b}  \rrbracket_\cap = (b_1-b_2, b_1-b_2, b_1-b_4, b_1-b_5, b_2-b_3, b_2-b_4, b_2-b_5, b_3-b_4, b_3-b_5, b_4-b_5) \in \R^{10}. \]
The vector-valued case is analogous. 
Note also that, when $N_\Gamma=2$, these definitions coincide with the definitions of jump and average operators introduced in \cite{DGcurvilineari,DGcurvihighorder}, for the generalization of DG methods to curved surfaces. Indeed we have
\begin{alignat}{5}
\{\underline{b}\}_\cap &= \frac{1}{2}(b_1+b_2), &\quad \quad  \llbracket \underline{b}\rrbracket_\cap &= b_1-b_2,\\
\{\underline{a}\}_\cap &= \frac{1}{2}(\textbf{a}_1 \cdot \btau_1 -\textbf{a}_2 \cdot \btau_2), &\quad \quad  \llbracket \underline{a} \rrbracket_\cap &= \textbf{a}_1 \cdot \btau_1 + \textbf{a}_2 \cdot \btau_2.
\end{alignat}

Definition \ref{salto media intersezione} allows us to find an equivalent version of identity \eqref{prima magic} on the intersection:
\begin{prop} \label{prop:magic formula generalizzata}
The following identity holds
\begin{equation} \label{magic formula generalizzata}
\llbracket \underline{b} \underline{\textbf{a}} \rrbracket_\cap =  \llbracket \underline{\textbf{a}}  \rrbracket_\cap \{\underline{b}\}_\cap + \{ \underline{\textbf{a}}\}_\cap \cdot \llbracket \underline{b}  \rrbracket_\cap,
\end{equation}
where the vector-valued function $\underline{b \, \textbf{a}}$ is defined as $\underline{b \, \textbf{a}}= (b_1 \, \textbf{a}_1, b_2 \, \textbf{a}_2, \dots , b_{N_\Gamma} \textbf{a}_{N_\Gamma})$ and $\cdot$ is the standard scalar-product in $\R^{\binom{N_\Gamma}{2}}$.
\end{prop}
\begin{proof}
By definition we have 
\begin{equation}
\llbracket \underline{b} \underline{\textbf{a}}\rrbracket_\cap = \sum_{k=1}^{N_\Gamma} b_k \textbf{a}_k \cdot \btau_k.
\end{equation}
Moreover, we can write
\begin{equation}
\llbracket \underline{\textbf{a}} \rrbracket_\cap \{\underline{b}\}_\cap = \frac{1}{N_\Gamma} \Big( \sum_{k=1}^{N_\Gamma} b_k\Big)  \Big(\sum_{j=1}^{N_\Gamma} \textbf{a}_j \cdot \btau_j \Big) =  \frac{1}{N_\Gamma}  \sum_{k=1}^{N_\Gamma}( b_k  \textbf{a}_k \cdot \btau_k )+ \frac{1}{N_\Gamma} \sum_{k=1}^{N_\Gamma} (b_k \sum_{\substack{j=1 \\ j \neq k}}^{N_\Gamma} \textbf{a}_j \cdot \btau_j ),
\end{equation}
while we have
\begin{align*}
 \{ \underline{\textbf{a}}\}_\cap \cdot \llbracket \underline{b}  \rrbracket_\cap &= 
 \frac{1}{N_\Gamma}\sum_{k=1}^{N_\Gamma} \sum_{j=k+1}^{N_\Gamma} (b_k-b_j)(\textbf{a}_k \cdot \btau_k - \textbf{a}_j \cdot \btau_j) \\ &= \frac{1}{N_\Gamma} \sum_{k=1}^{N_\Gamma} \sum_{\substack{j=1 \\ j\neq k}}^{N_\Gamma} (b_k-b_j) \textbf{a}_k \cdot \btau_k \\
 &= \frac{1}{N_\Gamma} \sum_{k=1}^{N_\Gamma} \sum_{\substack{j=1 \\ j \neq k}}^{N_\Gamma} b_k \textbf{a}_k \cdot \btau_k -\frac{1}{N_\Gamma}  \sum_{k=1}^{N_\Gamma} \Big( \textbf{a}_k \cdot \btau_k  \sum_{\substack{j=1 \\ j \neq k}}^{N_\Gamma} b_j \Big)\\ 
 &= \frac{1}{N_\Gamma} \sum_{k=1}^{N_\Gamma} (N_\Gamma-1) b_k \textbf{a}_k \cdot \btau_k -\frac{1}{N_\Gamma}  \sum_{k=1}^{N_\Gamma} b_k \Big( \sum_{\substack{j=1 \\ j \neq k}}^{N_\Gamma} \textbf{a}_j \cdot \btau_j \Big). 
\end{align*}
This implies 
\begin{align*}
\llbracket \underline{\textbf{a}}  \rrbracket_\cap \{\underline{b}\}_\cap + \{ \underline{\textbf{a}}\}_\cap \cdot \llbracket \underline{b}  \rrbracket_\cap &= \frac{1}{N_\Gamma} \sum_{k=1}^{N_\Gamma} b_k \textbf{a}_k \cdot \btau_k + \frac{1}{N_\Gamma} \sum_{k=1}^{N_\Gamma} (N_\Gamma-1)b_k \textbf{a}_k \cdot \btau_k \\ & = \frac{1}{N_\Gamma} \sum_{k=1}^{N_\Gamma} N_\Gamma  b_k \textbf{a}_k \cdot \btau_k,
\end{align*}
and the proof is concluded.
\end{proof}

Now we take our focus back to the derivation of a DG discrete formulation for the problem in the fracture network. Using the above definition of jump and average at the intersection \ref{salto media intersezione} and identity \eqref{magic formula generalizzata}, we can rewrite \eqref{integrale intersezioni} as 
\begin{align}
\sum_{k=1}^{N_\Gamma} \int_{\mathcal{E}_{\gamma_k,h}^\cap} q_\Gamma^k \boldsymbol{\nu}_{\gamma_k}^{\tau}  \ell_k \nabla_{\tau}p_{\Gamma}^k \cdot \boldsymbol{\tau}_k &= \int_{\mathcal{E}_{\Gamma,h}^\cap} \llbracket q_\Gamma  \boldsymbol{\nu}_{\Gamma}^{\tau}  \ell_{\Gamma} \nabla_{\tau}p_{\Gamma}\rrbracket_\cap \\
 &= \int_{\mathcal{E}_{\Gamma,h}^\cap} \llbracket q_\Gamma  \rrbracket_\cap \cdot \{\boldsymbol{\nu}_{\Gamma}^{\tau}  \ell_{\Gamma} \nabla_{\tau}p_{\Gamma}\}_\cap +  \int_{\mathcal{E}_{\Gamma,h}^\cap} \llbracket  \boldsymbol{\nu}_{\Gamma}^{\tau}  \ell_{\Gamma} \nabla_{\tau}p_{\Gamma} \rrbracket_\cap\{q_\Gamma\}_\cap. 
\end{align}

\noindent The formulation on the fracture network becomes
\begin{multline} \label{formulaz fratt prima termini nulli}
\int_{\Gamma_h} \boldsymbol{\nu}_{\Gamma}^{\tau}  \ell_{\Gamma} \nabla_{\tau}p_{\Gamma} \cdot \nabla_{\tau}q_{\Gamma} 
- \int_{\mathcal{E}_{\Gamma,h}^I} \llbracket  \boldsymbol{\nu}_{\Gamma}^{\tau}  \ell_{\Gamma} \nabla_{\tau}p_{\Gamma} \rrbracket \{ q_\Gamma \} -  \int_{\mathcal{E}_{\Gamma,h}^I \cup \mathcal{E}_{\Gamma,h}^B} \{ \boldsymbol{\nu}_{\Gamma}^{\tau}  \ell_{\Gamma} \nabla_{\tau}p_{\Gamma}  \} \cdot \llbracket q_\Gamma  \rrbracket \\
   - \int_{\mathcal{E}_{\Gamma,h}^F}  q_\Gamma \boldsymbol{\nu}_{\Gamma}^{\tau}  \ell_{\Gamma} \nabla_{\tau}p_{\Gamma} \cdot \boldsymbol{\tau} -\int_{\mathcal{E}_{\Gamma,h}^\cap} \llbracket q_\Gamma  \rrbracket_\cap \cdot \{\boldsymbol{\nu}_{\Gamma}^{\tau}  \ell_{\Gamma} \nabla_{\tau}p_{\Gamma} \}_\cap - \int_{\mathcal{E}_{\Gamma,h}^\cap} \llbracket  \boldsymbol{\nu}_{\Gamma}^{\tau}  \ell_{\Gamma} \nabla_{\tau}p_{\Gamma} \rrbracket_\cap\{q_\Gamma\}_\cap \\
=  \int_{\Gamma_h} \ell_\Gamma f_\Gamma q_\Gamma - \int_{\Gamma_h}  \alpha_{\Gamma}(\{p\}-p_{\Gamma})q_\Gamma.
\end{multline}
From the fact that $p \in Q^\Gamma$ satisfies problem \eqref{pb continuo weak} and from the regularity Assumption \ref{salti grad nulli}, it holds:
\begin{itemize}
\item   $\llbracket \boldsymbol{\nu}_{\Gamma}^{\tau}  \ell_{\Gamma} \nabla_{\tau}p_{\Gamma} \rrbracket=0$ on $\mathcal{E}_{\Gamma,h}^I$;
\item $\llbracket p_\Gamma  \rrbracket=0$ on $\mathcal{E}_{\Gamma,h}^I$;
\item $\boldsymbol{\nu}_{\Gamma}^{\tau}  \ell_{\Gamma} \nabla_{\tau}p_{\Gamma} \cdot \boldsymbol{\tau} =0$ on 
$\mathcal{E}_{\Gamma,h}^F \cup \E^N$;
\item $\llbracket p_\Gamma  \rrbracket_\cap =0$ on $\mathcal{E}_{\Gamma,h}^\cap$;
\item $ \llbracket  \boldsymbol{\nu}_{\Gamma}^{\tau}  \ell_{\Gamma} \nabla_{\tau}p_{\Gamma} \rrbracket_\cap=0$ on $\mathcal{E}_{\Gamma,h}^\cap$.
\end{itemize}
It follows that, for any test function $q_\Gamma \in Q_h^\Gamma$, identity \eqref{formulaz fratt prima termini nulli} is equivalent to
\begin{multline} \label{formulaz frattura}
\int_{\Gamma_h} \boldsymbol{\nu}_{\Gamma}^{\tau}  \ell_{\Gamma} \nabla_{\tau}p_{\Gamma} \cdot \nabla_{\tau}q_{\Gamma} 
 -  \int_{\mathcal{E}_{\Gamma,h}^I \cup \mathcal{E}_{\Gamma,h}^D} \{ \boldsymbol{\nu}_{\Gamma}^{\tau}  \ell_{\Gamma} \nabla_{\tau}p_{\Gamma}  \} \cdot \llbracket q_\Gamma  \rrbracket -  \int_{\mathcal{E}_{\Gamma,h}^I \cup \mathcal{E}_{\Gamma,h}^D} \{ \boldsymbol{\nu}_{\Gamma}^{\tau}  \ell_{\Gamma} \nabla_{\tau}q_{\Gamma}  \} \cdot \llbracket p_\Gamma  \rrbracket   \\
  -\int_{\mathcal{E}_{\Gamma,h}^\cap} \{\boldsymbol{\nu}_{\Gamma}^{\tau}  \ell_{\Gamma} \nabla_{\tau}p_{\Gamma} \}_\cap \cdot  \llbracket q_\Gamma  \rrbracket_\cap   -\int_{\mathcal{E}_{\Gamma,h}^\cap} \{\boldsymbol{\nu}_{\Gamma}^{\tau}  \ell_{\Gamma} \nabla_{\tau}q_{\Gamma}\}_\cap \cdot \llbracket p_\Gamma  \rrbracket_\cap \\
  +  \int_{\mathcal{E}_{\Gamma,h}^I \cup \mathcal{E}_{\Gamma,h}^D} \sigma_e^\Gamma \llbracket p_\Gamma  \rrbracket \cdot \llbracket q_\Gamma  \rrbracket + \int_{\mathcal{E}_{\Gamma,h}^\cap} \sigma_e^\cap \llbracket p_\Gamma  \rrbracket_\cap \cdot \llbracket q_\Gamma  \rrbracket_\cap  \\
=  \int_{\Gamma_h} \ell_\Gamma f_\Gamma q_\Gamma + \int_{\Gamma_h}  \alpha_{\Gamma}(\{p\}-p_{\Gamma})q_\Gamma -  \int_{\mathcal{E}_{\Gamma,h}^D} ( \boldsymbol{\nu}_{\Gamma}^{\tau}  \ell_{\Gamma} \nabla_{\tau}q_{\Gamma} \cdot \boldsymbol{\tau} - \sigma_e^\Gamma q_\Gamma )g_\Gamma,
\end{multline}
where $\sigma^\Gamma \in L^{\infty}(\E^I \cup \E^D)$ and $\sigma^\cap \in L^{\infty}(\E^\cap)$ are discontinuity penalization parameters, whose precise definition will be given in \ref{defi sigma fracture} below.\\

In conclusion, we obtain the following discrete formulation for the coupled bulk-network problem:\\
Find $(p_h, p_{\Gamma,h}) \in Q_h^b \times Q_h^{\Gamma}$ such that
 \begin{equation} \label{discrete formulation}
 \mathcal{A}_h\left((p_h,p_{\Gamma,h}), (q, q_{\Gamma}) \right) = \mathcal{L}_h(q,q_{\Gamma}) \;\;\;\; \forall (q, q_{\Gamma}) \in  Q_h^b \times Q_h^{\Gamma},
 \end{equation}
where the bilinear form $\mathcal{A}_h: ( Q_h^b \times Q_h^{\Gamma}) \times ( Q_h^b \times Q_h^{\Gamma}) \rightarrow \R $ is defined as
\begin{equation} \label{bilinear form}
\mathcal{A}_h \left((p_h,p_{\Gamma,h}), (q, q_{\Gamma}) \right)= \mathcal{A}_b(p_h,q)+\mathcal{A}_{\Gamma}(p_{\Gamma,h}, q_{\Gamma})+ \mathcal{C}((p_h,p_{\Gamma,h}),(q, q_{\Gamma})),
 \end{equation}
 and the linear functional $\mathcal{L}_h: Q_h^b \times Q_h^{\Gamma} \rightarrow \R$ is defined as
 \begin{equation} \label{linear operator}
 \mathcal{L}_h(q,q_{\Gamma})=\mathcal{L}_b(q)+ \mathcal{L}_{\Gamma}(q_{\Gamma}),
 \end{equation}
with
\begin{align}
\mathcal{A}_b(p_h,q) & =  \int_{\T} \boldsymbol{\nu} \nabla p_h \cdot \nabla q 
 - \int_{\mathcal{F}_h^I \cup \mathcal{F}_h^D} \{ \boldsymbol{\nu} \nabla p_h \} \cdot \llbracket q \rrbracket  \\ &\quad -  \int_{\mathcal{F}_h^I \cup \mathcal{F}_h^D} \{ \boldsymbol{\nu} \nabla q \} \cdot \llbracket p_h \rrbracket +   \int_{\mathcal{F}_h^I \cup \mathcal{F}_h^D}  \sigma_F  \llbracket p_h \rrbracket \cdot  \llbracket q \rrbracket  \label{bilinear form bulk} \\ \\
\mathcal{A}_\Gamma(p_{\Gamma,h},q_\Gamma) & =\int_{\Gamma_h} \boldsymbol{\nu}_{\Gamma}^{\tau}  \ell_{\Gamma} \nabla_{\tau}p_{\Gamma,h} \cdot \nabla_{\tau}q_{\Gamma} \\
 & \; -  \int_{\mathcal{E}_{\Gamma,h}^I \cup \mathcal{E}_{\Gamma,h}^D} \{ \boldsymbol{\nu}_{\Gamma}^{\tau}  \ell_{\Gamma} \nabla_{\tau}p_{\Gamma,h}  \} \cdot \llbracket q_\Gamma  \rrbracket -  \int_{\mathcal{E}_{\Gamma,h}^I \cup \mathcal{E}_{\Gamma,h}^D} \{ \boldsymbol{\nu}_{\Gamma}^{\tau}  \ell_{\Gamma} \nabla_{\tau}q_{\Gamma}  \} \cdot \llbracket p_{\Gamma,h}  \rrbracket   \\ & \;  -\int_{\mathcal{E}_{\Gamma,h}^\cap}  \{\boldsymbol{\nu}_{\Gamma}^{\tau}  \ell_{\Gamma} \nabla_{\tau}p_{\Gamma,h} \}_\cap \cdot \llbracket q_\Gamma  \rrbracket_\cap   -\int_{\mathcal{E}_{\Gamma,h}^\cap} \{\boldsymbol{\nu}_{\Gamma}^{\tau}  \ell_{\Gamma} \nabla_{\tau}q_{\Gamma}\}_\cap  \cdot \llbracket p_{\Gamma,h}  \rrbracket_\cap \\
& \;   +  \int_{\mathcal{E}_{\Gamma,h}^I \cup \mathcal{E}_{\Gamma,h}^D} \sigma^\Gamma_e \llbracket p_{\Gamma,h}  \rrbracket \cdot \llbracket q_\Gamma  \rrbracket + \int_{\mathcal{E}_{\Gamma,h}^\cap} \sigma_e^\cap \llbracket p_\Gamma  \rrbracket_\cap \cdot \llbracket q_\Gamma  \rrbracket_\cap   \label{bilinear form fracture} \\ \\
\mathcal{C}((p_h, p_{\Gamma,h}),(q,q_{\Gamma}))&= \int_{\Gamma_h} \beta_{\Gamma} \llbracket p_h \rrbracket \cdot  \llbracket q \rrbracket + \int_{\Gamma_h} \alpha_{\Gamma} (\{ p_h\}- p_{\Gamma,h})(\{q\}- q_{\Gamma}), \label{interface bilinear form}
\end{align}
and
\begin{align}
\mathcal{L}_b(q)&= \int_{\mathcal{T}_h} fq - \int_{\mathcal{F}_h^D} (\boldsymbol{\nu} \nabla q \cdot \textbf{n}_F - \sigma_F q)g_D, \label{linear operator bulk}\\
\mathcal{L}_{\Gamma}(q_{\Gamma})&= \int_{\Gamma_h} \ell_\Gamma f_{\Gamma} q_{\Gamma} -  \int_{\mathcal{E}_{\Gamma,h}^D} ( \boldsymbol{\nu}_{\Gamma}^{\tau}  \ell_{\Gamma} \nabla_{\tau}q_{\Gamma} \cdot \boldsymbol{\tau} - \sigma_e^\Gamma q_\Gamma )g_\Gamma. \label{linear operator fracture}
\end{align}

In the following, in order to simplify the notation, we will drop the subscript $\btau$ for the tangent operators on the fracture network.

\section{Well-posedness of the discrete formulation}\label{ntw:sec:well posed}

In this section, we prove that formulation \eqref{discrete formulation} is well-posed.

We recall that, for simplicity in the analysis, we are assuming the permeability tensors $\bnu$ and $\ell_\Gamma \bnu_\Gamma^\tau$ to be piecewise constant. We will employ the following notation $\bar{\boldsymbol{\nu}}_E = |\sqrt{\boldsymbol{\nu}|_E}|^2_2$ and $\bar{\boldsymbol{\nu}}_F^\tau = |\sqrt{ \ell_\Gamma \boldsymbol{\nu}_\Gamma^\tau|_F}|^2_2$, where $|\cdot|_2$ denotes the $l_2$-norm. \\

In order to work in a polytopic framework, we need to introduce some technical tools as in \cite{poligoni1,poligoni2,antonietti2015review,cangiani2016hp, libropoligoni}.
The first tool consists in trace inverse estimates, so that the norm of a polynomial on a polytope's face/edge can be bounded by the norm on the element itself. To this aim, we need to make some regularity assumptions on the mesh.
\begin{defi}\label{mesh xxx}
A mesh $\mathcal{T}_h$ is said to be \emph{polytopic-regular} if, for any $E \in \mathcal{T}_h$, there exists a set of non-overlapping (not necessarily shape-regular) $d$-dimensional simplices $\{S_E^i \}_{i=1}^{n_E}$ contained in $E$, such that $\bar{F} = \partial \bar{E} \cap \bar{S_E^i}$, for any face $F \subseteq \partial E$, and 
\begin{equation} \label{bound mesh xxx}
h_E \lesssim \frac{d |S_E^i|}{|F|}, \quad \quad i=1, \dots, n_E,
\end{equation}
with the hidden constant independent of the discretization parameters, the number of faces of the element $n_E$, and the face measure.
\end{defi}
We remark that this definition does not give any restriction on the number of faces per element, nor on their measure.
\begin{ass} \label{meshes are xxx}
We assume that $\mathcal{T}_h$ and $\Gamma_h$ are polytopic-regular meshes.
\end{ass}

The above assumption allows us to state the following inverse-trace estimate for polytopic elements \cite{cangiani2016hp, libropoligoni}:
\begin{lem}\label{inverse estimate poligoni}
Let $E$ be a polygon/polyhedron belonging to a mesh satisfying Definition \ref{mesh xxx} and let $v \in \mathbb{P}_{k_E}(E)$. Then, we have
\begin{equation} \label{eq:inverse estimate poligoni}
||v||^2_{L^2(\partial E)} \lesssim \frac{(k_E+1)(k_E+d)}{h_E} ||v||^2_{L^2(E)},
\end{equation}
where the hidden constant depends on the dimension $d$, but it is independent of the discretization parameters, of the number of faces of the element and of the relative size of the face compared to the diameter $k_E$ of $E$.
\end{lem}

The second fundamental tool to deal with polytopic discretizations, is an appropriate definition of the discontinuity penalization parameter. In particular, this will be instrumental for handling elements with arbitrarily small faces. Following \cite{poligoni1,poligoni2,antonietti2015review,cangiani2016hp, libropoligoni}, we define the bulk and fracture penalty functions as:
\begin{defi}\label{defi sigma bulk}
The discontinuity-penalization parameter $\sigma: \mathcal{F}_h \cup \F^D \rightarrow \R^+$ for the bulk problem is defined facewise as
 \begin{equation} \label{sigma bulk}
  \sigma(\textbf{x}) = \sigma_0  \begin{cases} \max_{E \in \{E^+,E^-\}} \frac{\bar{\boldsymbol{\nu}}_E (k_E + 1)(k_E+d)}{h_E} & \mbox{if } \textbf{x} \subset F \in \mathcal{F}_h^I, \, \bar{F}= \partial \bar{E}^+ \cap \partial \bar{E}^-, \\[3ex]
   \frac{\bar{\boldsymbol{\nu}}_E (k_E +1)(k_E+d)}{h_E} & \mbox{if } \textbf{x} \subset F \in \mathcal{F}_h^D, \, \bar{F}= \partial \bar{E} \cap \partial \bar{\Omega},
\end{cases}
 \end{equation}
 with $\sigma_0>0$ independent of $k_E$, $|E|$ and $|F|$.
\end{defi}

\begin{defi}\label{defi sigma fracture}
The discontinuity-penalization parameter $\sigma_\Gamma: \E^I \cup \E^D \cup \E^\cap \rightarrow \R^+$ for the fracture problem is defined edgewise as
 \begin{equation} \label{sigma fracture}
  \sigma^\Gamma(\textbf{x}) = \sigma_{0}^\Gamma  \begin{cases} \max\limits_{F \in \{F^+,F^-\}} \frac{\bar{\boldsymbol{\nu}}_F^\tau (k_F + 1)(k_F+d-1)}{h_F} & \mbox{if } \textbf{x} \subset e \in \E^I, \, \bar{e}= \partial \bar{F}^+ \cap \partial \bar{F}^-, \\[3ex]
   \frac{\bar{\boldsymbol{\nu}}_F^\tau (k_F +1)(k_F+d-1)}{h_F}, & \mbox{if } \textbf{x} \subset e \in \E^D, \, \bar{e}= \partial \bar{F} \cap \partial \bar{\Gamma}, \\[3ex]
   \max\limits_{F \in \{F^1, \dots, F^{N_\Gamma}\}} \frac{\bar{\boldsymbol{\nu}}_F^\tau (k_F + 1)(k_F+d-1)}{h_F} & \mbox{if } \textbf{x} \subset e \in \E^\cap, \, \bar{e}= \partial \bar{F}^1 \cap \dots \cap \partial \bar{F}^{N_\Gamma},
\end{cases}
 \end{equation}
 with $\sigma_0^\Gamma >0$ independent of $k_F$, $|F|$ and $|e|$.
\end{defi} 
Note that the definition of the fracture parameter on the intersection edges will play a crucial role in proving the well-posedness of our method. In the following we will write $\sigma^\cap$ to denote $\sigma^\Gamma|_{\E^\cap}$.

\vspace{1cm}

Next, we endow the discrete space $Q_h^b \times Q_h^\Gamma$ with the \emph{energy} norm
\begin{equation}\label{norma energia}
|||(q,q_\Gamma)|||^2 = ||q||^2_{b,DG}  +||q_\Gamma||^2_{\Gamma,DG} +||(q,q_\Gamma)||_{\mathcal{C}}^2,
\end{equation}
where
\begin{align*}
||q||^2_{b,DG} &=  ||\boldsymbol{\nu}^{1/2} \nabla q||_{0,\T}^2 +  ||\sigma_F^{1/2} \llbracket q \rrbracket ||_{0,\F^I \cup \F^D}^2,\\
||q_\Gamma||^2_{\Gamma,DG} &= ||(\boldsymbol{\nu_{\Gamma}^\tau} \ell_\Gamma)^{1/2} \nabla q_\Gamma||^2_{0,\Gamma_h} + || \sigma_e^{1/2} \llbracket q_\Gamma  \rrbracket ||^2_{0,\E^I \cup \E^D \cup \E^\cap},\\
||(q,q_\Gamma)||_{\mathcal{C}}^2&=  ||\beta_\Gamma^{1/2} \llbracket q  \rrbracket||_{0,\Gamma_h}^2 +||\alpha_\Gamma^{1/2}(\{q\} - q_\Gamma)||^2_{0,\Gamma_h}.
\end{align*}
\begin{oss}
Since we are assuming that there is a single intersection in the fracture network $\Gamma$, we have that $|| \cdot||_{b,DG}$ is a norm on the bulk space $Q_h^b$. In the case of a general fracture network, this holds true only if every connected component of $\Omega \setminus \bar{\Gamma}$ does have part of its boundary in $\partial \Omega_D$. Otherwise, $|| \cdot||_{b,DG}$ is only a seminorm. 
Similarly, we have that $|| \cdot ||_{\Gamma,DG}$ is a norm on the network space $Q_h^\Gamma$, provided that the network is non-immersed, that is $\partial \Gamma_D \neq \emptyset$. However, we remark that, thanks to the coupling term $|| \cdot ||_{\mathcal{C}}$, we have that $||| \cdot |||$ is a norm on $Q_h^b \times Q_h^\Gamma$ for every possible configuration of the fracture network, including the totally immersed case. Moreover, $|||\cdot|||$ is well defined also on the extended space $Q^b(h) \times Q^\Gamma(h)$ introduced in \eqref{spazio esteso bulk}-\eqref{spazio esteso fracture}.
\end{oss}

Following \cite{poligoni1,poligoni2,antonietti2015review,cangiani2016hp,mioUnafrattura}, the analysis will be based on the introduction of an appropriate \emph{inconsistent} formulation and, consequently, on Strang's second Lemma, \cite{strang}. To this end, we introduce the following \emph{extensions} of the forms $\mathcal{A}_b(\cdot, \cdot)$ and $\mathcal{A}_\Gamma(\cdot, \cdot)$ and $\mathcal{L}_b(\cdot)$ and $\mathcal{L}_\Gamma(\cdot)$ to the space $Q^b(h) \times Q^\Gamma(h)$:
\begin{align} \label{forme estese}
\tilde{\mathcal{A}}_b(p,q) & =  \int_{\T} \boldsymbol{\nu} \nabla p_h \cdot \nabla q 
 - \int_{\mathcal{F}_h^I \cup \mathcal{F}_h^D} \{ \boldsymbol{\nu} \boldsymbol{\Pi}_{\textbf{W}_h^b}(\nabla p_h) \} \cdot \llbracket q \rrbracket \\ &\;  -  \int_{\mathcal{F}_h^I \cup \mathcal{F}_h^D} \{ \bnu \boldsymbol{\Pi}_{\textbf{W}_h^b}( \nabla q) \} \cdot \llbracket p \rrbracket  +   \int_{\mathcal{F}_h^I \cup \mathcal{F}_h^D}  \sigma_F  \llbracket p \rrbracket \cdot  \llbracket q \rrbracket  \\
\tilde{\mathcal{A}}_\Gamma(p_{\Gamma,h},q_\Gamma) & =\int_{\Gamma_h} \boldsymbol{\nu}_{\Gamma}^{\tau}  \ell_{\Gamma} \nabla p_{\Gamma} \cdot \nabla q_{\Gamma} 
  -  \int_{\E^I \cup \E^D \cup \E^\cap} \{ \bnu_{\Gamma}^{\tau}  \ell_{\Gamma} \boldsymbol{\Pi}_{\textbf{W}_h^\Gamma}(\nabla p_{\Gamma})  \} \cdot \llbracket q_\Gamma  \rrbracket \\
  & \; -  \int_{\E^I \cup \E^D \cup \E^\cap} \{ \bnu_{\Gamma}^{\tau}  \ell_{\Gamma} \boldsymbol{\Pi}_{\textbf{W}_h^\Gamma}(\nabla q_{\Gamma} ) \} \cdot \llbracket p_{\Gamma}  \rrbracket     +  \int_{\E^I \cup \E^D \E^\cap} \sigma^\Gamma_e \llbracket p_{\Gamma}  \rrbracket \cdot \llbracket q_\Gamma  \rrbracket  \\
\tilde{\mathcal{L}}_b(q)&= \int_{\mathcal{T}_h} fq - \int_{\mathcal{F}_h^D} (\bnu \boldsymbol{\Pi}_{\textbf{W}_h^b}(\nabla q) \cdot \textbf{n}_F - \sigma_F q)g_D, \\
\tilde{\mathcal{L}}_{\Gamma}(q_{\Gamma})&= \int_{\Gamma_h} \ell_\Gamma f_{\Gamma} q_{\Gamma} -  \int_{\E^D} ( \boldsymbol{\nu}_{\Gamma}^{\tau}  \ell_{\Gamma} \boldsymbol{\Pi}_{\textbf{W}_h^\Gamma}( \nabla q_{\Gamma}) \cdot \boldsymbol{\tau}_e - \sigma_e^\Gamma q_\Gamma )g_\Gamma. 
\end{align}
They were obtained by replacing the trace of the gradient operators $\nabla$ and $\nabla_{\tau}$ with the trace of their $L^2$-projection onto the DG vector-valued spaces $\textbf{W}_h^b$ and $\textbf{W}_h^\Gamma$, respectively.  
It follows that these newly introduced forms are well-defined on $Q^b(h) \times Q^\Gamma(h)$ and that they coincide with the formers on the discrete space $Q_h^b \times Q_h^\Gamma$. This means, in particular, that we can consider for the analysis the following \emph{equivalent} version of the discrete problem \eqref{discrete formulation}:\\
Find $(p_h, p_{h,\Gamma}) \in Q_h^b \times Q_h^{\Gamma}$ such that
 \begin{equation} \label{inconsistent discrete formulation fully coupled}
 \tilde{\mathcal{A}}_h \left((p_h,p_h^{\Gamma}), (q, q_{\Gamma}) \right) = \tilde{\mathcal{L}}_h(q,q_{\Gamma}) \;\;\;\; \forall (q, q_{\Gamma}) \in  Q_h^b \times Q_h^{\Gamma},
 \end{equation}
 where $\tilde{\mathcal{A}}_h$ is obtained from $\mathcal{A}_h$ by replacing the bilinear forms with their extended versions \eqref{forme estese}.
 Note that formulation \eqref{inconsistent discrete formulation fully coupled} is no longer consistent due to the discrete nature of the $L^2$-projection operators.\\
 
Next, we prove that problem \eqref{inconsistent discrete formulation fully coupled} extended to the space  $Q^b(h) \times Q^\Gamma(h)$ is well-posed. This, on the one hand, will ensure the well-posedness of discrete problem \eqref{discrete formulation} and, on the other hand, will be of future use in the error analysis.
We remark that the results involving the bulk problem, contained in \cite{mioUnafrattura} in the case of one single non-immersed fracture, can be easily extended to the case of a network of fractures. For this reason, our focus will mainly be on the fracture problem. 

Following \cite{mioUnafrattura}, one can prove that the bulk bilinear form $\tilde{\mathcal{A}}_b$ is continuous and coercive:
\begin{lem} \label{continuita coercivita bulk}
Let $\sigma: \mathcal{F}_h^I \cup \F^D \rightarrow \R^+$ be defined as in \eqref{sigma bulk}. Then, if Assumption \ref{meshes are xxx} holds, the bilinear form $\tilde{\mathcal{A}}_b(\cdot, \cdot)$ is continuous on $Q^b(h) \times Q^b(h)$ and, provided that $\sigma_0$ is sufficiently large, it is also coercive on $Q^b(h) \times Q^b(h)$, i.e.,
 \begin{equation}
 \tilde{\mathcal{A}}_b (p,q) \lesssim ||q||_{b,DG} ||p||_{b,DG}, \quad  \quad \quad  \tilde{\mathcal{A}}_b(q, q) \gtrsim ||q||_{b,DG}^2, 
 \end{equation}
for any  $q,p\in Q^b(h)$.
\end{lem}

\begin{proof}
We refer to Lemma 7.4 in \cite{mioUnafrattura}.
\end{proof}  

Next, we prove an analogous result for the problem in fracture network.
\begin{lem} \label{continuita coercivita network}
Let $\sigma^\Gamma: \E^I \cup \E^D \cup \E^\cap \rightarrow \R^+$ be defined as in \eqref{sigma fracture}. Then, if Assumption \ref{meshes are xxx} holds, the bilinear form $\tilde{\mathcal{A}}_\Gamma(\cdot, \cdot)$ is continuous on $Q^\Gamma(h) \times Q^\Gamma(h)$ and, provided that $\sigma_0^\Gamma$ is sufficiently large, it is also coercive on $Q^\Gamma(h) \times Q^\Gamma(h)$, i.e.,
 \begin{equation}
 \tilde{\mathcal{A}}_\Gamma(p_\Gamma,q_\Gamma) \lesssim ||q_\Gamma||_{\Gamma,DG} ||p_\Gamma||_{\Gamma,DG}, \quad  \quad \quad  \tilde{\mathcal{A}}_\Gamma(q_\Gamma, q_\Gamma) \gtrsim ||q_\Gamma||_{\Gamma,DG}^2, 
 \end{equation}
for any  $q_\Gamma,p_\Gamma\in Q^\Gamma(h)$.
\end{lem}
\begin{proof}
We start with coercivity. For any $q_\Gamma \in Q^\Gamma(h)$, we have
\begin{align} \label{coerc fratt tot}
\tilde{\mathcal{A}}_\Gamma(q_\Gamma, q_\Gamma)&= ||q_\Gamma||_{DG}^2 - 2 \int_{\mathcal{E}_{\Gamma,h}^I \cup \mathcal{E}_{\Gamma,h}^D} \hspace{-3mm} \{ \boldsymbol{\nu}_{\Gamma}^{\tau}  \ell_{\Gamma} \boldsymbol{\Pi}_{\textbf{W}_h^\Gamma}(\nabla q_{\Gamma})  \} \cdot \llbracket q_\Gamma  \rrbracket \\
&\qquad -2 \int_{\mathcal{E}_{\Gamma,h}^\cap} \hspace{-3mm} \{\boldsymbol{\nu}_{\Gamma}^{\tau}  \ell_{\Gamma} \boldsymbol{\Pi}_{\textbf{W}_h^\Gamma}(\nabla q_{\Gamma}) \}_\cap\cdot  \llbracket q_\Gamma  \rrbracket_\cap \\
&= I + II + III
\end{align}

In order to bound term II, we proceed as in \cite{mioUnafrattura}, Lemma 7.4. We employ Cauchy-Schwarz's, triangular and Young's inequalities to write:
\begin{multline}\int_{\E^I \cup \E^D} \hspace{-3mm} \{ \boldsymbol{\nu}_{\Gamma}^{\tau}  \ell_{\Gamma} \boldsymbol{\Pi}_{\textbf{W}_h^\Gamma}(\nabla q_{\Gamma})  \} \cdot \llbracket q_\Gamma  \rrbracket \\
\lesssim 
  \sum_{\E^I \cup \E^D}  \Big[ \varepsilon \int_e (\sigma^\Gamma_e)^{-1} \{\bnu_{\Gamma}^{\tau}  \ell_{\Gamma} \boldsymbol{\Pi}_{\textbf{W}_h^\Gamma}(\nabla q_{\Gamma})\}^2 + \frac{1}{4 \varepsilon} 
   \int_e \sigma^\Gamma_e \llbracket q_\Gamma  \rrbracket^2 \Big]. 
   \end{multline}

From inverse inequality \eqref{inverse estimate poligoni}, the definition of the penalty parameter $\sigma^\Gamma$ \eqref{sigma fracture}, Assumption \ref{meshes are xxx} and the $L^2$-stability of the projector $ \boldsymbol{\Pi}_{\textbf{W}_h^\Gamma}$, we obtain
\begin{multline} \label{coerc fratt dopo young altri}
\int_{\mathcal{E}_{\Gamma,h}^I \cup \mathcal{E}_{\Gamma,h}^D} \{ \boldsymbol{\nu}_{\Gamma}^{\tau}  \ell_{\Gamma}  \boldsymbol{\Pi}_{\textbf{W}_h^\Gamma}(\nabla q_{\Gamma} ) \} \cdot \llbracket q_\Gamma  \rrbracket\\
 \lesssim \frac{\varepsilon}{ \sigma_{0, \Gamma}}  || {(\nug)}^{1/2}  \nabla q_\Gamma||^2_{0,\Gamma_h} + \frac{1}{4 \varepsilon }  ||\sigma_e^{1/2}  \llbracket q_\Gamma \rrbracket ||^2_{0, \E^I \cup \E^D}.
\end{multline}

We now consider the intersection term III. Multiplying and dividing by $\sigma^\cap$ and applying Cauchy-Schwarz's and Young's inequalities we have
\begin{multline} \label{coerc fratt dopo young inters}
\int_{\mathcal{E}_{\Gamma,h}^\cap} \hspace{-3mm} \{\boldsymbol{\nu}_{\Gamma}^{\tau}  \ell_{\Gamma} \boldsymbol{\Pi}_{\textbf{W}_h^\Gamma}(\nabla q_{\Gamma}) \}_\cap\cdot  \llbracket q_\Gamma  \rrbracket_\cap\\
\lesssim  \sum_{e \in \mathcal{E}_{\Gamma,h}^\cap} \Big[ \varepsilon \int_e (\sigma^\cap_e
)^{-1} \{\boldsymbol{\nu}_{\Gamma}^{\tau}  \ell_{\Gamma}  \boldsymbol{\Pi}_{\textbf{W}_h^\Gamma}( \nabla q_{\Gamma})
\}^2_\cap + \frac{1}{4 \varepsilon} \int_e \sigma^\cap_e \llbracket q_\Gamma  \rrbracket^2_\cap \Big].
\end{multline}
Using the definition of $\{ \cdot \}_\cap$ \eqref{salto media intersezione} and triangular inequality, we obtain
\begin{align*}
\int_e &\sigma_\cap^{-1} \{\boldsymbol{\nu}_{\Gamma}^{\tau}  \ell_{\Gamma}  \boldsymbol{\Pi}_{\textbf{W}_h^\Gamma}(\nabla q_{\Gamma}) \}^2_\cap \\
& = \frac{1}{N_\Gamma} \sum_{\substack{i,k=1 \\ i<k}}^{N_\Gamma} \int_e (\sigma_e^\cap)^{-1} (\boldsymbol{\nu}_{\gamma_i}^{\tau}  \ell_{i}  \boldsymbol{\Pi}_{\textbf{W}_h^{\gamma_i}}(\nabla q_{\Gamma}^i) \cdot \boldsymbol{\tau}_i - \boldsymbol{\nu}_{\gamma_k}^{\tau}  \ell_{k}  \boldsymbol{\Pi}_{\textbf{W}_h^{\gamma_k}}(\nabla q_{\Gamma}^k) \cdot \btau_k)^2 \\
& \leq \frac{2}{N_\Gamma} \sum_{\substack{i,k=1 \\ i<k}}^{N_\Gamma} \Big[ \int_e(\sigma_e^\cap)^{-1} (\boldsymbol{\nu}_{\gamma_i}^{\tau}  \ell_{i}  \boldsymbol{\Pi}_{\textbf{W}_h^{\gamma_i}}(\nabla q_{\Gamma}^i) \cdot )^2  + \int_e (\sigma_e^\cap)^{-1} (\boldsymbol{\nu}_{\gamma_k}^{\tau}  \ell_{k}  \boldsymbol{\Pi}_{\textbf{W}_h^{\gamma_k}}(\nabla q_{\Gamma}^k))^2 \Big] \\
&= \frac{2(N_\Gamma -1)}{N_\Gamma} \sum_{k=1}^{N_\Gamma}  \int_e (\sigma_e^\cap)^{-1} (\boldsymbol{\nu}_{\gamma_k}^{\tau}  \ell_{k}  \boldsymbol{\Pi}_{\textbf{W}_h^{\gamma_k}}(\nabla q_{\Gamma}^k))^2,
\end{align*} 
where the last equality follows from the fact that every term appears in the sum exactly $(N_\Gamma-1)$ times.
Since we are assuming that $\ell_{\Gamma} \bnu_{\tau}^\Gamma$ is constant on each $F \in \Gamma_h$, this implies that
\begin{align*}
(&a):=\sum_{e \in \E^\cap}  \int_e \sigma_\cap^{-1} \{\boldsymbol{\nu}_{\Gamma}^{\tau}  \ell_{\Gamma}  \boldsymbol{\Pi}_{\textbf{W}_h^\Gamma}(\nabla q_{\Gamma}) \}^2_\cap \\
& \leq \varepsilon \frac{2(N_\Gamma-1)}{N_\Gamma} \sum_{k=1}^{N_\Gamma} \sum_{\substack{F \in \gamma_{k,h} \\ \partial F \cap \mathcal{I}_\cap \neq \emptyset}} \int_{\partial F} \sigma_\cap^{-1}  (\boldsymbol{\nu}_{\gamma_k}^{\tau}  \ell_{k}  \boldsymbol{\Pi}_{\textbf{W}_h^{\gamma_k}}(\nabla q_{\Gamma}^k))^2 \\
& \leq 
\varepsilon \frac{2(N_\Gamma-1)}{N_\Gamma} \sum_{k=1}^{N_\Gamma} \sum_{\substack{F \in \gamma_{k,h} \\ \partial F \cap \mathcal{I}_\cap \neq \emptyset}} \frac{1}{\sigma_{0, \Gamma}} \Big( \frac{\bar{\boldsymbol{\nu}}_F^\tau (k_F + 1)(k_F+d-1)}{h_F} \Big)^{-1} \bar{\boldsymbol{\nu}}_F^\tau  ||(\boldsymbol{\nu}_{\Gamma}^{\tau}  \ell_{\Gamma})^{1/2} \nabla q_{\Gamma}^k ||_{0, \partial F}^2,
\end{align*}
where we have employed the definition of $\sigma^\Gamma$ and the fact that for all $e \subseteq \partial F$ 
\[\sigma^\Gamma_e \geq \sigma_{0, \Gamma}  \frac{\bar{\boldsymbol{\nu}}_F^\tau (k_F + 1)(k_F+d-1)}{h_F}.\]
Note that this is also true if $e \subseteq \mathcal{I}_\cap$.
Finally, employing inverse inequality \eqref{inverse estimate poligoni} and the stability of the projection operator $ \boldsymbol{\Pi}_{\textbf{W}_h^\Gamma}$ we have
\begin{equation}\label{ultima coerc}
(a) \lesssim  \frac{\varepsilon}{\sigma_{0, \Gamma}} ||(\boldsymbol{\nu}_{\Gamma}^{\tau}  \ell_{\Gamma})^{1/2} \nabla q_{\Gamma} ||_{0, \Gamma_h}^2.
\end{equation}
From \eqref{coerc fratt tot}, employing the derived bounds \eqref{coerc fratt dopo young altri}, \eqref{coerc fratt dopo young inters} and \eqref{ultima coerc}, we obtain that the bilinear $\tilde{\mathcal{A}}_\Gamma(\cdot, \cdot)$ form is coercive, provided that the parameter $\sigma_{0, \Gamma}$ is chosen big enough. Continuity can be proved with analogous arguments.
\end{proof}

Employing Lemma \ref{continuita coercivita bulk} and Lemma \ref{continuita coercivita network}, we can now prove the well-posedness of the discrete problem \eqref{discrete formulation}.

\begin{prop} \label{well posed discrete}
Let the penalization parameters $\sigma$ for the problem in the bulk and in the fracture network be defined as in \eqref{sigma bulk} and \eqref{sigma fracture}, respectively. Then, problem \eqref{discrete formulation} is well-posed, provided that $\sigma_0$ and $\sigma_{0, \Gamma}$ are chosen big enough.
\end{prop}
\begin{proof}
In order to use Lax-Milgram Theorem, we prove that the bilinear form $\tilde{\mathcal{A}}_h(\cdot, \cdot)$ is continuous and coercive on $Q^b(h) \times Q^\Gamma(h)$. We have, from Cauchy-Schwarz's inequality
\begin{align*}
\mathcal{C}((q,q_{\Gamma}), (q, q_{\Gamma}))  &= ||(q, q_{\Gamma})||_{\mathcal{C}}^2 \\
\mathcal{C}((q, q_{\Gamma}),(p, p_{\Gamma})) &\leq  \sum_{F \in \Gamma_h}||\beta_{\Gamma}^{1/2} \llbracket q \rrbracket||^2_{L^2(F)} ||\beta_{\Gamma}^{1/2}\llbracket p \rrbracket||^2_{L^2(F)}  \\\quad& \quad+  \sum_{F \in \Gamma_h}|| \alpha_{\Gamma}^{1/2}(\{q\}-q_{\Gamma})||^2_{L^2(F)}|| \alpha_{\Gamma}^{1/2}(\{p\}-p_{\Gamma})||^2_{L^2(F)} \\ \quad& \leq  |||(q, q_{\Gamma})||| \cdot |||(p, p_{\Gamma})|||,
\end{align*}
so that coercivity and continuity are a direct consequence of the definition of the norm $||| \cdot|||$ and of Lemmas \ref{continuita coercivita bulk} and \ref{continuita coercivita network}. The continuity of the linear operator $ \tilde{\mathcal{L}}_h(\cdot)$ can be easily proved by using the Cauchy-Schwarz inequality, thanks to the regularity assumptions on the forcing terms $f$ and $f_{\Gamma}$ and on the boundary data $g_D$ and $g_\Gamma$.
\end{proof}

\section{Error analysis}\label{ntw:sec: error}
In this section, we derive a-priori error estimates for the discrete problem \eqref{discrete formulation}. To this aim, in the following, we summarize the results contained in \cite{poligoni1,poligoni2,antonietti2015review,cangiani2016hp, libropoligoni}, where standard $hp$-approximation bounds on simplices are extended to arbitrary polytopic elements. These results are indeed the basic tool for the error analysis of DG-methods.
\subsection{$hp$-approximation bounds}
All the theory is based on the existence of a suitable covering of the polytopic mesh, made of a set of overlapping simplices \cite{libropoligoni}:
\begin{defi} \label{simplex covering}
A \emph{covering} $\mathcal{T}_{\#} = \{T_E\}$ related to the polytopic mesh $\mathcal{T}_h$ is a set of shape-regular $d$-dimensional simplices $T_E$, such that for each $E \in \mathcal{T}_h$, there exists a $T_E \in \mathcal{T}_{\#}$ such that $E \subsetneq T_E$.
\end{defi}

\begin{ass} \label{A2 bulk + fracture}
\cite{poligoni1,poligoni2,antonietti2015review,cangiani2016hp, libropoligoni}
There exists a covering $\mathcal{T}_{\#}$ of $\mathcal{T}_h$ and a positive constant $O_{\Omega}$, independent of the mesh parameters, such that
\[\max_{E \in \mathcal{T}_h} card\{ E' \in \mathcal{T}_h: \,E' \cap T_E \neq \emptyset, \, T_E \in \mathcal{T}_{\#} \;\,\mbox{s.t.} \,\; E \subset T_E \} \leq O_{\Omega},\]
and $ h_{T_E} \lesssim h_E$ for each pair $E \in \T$ and $T_E \in \mathcal{T}_{\#}$, with $E \subset T_E$.\\

Moreover, there exists a covering $\mathcal{F}_{\#}$ of $\Gamma_h$ and a positive constant $O_{\Gamma}$, independent of the mesh parameters, such that
\[\max_{F \in \Gamma_h}  card\{ F' \in \Gamma_h: \,F' \cap T_F \neq \emptyset, \, T_F \in \mathcal{F}_{\#} \;\,\mbox{s.t.} \,\; F \subset T_F \} \leq O_{\Gamma},\]
and $h_{T_F} \lesssim h_F$ for each pair $F \in \Gamma_h$ and $T_F \in \mathcal{F}_{\#}$, with $F \subset T_F$.
\end{ass}

From this assumption and standard results for simplices, we can state the following approximation result:
 \begin{lem}\label{interpolazione poligoni} \emph{\cite{poligoni1,poligoni2,antonietti2015review,cangiani2016hp, libropoligoni}}
 Let $E\in \mathcal{T}_h$, $F\subset \partial E$ denote one of its faces, and $T_E\in \mathcal{T}_{\#}$ denote the corresponding simplex such that $E \subset T_E$ (see Definition~\ref{simplex covering}). Suppose that $v \in L^2(\Omega)$ is such that $\mathscr{E} v|_{T_E} \in H^{r_E}(T_E)$, for some $r_E\geq0$. Then, if Assumption \ref{meshes are xxx} and \ref{A2 bulk + fracture} are satisfied, there exists $\widetilde{\Pi}v$, such that $\widetilde{\Pi}v|_E \in \mathbb{P}_{k_E}(E)$, and the following bound holds
 \begin{equation} \label{stima interp Hq}
 ||v- \widetilde{\Pi}v||_{H^q(E)} \lesssim \frac{h_E^{s_E-q}}{k_E^{r_E-q}}||\mathscr{E}v||_{H^{r_E}(T_E)}, \quad \quad \quad 0 \leq q \leq r_E.
 \end{equation}
 Moreover, if $r_E>1/2$,
 \begin{equation} \label{stima interp su bordo}
 ||v- \widetilde{\Pi}v||_{L^2(\partial E)} \lesssim \frac{h_E^{s_E-1/2}}{k_E^{r_E-1/2}}||\mathscr{E}v||_{H^{r_E}(T_E)}.
 \end{equation}
Here, $s_E = \min(k_E+1,r_E)$ and the hidden constants depend on the shape-regularity of $T_E$, but are independent of $v$, $h_E$, $k_E$ and the number of faces per element, while $\mathscr{E}$ is the continuous extension operator as defined in \cite{stein}.
 \end{lem}
\begin{proof}
See \cite{poligoni1} for a detailed proof of \eqref{stima interp Hq} and \cite{cangiani2016hp} for the proof of \eqref{stima interp su bordo}.
\end{proof}
Clearly, analogous approximation results can be stated for the fracture faces, if Assumptions \ref{meshes are xxx} and \ref{A2 bulk + fracture} are both satisfied.  

\subsection{Error estimates}
For each subdomain $\omega_j$, $j=1, \dots, N_\omega$, we denote by $\mathscr{E}_j$ the classical continuous extension operator (cf. \cite{stein}, see also \cite{mioUnafrattura}) $\mathscr{E}_j: H^s(\Omega_j) \rightarrow H^s(\R^d)$, for $s \in \N_0$. Similarly, we denote by $\mathscr{E}_{\gamma_k}$ the continuous extension operator  $\mathscr{E}_{\gamma_k}: H^s(\gamma_k) \rightarrow H^s(\R^{d-1})$, for $s \in \N_0$.
We then make the following regularity assumptions for the exact solution $(p,p_\Gamma)$ of problem \eqref{pb continuo weak}:
\begin{ass}\label{ass regolarita estensioni}
Let $\mathcal{T}_{\#}=\{T_E\}$ and $\mathcal{F}_{\#}=\{T_F\}$ denote the associated coverings of $\Omega$ and $\Gamma$, respectively, of Definition \ref{simplex covering}. We assume that the exact solution $(p,p_{\Gamma})$ is such that:
\begin{itemize}
\item[A1. ]  for every $E \in \mathcal{T}_h$, if $E \subset \omega_j$ and $p_j$ denotes the restriction of $p$ to $\omega_j$, it holds $\mathscr{E}_j p_j|_{T_E} \in H^{r_E}(T_E)$, with  $r_E \geq 1+d/2$ and $T_E \in \mathcal{T}_{\#}$ with $E \subset T_E$;
\item[A2. ] for every $F \in \Gamma_h$, if $F \subset \gamma_k$, it holds $\mathscr{E}_{\gamma_k} p_\Gamma^k|_{T_F} \in H^{r_F}(T_F)$, with  $r_F \geq 1+(d-1)/2$ and $T_F \in \mathcal{F}_{\#}$ with $F \subset T_F$.
\end{itemize} 
\end{ass}

From Proposition \ref{well posed discrete} and Strang's second Lemma directly follows this abstract error bound.
\begin{lem} \label{lem strang}
Assuming that the hypotheses of Proposition \ref{well posed discrete} are satisfied, it holds
\begin{multline}
|||(p, p_{\Gamma})-(p_h, p_{\Gamma,h})||| \lesssim \inf_{(q,q_{\Gamma}) \in Q_h^b \times Q_h^{\Gamma}} |||(p, p_{\Gamma})-(q, q_{\Gamma})||| \quad \\
+ \sup_{(w,w_{\Gamma}) \in Q_h^b \times Q_h^{\Gamma}} \frac{|\mathcal{R}_h((p, p_{\Gamma}),(w, w_{\Gamma}))|}{|||(w, w_{\Gamma})|||},
\end{multline}
where the residual $\mathcal{R}_h$ is defined as
\[\mathcal{R}_h((p, p_{\Gamma}),(w, w_{\Gamma}))=\tilde{\mathcal{A}}_h((p,p_{\Gamma}),(w,w_{\Gamma}))- \tilde{\mathcal{L}}_h(w, w_{\Gamma}).\]
\end{lem}

It is easy to show that the residual is the sum of two contributions, one involving only the bulk problem and one involving only the network problem: 
\begin{equation}\label{split residuo}
\mathcal{R}_h((p, p_{\Gamma}),(w, w_{\Gamma}))= \mathcal{R}_b(p, w) + \mathcal{R}_\Gamma(p_\Gamma,w_\Gamma)
\end{equation} 
It follows that, to derive a bound for the global residual, we can bound each of the two contributions separately. Again, we will focus mainly on the term related to the fracture network.

\begin{lem} \label{stima residuo}
Let $(p,p_\Gamma)$ be the exact solution of problem \eqref{pb continuo weak} satisfying the regularity Assumptions \ref{salti grad nulli} and \ref{ass regolarita estensioni}. Then, for every $w \in Q^b(h)$ and $w_\Gamma \in Q^\Gamma(h)$, it holds

\begin{align}
&|\mathcal{R}_b(p,w)|^2 \lesssim  \sum_{E \in \mathcal{T}_h} \frac{h_E^{2(s_E-1)}}{k_E^{2(r_E-1)}} || \mathscr{E}p||^2_{H^{r_E}(T_E)} \Big[ \bar{\boldsymbol{\nu}}_E^2 \max_{F \subset \partial E \setminus (\Gamma \cup \partial \Omega_D)} \sigma_F^{-1}   (  \frac{k_E}{ h_E}  + \frac{k_E^2}{ h_E} )\Big]  \, \cdot \, ||w||_{b, DG}^2, \label{stima residuo bulk}\\
&|\mathcal{R}_\Gamma(p_\Gamma,w_\Gamma)|^2 \lesssim \Bigg( \sum_{F \in \Gamma_h} \frac{h_F^{2(s_F -1)}}{k_F^{2(r_F -1)}} || \mathscr{E}_\Gamma p_\Gamma ||^2_{H^{r_F}(T_F)} \Big[ (\bar{\boldsymbol{\nu}}^{\tau}_{F})^2  \max_{e \subseteq \partial F \setminus (\mathcal{I}_\cap \cup \partial \Gamma_N \cup \partial \Gamma_F)} \sigma_e^{-1} (\frac{k_F}{h_F}+ \frac{k_F^2}{h_F} )\Big] \\
 & +\sum_{k=1}^{N_\Gamma}  \sum_{\substack{F \in \gamma_{h,k}\\ \partial F \cap \mathcal{I}_\cap \neq \emptyset}}  \frac{h_F^{2(s_F-1)}}{k_F^{2(r_F-1)}}  ||\mathscr{E}_{\gamma_k} p_\Gamma^k||^2_{H^{r_F}(T_F)} \Big[ (\bar{\boldsymbol{\nu}}_F^\tau) ^2 \max_{e \subset \partial F \cap \mathcal{I}_\cap} (\sigma_e^\cap)^{-1} (\frac{k_F}{h_F}+ \frac{k_F^2}{h_F} ) \Big] \Bigg) \, \cdot \, ||w_\Gamma||_{\Gamma, DG}^2, \label{stima residuo fracture}
\end{align}
where, in \eqref{stima residuo bulk}, the extension operator $\mathscr{E}$ is to be interpreted as $\mathscr{E}_j$ if $E \subset \Omega_j$. Similarly, in \eqref{stima residuo fracture}, $\mathscr{E}_\Gamma$ is to be interpreted as $\mathscr{E}_{\gamma_k}$ if $F \subset \gamma_k$. 
\end{lem}

\begin{proof}
Integrating by parts elementwise and using the fact that $(p,p_\Gamma)$ satisfies \eqref{pb continuo weak} and the regularity Assumption \ref{salti grad nulli}, we obtain the following expression for the residuals
\begin{align*}
\mathcal{R}_b(p,w)&=\sum_{F \in \F^I \cup \F^D} \int_F \{ \boldsymbol{\nu}(\nabla p -  \boldsymbol{\Pi}_{\textbf{W}_h^b} (\nabla p))\} \cdot \llbracket w \rrbracket, \\
 \mathcal{R}_\Gamma(p_\Gamma,w_\Gamma)&=\sum_{e \in \E^I \cup \E^D \cup \E^\cap} \int_e \{ \boldsymbol{\nu}_{\Gamma}^\tau \ell_\Gamma (\nabla p_\Gamma -  \boldsymbol{\Pi}_{\textbf{W}_h^\Gamma} (\nabla p_\Gamma))\} \cdot \llbracket w_\Gamma \rrbracket. 
\end{align*}
For the proof of \eqref{stima residuo bulk}, we refer to \cite{mioUnafrattura, mioUnified}. Here, we only focus on the proof of \eqref{stima residuo fracture}. To this aim, we consider the following two terms separately:
\begin{align} \label{residuo 2 termini}
(a)&:=\int_{ \E^I \cup \E^D}\{ \ell_\Gamma \bnu_\Gamma^\tau (\nabla p_\Gamma -  \boldsymbol{\Pi}_{\textbf{W}_h^\Gamma} (\nabla p_\Gamma))\} \cdot \llbracket q_\Gamma \rrbracket \\
(b) &:= \int_{ \E^\cap}\{ \ell_\Gamma \bnu_\Gamma^\tau (\nabla p_\Gamma -  \boldsymbol{\Pi}_{\textbf{W}_h^\Gamma} (\nabla p_\Gamma))\}_\cap \cdot \llbracket q_\Gamma \rrbracket_\cap. 
\end{align}
Employing the Cauchy-Schwarz inequality and the definition of norm $||\cdot||_{\Gamma,DG}$, we obtain
\[ |(a)|^2 \lesssim \left(\int_{\E^I \cup \E^D} \sigma_{\Gamma}^{-1} | \{ \ell_\Gamma \bnu_\Gamma^\tau(\nabla p_\Gamma - \boldsymbol{\Pi}_{\textbf{W}_h^\Gamma} (\nabla p_\Gamma))\}|^2 \right) \, \cdot \, ||q_\Gamma||_{\Gamma,DG}^2 . \]

Let $\widetilde{\Pi}$ denote also the vector-valued generalization of the interpolation operator $\widetilde{\Pi}$ defined in \ref{interpolazione poligoni}. Then, using triangular inequality we can write
\begin{multline*}
\sum_{e \in \E^I \cup \E^D} \sigma_{e}^{-1} \int_e | \{ \ell_\Gamma \bnu_\Gamma^\tau(\nabla p_\Gamma - \boldsymbol{\Pi}_{\textbf{W}_h^\Gamma} (\nabla p_\Gamma))\}|^2\\
 \lesssim \sum_{e \in \E^I \cup \E^D} \sigma_e^{-1} \int_e | \{ \ell_\Gamma \bnu_\Gamma^\tau(\nabla p_\Gamma -\widetilde{\Pi}(\nabla p_\Gamma))\}|^2 \\
 + \sum_{e \in \E^I \cup \E^D} \sigma_e^{-1} \int_e | \{ \ell_\Gamma \bnu_\Gamma^\tau \boldsymbol{\Pi}_{\textbf{W}_h^\Gamma} (\nabla p_\Gamma -\widetilde{\Pi}(\nabla p_\Gamma))\}|^2 \equiv (1a)+(2a).
\end{multline*}
Term (1a) can be bounded, employing the approximation results of Lemma \ref{interpolazione poligoni}, as
\[(1a) \lesssim \sum_{F \in \Gamma_h} \frac{h_F^{2(s_F-1)}}{k_F^{2(r_F-1)}} \big( (\bar{\boldsymbol{\nu}}_F^\tau)^2 \max_{e \subset \partial F \setminus (\mathcal{I}_\cap \cup \partial \Gamma_N \cup \partial \Gamma_F) } \sigma_e^{-1} \frac{h_F^{-1}}{k_F^{-1}} \big) ||\mathscr{E}_\Gamma p_\Gamma||^2_{H^{r_F}(T_F)}. \] 
Exploiting, in order: the boundedness of the permeability tensor $\ell_\Gamma \bnu_\Gamma^\tau$, inverse inequality \eqref{eq:inverse estimate poligoni}, the $L^2$-stability of the projector $ \boldsymbol{\Pi}_{\textbf{W}_h^\Gamma} $  and the approximation results of Lemma~\ref{interpolazione poligoni}, we can bound term (1b) as:
\begin{align*}
(2a)& \lesssim \sum_{F \in \Gamma_h} \max_{e \subset \partial F \setminus (\mathcal{I}_\cap \cup \partial \Gamma_N \cup \partial \Gamma_F)} \sigma_e^{-1}(\bar{\boldsymbol{\nu}}_F^\tau )^2 ||  \boldsymbol{\Pi}_{\textbf{W}_h^\Gamma} (\widetilde{\Pi}(\nabla p_\Gamma)-\nabla p_\Gamma)||^2_{L^2(\partial F)}\\
& \lesssim \sum_{F \in \Gamma_h} \max_{e \subset \partial F \setminus (\mathcal{I}_\cap \cup \partial \Gamma_N \cup \partial \Gamma_F)} \sigma_e^{-1} (\bar{\boldsymbol{\nu}}_F^\tau)^2 \frac{k_F^2}{h_F} || \widetilde{\Pi}(\nabla p_\Gamma)-\nabla p_\Gamma||^2_{L^2(F)} \\
& \lesssim \sum_{F \in \Gamma_h} \frac{h_F^{2(s_F-1)}}{k_F^{2(r_F-1)}}  ||\mathscr{E}_\Gamma p_\Gamma||^2_{H^{r_F}(T_F)} \Big(  (\bar{\boldsymbol{\nu}}_F^\tau)^2 \frac{k_F^2}{h_F} \max_{e \subset \partial F \setminus (\mathcal{I}_\cap \cup \partial \Gamma_N \cup \partial \Gamma_F)} \sigma_e^{-1} \Big).
\end{align*}

Next, we consider term (b). Employing the Cauchy-Schwarz inequality and the definition of the average operator at the intersection $\{ \cdot \}_\cap$, we obtain
\begin{equation*}
|(b)|^2 \lesssim \left(\sum_{k=1}^{N_\Gamma} \sum_{ e \in \mathcal{E}_{\gamma_k, h}^\cap}\int_e (\sigma_{e}^\cap)^{-1} | \ell_\Gamma \bnu_\Gamma^\tau(\nabla p_\Gamma^k - \boldsymbol{\Pi}_{\textbf{W}_h^{\gamma_k}} (\nabla p_\Gamma))|^2 \right) \, \cdot \, ||q_\Gamma||_{\Gamma,DG}^2 .
\end{equation*}
Recalling that $\widetilde{\Pi}$ denotes the vector-valued generalization of the interpolation operator of Lemma \ref{interpolazione poligoni},  we can write
\begin{multline*}
\sum_{k=1}^{N_\Gamma} \sum_{ e \in \mathcal{E}_{\gamma_k, h}^\cap}\int_e (\sigma_{e}^\cap)^{-1} | \ell_k \bnu_{\gamma_k}^\tau(\nabla p_\Gamma^k - \boldsymbol{\Pi}_{\textbf{W}_h^{\gamma_k}} (\nabla p_\Gamma))|^2 \\
\lesssim \sum_{k=1}^{N_\Gamma} \sum_{ e \in \mathcal{E}_{\gamma_k, h}^\cap} \Big(\int_e (\sigma_{e}^\cap)^{-1} | \ell_k \bnu_{\gamma_k}^\tau(\nabla p_\Gamma^k -\widetilde{\Pi}(\nabla p_\Gamma^k))|^2 \\
+ \int_e (\sigma_{e}^\cap)^{-1} | \ell_k \bnu_{\gamma_k}^\tau \boldsymbol{\Pi}_{\textbf{W}_h^\Gamma} (\nabla p_\Gamma^k -\widetilde{\Pi}(\nabla p_\Gamma^k))|^2\Big) \equiv (1b)+(2b).
\end{multline*}
Employing arguments analogous to those for bounding terms (1a) and (2a), we can then write 
\[(1b) \lesssim \sum_{k=1}^{N_\Gamma} \sum_{\substack{F \in \gamma_{h,k}\\ \partial F \cap \mathcal{I}_\cap \neq \emptyset}} \frac{h_F^{2(s_F-1)}}{k_F^{2(r_F-1)}} \big( (\bar{\boldsymbol{\nu}}_F^\tau) ^2 \max_{e \subset \partial F \cap \mathcal{I}_\cap } (\sigma_e^\cap)^{-1} \frac{h_F^{-1}}{k_F^{-1}} \big) ||\mathscr{E}_{\gamma_k} p_\Gamma^k||^2_{H^{r_F}(T_F)}, \] 
and 
\begin{align*}
(2b)& \lesssim \sum_{k=1}^{N_\Gamma} \sum_{\substack{F \in \gamma_{h,k}\\ \partial F \cap \mathcal{I}_\cap \neq \emptyset}} \max_{e \subset \partial F \cap \mathcal{I}_\cap} (\sigma_e^\cap)^{-1}(\bar{\boldsymbol{\nu}}_F^\tau )^2 ||  \boldsymbol{\Pi}_{\textbf{W}_h^\Gamma} (\widetilde{\Pi}(\nabla p_\Gamma^k)-\nabla p_\Gamma^k)||^2_{L^2(\partial F )}\\
& \lesssim\sum_{k=1}^{N_\Gamma} \sum_{\substack{F \in \gamma_{h,k}\\ \partial F \cap \mathcal{I}_\cap \neq \emptyset}}  \max_{e \subset \partial F \cap \mathcal{I}_\cap} (\sigma_e^\cap)^{-1} (\bar{\boldsymbol{\nu}}_F^\tau)^2 \frac{k_F^2}{h_F} || \widetilde{\Pi}(\nabla p_\Gamma^k)-\nabla p_\Gamma^k||^2_{L^2(F)} \\
& \lesssim \sum_{k=1}^{N_\Gamma} \sum_{\substack{F \in \gamma_{h,k}\\ \partial F \cap \mathcal{I}_\cap \neq \emptyset}}  \frac{h_F^{2(s_F-1)}}{k_F^{2(r_F-1)}}  ||\mathscr{E}_{\gamma_k} p_\Gamma^k||^2_{H^{r_F}(T_F)} \Big(  (\bar{\boldsymbol{\nu}}_F^\tau)^2 \frac{k_F^2}{h_F} \max_{e \subset \partial F\cap \mathcal{I}_\cap} (\sigma_e^\cap)^{-1} \Big).
\end{align*}
This concludes the proof.
\end{proof}

\begin{teo}\label{stima errore}
Let $\mathcal{T}_{\#}=\{T_E\}$ and $\mathcal{F}_{\#}=\{T_F\}$ denote the associated coverings of $\Omega$ and $\Gamma$, respectively, consisting of shape-regular simplexes as in Definition \ref{simplex covering}, satisfying Assumptions \ref{A2 bulk + fracture}. 
Let $(p,p_{\Gamma})$ be the solution of problem \eqref{pb continuo weak} and $(p_h, p_{\Gamma,h}) \in Q_h^b \times Q_h^{\Gamma}$ be its approximation obtained with the method \eqref{discrete formulation}, with the penalization parameters given by \eqref{sigma bulk} and \eqref{sigma fracture} and $\sigma_0$ and $\sigma_{0,\Gamma}$ sufficiently large. Moreover, suppose that the exact solution $(p,p_{\Gamma})$ satisfies the regularity Assumptions \ref{salti grad nulli} and \ref{ass regolarita estensioni}. 
 Then, the following error bound holds:
\begin{align*}
|||(p,p_{\Gamma})-(p_h, p_{\Gamma,h})|||^2 & \lesssim  \sum_{E \in \mathcal{T}_h} \frac{h_E^{2(s_E-1)}}{k_E^{2(r_E-1)}} G_E(h_E,k_E,\bar{\boldsymbol{\nu}}_E) ||\mathscr{E}p||^2_{H^{r_E}(T_E)} \\
& \quad + \sum_{F \in \Gamma_h} \frac{h_F^{2(s_F -1)}}{k_F^{2(r_F-1)}} G_F(h_F,k_F, \bar{\boldsymbol{\nu}}^\tau_F)|| \mathscr{E}_\Gamma p_{\Gamma}||^2_{H^{r_F}(T_F)}\\
& \quad + \sum_{k=1}^{N_\Gamma} \sum_{\substack{F \in \gamma_{h,k}\\ \partial F \cap \mathcal{I}_\cap \neq \emptyset}} \frac{h_F^{2(s_F -1)}}{k_F^{2(r_F-1)}} G_F^\cap(h_F,k_F, \bar{\boldsymbol{\nu}}^\tau_F)|| \mathscr{E}_{\gamma_k} p_{\Gamma}^k||^2_{H^{r_F}(T_F)},
\end{align*} 
where the $\mathscr{E}p$ is to be interpreted as $\mathscr{E}_j p_j$ when $E \subset \Omega_j$, $j=1, \dots, N_\omega$, and $\mathscr{E}_\Gamma p_\Gamma$ is to be interpreted as $\mathscr{E}_{\gamma_k} p_\Gamma^k$ when $F \subset\gamma_k$, $k=1, \dots, N_\Gamma$.
Here, $s_E= \min(k_E+1, r_E)$ and $s_F = \min(k_F+1, r_F)$ and for every $E \in \T$ and $F \in \Gamma_h$, the constants $G_E$, $G_F$ and $G_F^\cap$ are defined as:
\begin{align*}
G_E(h_E,k_E,\bar{\boldsymbol{\nu}}_E) &= \bar{\boldsymbol{\nu}}_E + h_E k_E^{-1} \max_{F \subset \partial E \setminus \Gamma} \sigma_F 
 + (\alpha_{\Gamma} + \beta_{\Gamma})h_E k_E^{-1} \\
& \quad + \bar{\boldsymbol{\nu}}_E^2  h_E^{-1}k_E  \max_{F \subset \partial E \setminus \Gamma} \sigma_F^{-1} +  \bar{\boldsymbol{\nu}}_E^2 h_E^{-1} k_E^2 \max_{F \subset \partial E \setminus \Gamma} \sigma_F^{-1},\\
G_F(h_F,k_F,\bar{\boldsymbol{\nu}}^\tau_F) &= \bar{\boldsymbol{\nu}}_F^\tau  + h_F k_F^{-1} \max_{e \subseteq \partial F \setminus (\mathcal{I}_\cap \cup \partial \Gamma_N \cup \partial \Gamma_F) } \sigma_e
 + \alpha_{\Gamma} h_F^2 k_F^{-2} \\
& \quad + (\bar{\boldsymbol{\nu}}_F^\tau )^2  h_F^{-1}k_F  \max_{e \subseteq \partial F \setminus (\mathcal{I}_\cap \cup \partial \Gamma_N \cup \partial \Gamma_F) } \sigma_e^{-1} \\
& \quad +  (\bar{\boldsymbol{\nu}}_F^\tau )^2 h_F^{-1} k_F^2 \max_{e \subseteq \partial F \setminus (\mathcal{I}_\cap \cup \partial \Gamma_N \cup \partial \Gamma_F)} \sigma_e^{-1}\\
G_F^\cap(h_F,k_F,\bar{\boldsymbol{\nu}}^\tau_F) &=   h_F k_F^{-1} \max_{e \subseteq \partial F \ \cap \mathcal{I}_\cap} \sigma_e^\cap  \\
& \quad+ (\bar{\boldsymbol{\nu}}_F^\tau )^2  h_F^{-1}k_F  \max_{e \subseteq \partial F \cap \mathcal{I}_\cap } (\sigma_e^\cap)^{-1} +  (\bar{\boldsymbol{\nu}}_F^\tau )^2 h_F^{-1} k_F^2 \max_{e \subseteq \partial F \cap \mathcal{I}_\cap} (\sigma_e^\cap)^{-1}.
\end{align*}
\end{teo}

\begin{proof}
From Lemma \ref{lem strang} we know that the error satisfies the following bound 
\begin{multline}\label{errore da strang}
|||(p, p_{\Gamma})-(p_h, p_{\Gamma,h})||| \lesssim \underbrace{\inf_{(q,q_{\Gamma}) \in Q_h^b \times Q_h^{\Gamma}} |||(p, p_{\Gamma})-(q, q_{\Gamma})||| }_{I}\\  + \underbrace{\sup_{(w,w_{\Gamma}) \in Q_h^b \times Q_h^{\Gamma}} \frac{|\mathcal{R}_h((p, p_{\Gamma}),(w, w_{\Gamma}))|}{|||(w, w_{\Gamma})|||}}_{II}.
\end{multline}
 We estimate the two terms on the right-hand side of \eqref{errore da strang} separately. We can rewrite term I as
 \begin{align*}
I &= \inf_{(q,q_{\Gamma}) \in Q_h^b \times Q_h^{\Gamma}} \Big( ||p- q||_{b,DG}^2 +  ||p_{\Gamma}-q_{\Gamma}||^2_{\Gamma,DG}+  ||(p- q, p_{\Gamma}-q_{\Gamma})||_{\mathcal{C}}^2 \Big)\\
& \leq \underbrace{ ||p- \widetilde{\Pi}p||_{b,DG}^2}_{(a)} + \underbrace{ ||p_{\Gamma}-\widetilde{\Pi} p_\Gamma ||^2_{\Gamma, DG}}_{(b)}  \,+ \underbrace{ ||(p- \widetilde{\Pi}p,p_{\Gamma}- \widetilde{\Pi} p_\Gamma)||_{\mathcal{C}}^2}_{(c)}. 
 \end{align*}
We consider each of the three terms separately. To bound term (a), we exploit the two approximation results stated in Lemma  \ref{interpolazione poligoni}; we obtain that
\begin{align*}
(a)& \leq ||p- \widetilde{\Pi}p||_{b,DG}^2 = \sum_{E \in \mathcal{T}_h} ||\boldsymbol{\nu}^{1/2} \nabla(p-\widetilde{\Pi}p)||^2_{L^2(E)} + \sum_{F \in \F^I \cup \F^D} \sigma_F ||\llbracket p-  \widetilde{\Pi}p \rrbracket||^2_{L^2(F)} \\
&\lesssim  \sum_{E \in \mathcal{T}_h} \Big[ \bar{\boldsymbol{\nu}}_E |p- \widetilde{\Pi}p|^2_{H^1(E)} + (\max_{F \subset \partial E \setminus (\Gamma \cup \partial \Omega_N)} \sigma_F) || p- \widetilde{\Pi}p||^2_{L^2(\partial E)} \Big]\\
& \lesssim \sum_{E \in \mathcal{T}_h} \Big[ \frac{h_E^{2(s_E-1)} }{k_E^{2(r_E-1)}} \bar{\boldsymbol{\nu}}_E || \mathscr{E}p||_{H^{r_E}(T_E)}^2 +  \frac{h_E^{2(s_E-1/2)}}{k_E^{2(r_E-1/2)}} (\max_{F \subset \partial E \setminus (\Gamma \cup \partial \Omega_N)} \sigma_F) ||\mathscr{E}p||^2_{H^{r_E}(T_E)} \Big]\\
& = \sum_{E \in \mathcal{T}_h}   \frac{h_E^{2(s_E-1)}}{k_E^{2(r_E-1)}}|| \mathscr{E}p||_{H^{r_E}(T_E)}^2 \Big( \bar{\boldsymbol{\nu}}_E +   \frac{h_E}{k_E} (\max_{F \subset \partial E \setminus (\Gamma \cup \partial \Omega_N)} \sigma_F) \Big).
\end{align*} 
Using analogous interpolation estimates on the fracture we can bound term (b) as follows:
\begin{align*}
 (b) &\leq ||p_{\Gamma}- \widetilde{\Pi} p_\Gamma||_{\Gamma,DG}^2 \\
 & \lesssim \sum_{F \in \Gamma_h}||\boldsymbol{\nu}_{\Gamma}^\tau \ell_{\Gamma} \nabla (p_\Gamma - \widetilde{\Pi} p_\Gamma )||_{L^2(F)}^2 + \sum_{e \in \E^I \cup \E^D \cup \E^\cap} \sigma_e ||\llbracket p_\Gamma - \widetilde{\Pi} p_\Gamma  \rrbracket ||^2_{L^2(e)} \\
 & \lesssim \sum_{F \in \Gamma_h} \frac{h_F^{2(s_F -1)}}{k_F^{2(r_F -1)}} || \mathscr{E}_\Gamma p_\Gamma ||^2_{H^{r_F}(T_F)}  \left(\bar{\boldsymbol{\nu}}^{\tau}_{F}  + \frac{h_F}{k_F} \max_{e \subseteq \partial F \setminus (\mathcal{I}_\cap \cup \partial \Gamma_N \cup \partial \Gamma_F) } \sigma_e \right) \\
 & \quad  + \sum_{k=1}^{N_\Gamma} \sum_{\substack{F \in \gamma_{h,k}\\ \partial F \cap \mathcal{I}_\cap \neq \emptyset}} \frac{h_F^{2(s_F -1)}}{k_F^{2(r_F-1)}} || \mathscr{E}_{\gamma_k} p_{\Gamma}^k||^2_{H^{r_F}(T_F)}
 \big(  \frac{h_F}{k_F}  \max_{e \subseteq \partial F \ \cap \mathcal{I}_\cap} \sigma_e^\cap \big). 
\end{align*}

Finally, for term (c), we have
\begin{multline*}
(c) \leq ||(p-\widetilde{\Pi}p,p_{\Gamma}-\widetilde{\Pi} p_\Gamma)||_{\mathcal{C}}^2 \leq \beta_{\Gamma} \sum_{F \in \Gamma_h} ||\llbracket p-  \widetilde{\Pi}p \rrbracket||_{L^2(F)}^2 + \alpha_{\Gamma} \sum_{F \in \Gamma_h} ||\{ p-  \widetilde{\Pi}p \}||_{L^2(F)}^2 \\  + \alpha_{\Gamma} \sum_{F \in \Gamma_h} ||p_{\Gamma} - \widetilde{\Pi} p_\Gamma||^2_{L^2(F)}. 
\end{multline*}
 Exploiting the interpolation result \eqref{stima interp su bordo}, we deduce that
 \begin{align*}
 \beta_{\Gamma}\hspace{-2mm} \sum_{F \in \Gamma_h} ||\llbracket p-  \widetilde{\Pi}p \rrbracket||_{L^2(F)}^2 &\leq \beta_{\Gamma} \hspace{-3mm} \sum_{\substack{E \in \mathcal{T}_h \\ \partial E \cap \Gamma \neq \emptyset}} ||p- \widetilde{\Pi}p||^2_{L^2(\partial E)} 
  \lesssim   \beta_{\Gamma} \hspace{-3mm}\sum_{\substack{E \in \mathcal{T}_h \\ \partial E \cap \Gamma \neq \emptyset}} \frac{h_E^{2(s_E-\frac{1}{2})}}{k_E^{2(r_E-\frac{1}{2})}}|| \mathscr{E}p||_{H^{r_E}(T_E)}^2\\
 &=  \beta_{\Gamma} \sum_{ \substack{ E \in \mathcal{T}_h \\ \partial E \cap \Gamma \neq \emptyset}}  \frac{h_E^{2(s_E-1)}}{k_E^{2(r_E-1)}} || \mathscr{E}p||^2_{H^{r_E}(T_E)} \frac{h_E}{k_E}.
 \end{align*}
Similarly, we have
\[\alpha_{\Gamma} \sum_{F \in \Gamma_h} ||\{ p-  \widetilde{\Pi}p \}||_{L^2(F)}^2  \lesssim \alpha_{\Gamma} \sum_{ \substack{ E \in \mathcal{T}_h \\ \partial E \cap \Gamma \neq \emptyset}}  \frac{h_E^{2(s_E-1)}}{k_E^{2(r_E-1)}} || \mathscr{E}p||^2_{H^{r_E}(T_E)} \frac{h_E}{k_E} .\]
Moreover, using interpolation estimates for the fracture network, we obtain
\begin{align*}
\alpha_{\Gamma} \sum_{F \in \Gamma_h} ||p_{\Gamma}- \widetilde{\Pi}p_{\Gamma}||_{L^2(F)}^2 &\lesssim  \alpha_\Gamma \sum_{F \in \Gamma_h} \frac{h_F^{2s_F}}{k^{2r_F}}||  \mathscr{E}p_\Gamma ||^2_{H^{r_F}(T_F)} \\ &= \alpha_\Gamma \sum_{F \in \Gamma_h} \frac{h_F^{2(s_F -1)}}{k_F^{2(r_F -1)}} || \mathscr{E}p_\Gamma ||^2_{H^{r_F}(T_F)} \frac{h_F^2}{k_F^2}.
\end{align*}
Combining all the previous estimates, we can bound term I on the right-hand side of \eqref{errore da strang} as follows:
\begin{multline} \label{stima I}
I   \lesssim  \frac{h_E^{2(s_E-1)}}{k_E^{2(r_E-1)}} || \mathscr{E}p||^2_{H^{r_E}(T_E)} \Big[ \bar{\boldsymbol{\nu}}_E + \frac{h_E}{ k_E}  \max_{F \subset \partial E \setminus (\Gamma \cup \partial \Omega_N)} \sigma_F  + (\alpha_{\Gamma}+ \beta_{\Gamma}) \frac{h_E}{ k_E} \Big] \\
+ \sum_{F \in \Gamma_h} \frac{h_F^{2(s_F -1)}}{k_F^{2(r_F -1)}} || \mathscr{E}_\Gamma p_\Gamma ||^2_{H^{r_F}(T_F)} \left[ \bar{\boldsymbol{\nu}}^{\tau}_{F} + \frac{h_F}{k_F} \max_{e \subseteq \partial F \setminus (\mathcal{I}_\cap \cup \partial \Gamma_N \cup \partial \Gamma_F) } \sigma_e + \alpha_\Gamma \frac{h_F^2}{k_F^2} \right] \\
 + \sum_{k=1}^{N_\Gamma} \sum_{\substack{F \in \gamma_{h,k}\\ \partial F \cap \mathcal{I}_\cap \neq \emptyset}} \frac{h_F^{2(s_F -1)}}{k_F^{2(r_F-1)}} || \mathscr{E}_{\gamma_k} p_{\Gamma}^k||^2_{H^{r_F}(T_F)}
 \left[  \frac{h_F}{k_F}  \max_{e \subseteq \partial F \ \cap \mathcal{I}_\cap} \sigma_e^\cap \right] .
\end{multline}
Finally, the desired estimate follows from the combination of \eqref{stima I}, together with the bound on Term II that derives from what observed in \eqref{split residuo} and Lemma \ref{stima residuo}.
\end{proof}

\section{Numerical experiments} \label{ntw:sec:numerical exp}
In this section we present several numerical examples, with increasing complexity, in order to validate the performance of our method. In the experiments the analytical solution is known, so that we are able to verify the convergence rates obtained in Theorem \ref{stima errore}. We point out that choice of the model coefficients is here made only with the aim of testing the effectiveness of the numerical method and it does not intend to have any physical meaning.
We remark that, in all the presented test cases, the pressure continuity condition at the intersection points \eqref{int pressure continuity} is satisfied, however in some of them the no flux condition \eqref{int no flux} does not hold. To take this into account, we need to modify formulation \eqref{discrete formulation}, adding on the right hand side the term
\begin{equation} \label{flux at intersection}
\int_{\mathcal{E}_{\Gamma,h}^\cap} \llbracket  \boldsymbol{\nu}_{\Gamma}^{\tau}  \ell_{\Gamma} \nabla_{\tau}p_{\Gamma} \rrbracket_\cap\{q_\Gamma\}_\cap,
\end{equation}
where the quantity $ \llbracket  \boldsymbol{\nu}_{\Gamma}^{\tau}  \ell_{\Gamma} \nabla_{\tau}p_{\Gamma} \rrbracket_\cap = \sum_{k=1}^{N_\Gamma} \bnu_{\gamma_k}^{\tau} \ell_k \nabla_{\tau}p_{\Gamma}^k \cdot \boldsymbol{\tau}_k$ is given.\\

For all the experiments we choose quadratic polynomial degree for both the bulk and fracture problems. Moreover, we always choose the permeability tensor in the bulk $\bnu =  \textbf{I}$, so as to focus mainly on the fracture problem.
All the numerical tests have been implemented in MATLAB$^\circledR$ and employ polygonal grids.

\subsection{Example 1: Vertical fracture}
As first test case, we modify a test case presented in \cite{mioUnafrattura}, splitting the single fracture in 3 parts. In particular, we consider the domain $\Omega= (0,1)^2$ and the fracture network composed of the fractures $\gamma_1=\{(x,y) \in \Omega: \;\; x=0.5,\; 0< y < 0.5\}$, $\gamma_2=\{(x,y) \in \Omega: \;\; x=0.5,\; 0.5< y< 0.75\}$ and $\gamma_3=\{(x,y) \in \Omega: \;\; x=0.5,\; 0.75< y< 1\}$, see Figure~\ref{fig: vertical fracture domain}. Note that both the tips of the fracture $\gamma_2$ are intersection tips.

\begin{figure}[!htb] 
\centering
\subfigure[Computational domain]{
\label{fig: vertical fracture domain}
\begin{tikzpicture} [scale=3.5]
\draw[->] (-0.2,0)--(1.2,0) node[right]{$x$};
\draw[->] (0,-0.2)--(0,1.2) node[above]{$y$};
\draw[line width=0.5mm] (0,0)-- (1,0)--(1,1)--(0,1)--(0,0);
\draw[line width=1mm,blueMatlab] (0.5,0)--(0.5,0.5);
\draw[line width=1mm,redMatlab] (0.5,0.5)--(0.5,0.75);
\draw[line width=1mm,yellowMatlab] (0.5,0.75)--(0.5,1);
\filldraw[red, draw= black] (0.5,0.5) circle(0.027);
\filldraw[red, draw= black] (0.5,0.75) circle(0.027);
\node at (0.60,0.25){\small $\gamma_1$};
\node at (0.60,0.62){\small $\gamma_2$};
\node at (0.60,0.87){\small $\gamma_3$};
\node at (0.5,-0.1){\small $x=0.5$};
\node at (0.35,0.5){\small $\mathcal{I}^\cap_{1,2}$};
\node at (0.35,0.75){\small  $\mathcal{I}^\cap_{2,3}$};

\end{tikzpicture}
}
\subfigure[Solution in the network]{
\label{fig: vertical fracture solution }
\includegraphics[scale=0.45]{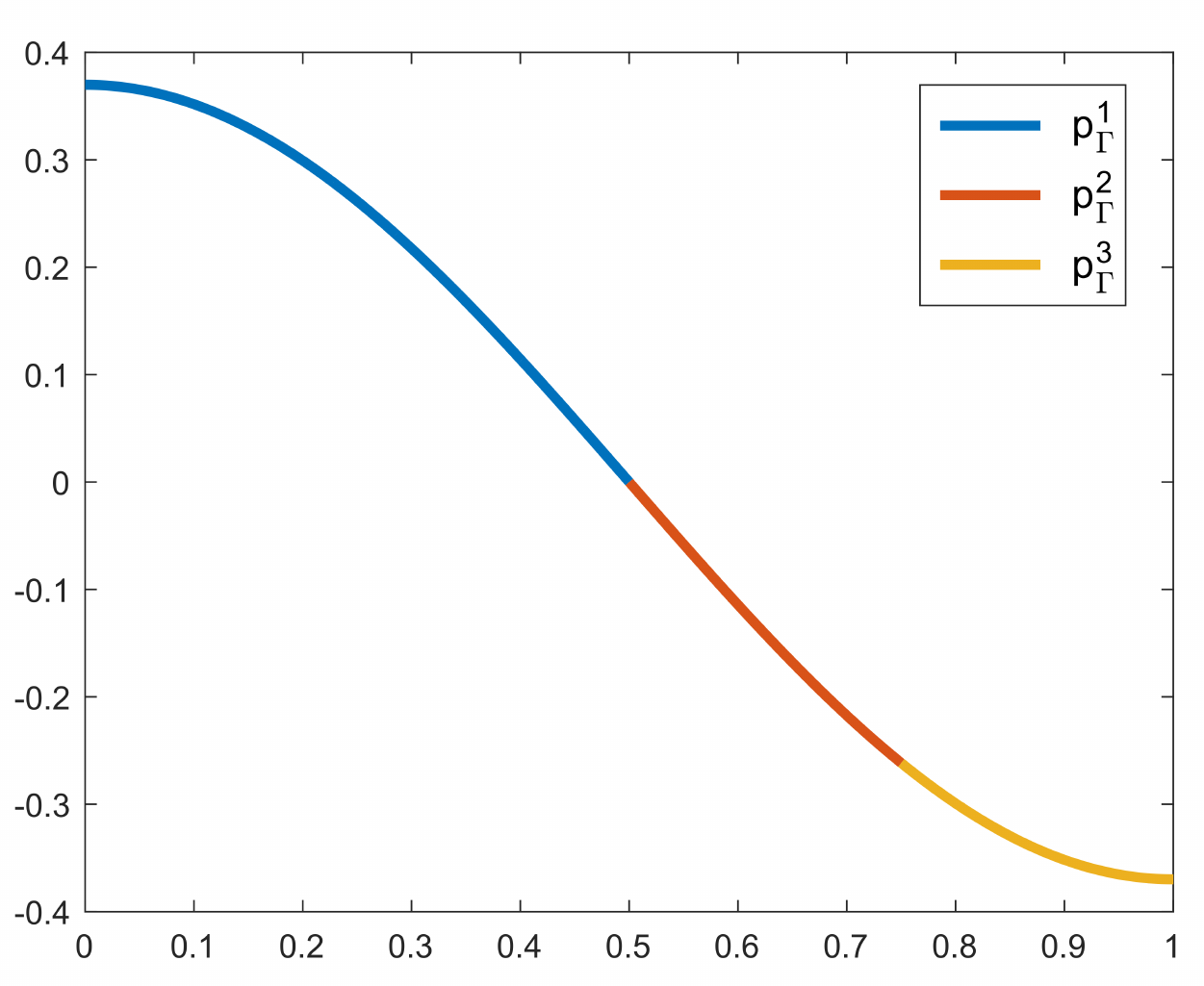}
}
\caption{Example 1: Computational domain (left) and computed fracture pressures (right).}
\end{figure}
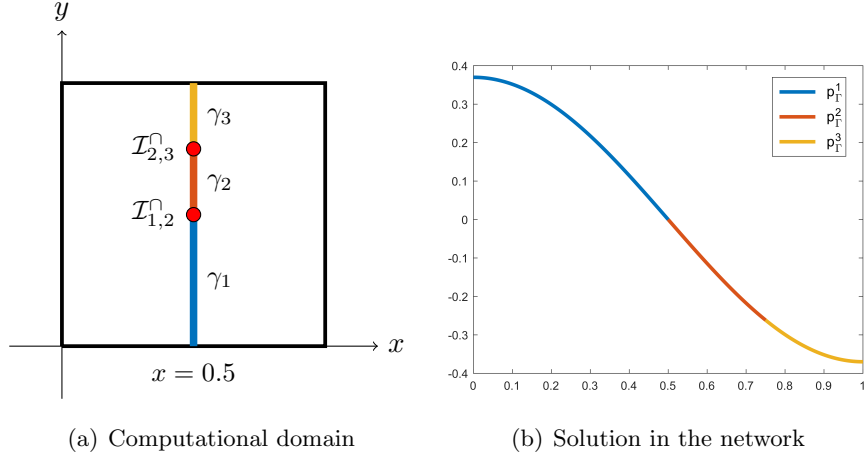
 
We choose the exact solutions in the bulk and in the fractures as follows
 \begin{equation*}
 p  =\begin{cases}  \sin(4x)\cos(\pi y)& \mbox{if } x<0.5, \\ \cos(4x)\cos(\pi y) & \mbox{if } x>0.5, \end{cases} \quad 
   p_{\Gamma}^k=\xi[\cos(2) + \sin(2)] \cos(\pi y), \;k=1,2,3
 \end{equation*}
 so that they satisfy the coupling conditions \eqref{CC} with $\boldsymbol{\nu}= \textbf{I}$, provided that, $\forall k=1,2,3$, we choose $\beta_{\gamma_k}=2$ that is $\bnu_{\gamma_k}^n / \ell_k= 4$.  We impose Dirichlet boundary conditions on the whole $\partial \Omega$ and also on $\partial \Gamma$. Finally, the source term in the bulk is chosen accordingly as 
  \begin{equation*}
   f = \begin{cases} \sin(4x)\cos(\pi y)(16 + \pi^2) & \mbox{if }x<0.5, \\\cos(4x)\cos(\pi y)(16+\pi^2) & \mbox{if } x>0.5,\end{cases} 
\end{equation*}  
and, given $\forall k=1,2,3$ the values $\boldsymbol{\nu}_{\gamma_k}^{\tau}$of the tangential components of the permeability tensor in the fracture, the fracture forcing terms are set as 
 \begin{equation*} 
   f_{\Gamma}^k = \cos(\pi y)[\cos(2) + \sin(2)](\xi \boldsymbol{\nu}_{\gamma_k}^{\tau} \pi^2 + \frac{4}{\ell_k}).
\end{equation*}  
Clearly, pressure continuity at the intersection point \eqref{int pressure continuity} is satisfied regardless of the values chosen for the fracture coefficients $\bnu_\Gamma^\tau$, $\nu_\Gamma^n$ and $\ell_\Gamma$. However, flux conservation \eqref{int no flux} does not hold if the values vary from fracture to fracture. For this reason, we need to modify the right hand side of the formulation as in \eqref{flux at intersection}.

We perform two simulations, varying the values of the fracture coefficients (always satisfying the constraint $\beta_\Gamma=2$). In particular, we take
\begin{itemize} 
\item Case (a):
\vspace{-.75cm}
\begin{align} \label{coeff fratt vert A}
\bnu_\Gamma^\tau &= [3 \cdot 10^4 \;\; 2 \cdot 10^3 \;\; 4\cdot 10^4 ],\\
 \nu_\Gamma^n &=4*[10^{-4} \; 10^{-2} \; 10^{-5}],\\
 \ell_\Gamma &=[10^{-4} \; 10^{-2} \; 10^{-5}];
\end{align} 
\item Case (b):
\vspace{-.75cm}
\begin{align} \label{coeff fratt vert B}
\bnu_\Gamma^\tau &= [3 \cdot 10^{-4} \;\; 2\cdot 10^{-3} \;\; 4\cdot 10^{-4} ],\\
 \nu_\Gamma^n &=[10^{4} \; 10^{2} \; 10^{5}],\\
 \ell_\Gamma &=0.25*[10^{4} \; 10^{2} \; 10^{5}];
\end{align} 
\end{itemize}
Finally, in all the experiments we set $\xi = 0.75$. \\

\noindent In Figure~\ref{fig: vertical fracture solution } we show the numerical solution for the problem in the fracture network for case (a), where one can clearly see that the continuity condition at the intersection points \eqref{int pressure continuity} is satisfied.
In Figures \ref{fig:vertical errori caso a}-\ref{fig:vertical errori caso b} we report the computed errors $||p-p_h||_{b,DG}$ (loglog scale) for the bulk problem as a function of the inverse of the mesh size $1/h$ and the corresponding computed errors $||p_{\Gamma}- p_{\Gamma,h}||_{\Gamma,DG}$ (loglog scale) in the fracture network. We recall that we are taking the polynomial degree $k=2$ for both the bulk and fracture problems. On the left we show the results obtained for test case (a) (with coefficients as in \eqref{coeff fratt vert A}), while on the right, we report the results for the case (b) (with coefficient as in \eqref{coeff fratt vert B}). As predicted from our theoretical error bounds, a convergence of order 2 is clearly observed for both $||p-p_h||_{b,DG}$ and  $||p_{\Gamma}- p_{\Gamma,h}||_{\Gamma,DG}$. Moreover, the convergence is improved of one order if we consider the errors in the $L^2$-norms $||p-p_h||_{L^2(\Omega)}$ and $||p_{\Gamma}- p_{\Gamma,h}||_{L^2(\Gamma)}$.   
 
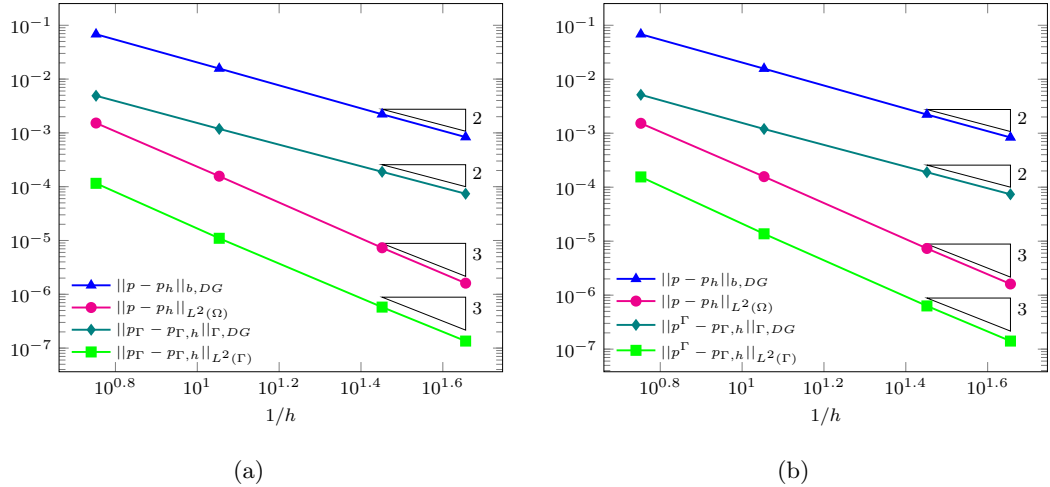
\begin{figure}[!htb]\label{fig:vertical error}
\subfigure[]{
\label{fig:vertical errori caso a}
   \resizebox{0.45\linewidth}{!}{
\begin{tikzpicture}
\scriptsize
\begin{loglogaxis}[
legend entries={$||p-p_h||_{b,DG}$,  $||p-p_h||_{L^2(\Omega)}$, $||p_{\Gamma}- p_{\Gamma,h}||_{\Gamma, DG}$, $||p_{\Gamma}- p_{\Gamma,h}||_{L^2(\Gamma)}$},
legend pos=outer north east,
legend style={font=\tiny, at={(0.01,0.01)},anchor=south west},
legend cell align=left,
 legend style={fill=none,draw=none},
xlabel=$1/h$
]
\addplot [blue, mark=triangle*, line width=0.3mm] table {vertical_fracture/caso_1/errH1.txt};
\addplot [magenta, mark=otimes*, line width=0.3mm] table {vertical_fracture/caso_1/errL2.txt};
\addplot [teal, mark=diamond*, line width=0.3mm] table {vertical_fracture/caso_1/errfracH1.txt};
\addplot [green, mark=square*, line width=0.3mm] table {vertical_fracture/caso_1/errfracL2.txt};

\def \xA {4.5254833995939038e+01};
\def \yA {4.8828125000000011e-04*2.2};
\def \xB {2.8284271247461753e+01 };
\def \yB {1.2500000000000133e-03*2.2};
\def \yC {0.5*\yA + 0.5*\yB};

\draw (\xA, \yA) -- (\xB, \yB)--(\xA, \yB)--cycle;
\draw (\xA, \yC) node [right] {2};

\def \xD {4.5254833995939038e+01 };
\def \yD {1.0789593218788875e-05*0.2};
\def \xE {2.8284271247461753e+01 };
\def \yE {4.4194173824159918e-05*0.2};
\def \yF {0.5*\yD + 0.5*\yE};

\draw (\xD, \yD) -- (\xE, \yE)--(\xD, \yE)--cycle;
\draw (\xD, \yF) node [right] {3};

%

\def \xG {4.5254833995939038e+01};
\def \yG {4.8828125000000011e-04*0.205};
\def \xH {2.8284271247461753e+01 };
\def \yH {1.2500000000000133e-03*0.205};
\def \yL {0.5*\yG + 0.5*\yH};

\draw (\xG, \yG) -- (\xH, \yH)--(\xG, \yH)--cycle;
\draw (\xG, \yL) node [right] {2};

\def \xM {4.5254833995939038e+01 };
\def \yM {1.0789593218788875e-05*0.02};
\def \xN {2.8284271247461753e+01 };
\def \yN {4.4194173824159918e-05*0.02};
\def \yQ {0.5*\yM + 0.5*\yN};

\draw (\xM, \yM) -- (\xN, \yN)--(\xM, \yN)--cycle;
\draw (\xM, \yQ) node [right] {3};

\end{loglogaxis}
\end{tikzpicture}
}
}
\subfigure[]{
\label{fig:vertical errori caso b}
  \resizebox{0.45\linewidth}{!}{
\begin{tikzpicture}
\scriptsize
\begin{loglogaxis}[
legend entries={$||p-p_h||_{b,DG}$,  $||p-p_h||_{L^2(\Omega)}$, $||p^{\Gamma}- p_{\Gamma,h}||_{\Gamma, DG}$, $||p^{\Gamma}- p_{\Gamma,h}||_{L^2(\Gamma)}$},
legend pos=outer north east,
legend style={font=\tiny, at={(0.01,0.01)},anchor=south west},
legend cell align=left,
 legend style={fill=none,draw=none},
xlabel=$1/h$
]
\addplot [blue, mark=triangle*, line width=0.3mm] table {vertical_fracture/caso_4/errH1.txt};
\addplot [magenta, mark=otimes*, line width=0.3mm] table {vertical_fracture/caso_4/errL2.txt};
\addplot [teal, mark=diamond*, line width=0.3mm] table {vertical_fracture/caso_4/errfracH1.txt};
\addplot [green, mark=square*, line width=0.3mm] table {vertical_fracture/caso_4/errfracL2.txt};

\def \xA {4.5254833995939038e+01};
\def \yA {4.8828125000000011e-04*2.2};
\def \xB {2.8284271247461753e+01 };
\def \yB {1.2500000000000133e-03*2.2};
\def \yC {0.5*\yA + 0.5*\yB};

\draw (\xA, \yA) -- (\xB, \yB)--(\xA, \yB)--cycle;
\draw (\xA, \yC) node [right] {2};

\def \xD {4.5254833995939038e+01 };
\def \yD {1.0789593218788875e-05*0.2};
\def \xE {2.8284271247461753e+01 };
\def \yE {4.4194173824159918e-05*0.2};
\def \yF {0.5*\yD + 0.5*\yE};

\draw (\xD, \yD) -- (\xE, \yE)--(\xD, \yE)--cycle;
\draw (\xD, \yF) node [right] {3};

%

\def \xG {4.5254833995939038e+01};
\def \yG {4.8828125000000011e-04*0.205};
\def \xH {2.8284271247461753e+01 };
\def \yH {1.2500000000000133e-03*0.205};
\def \yL {0.5*\yG + 0.5*\yH};

\draw (\xG, \yG) -- (\xH, \yH)--(\xG, \yH)--cycle;
\draw (\xG, \yL) node [right] {2};

\def \xM {4.5254833995939038e+01 };
\def \yM {1.0789593218788875e-05*0.02};
\def \xN {2.8284271247461753e+01 };
\def \yN {4.4194173824159918e-05*0.02};
\def \yQ {0.5*\yM + 0.5*\yN};

\draw (\xM, \yM) -- (\xN, \yN)--(\xM, \yN)--cycle;
\draw (\xM, \yQ) node [right] {3};

\end{loglogaxis}
\end{tikzpicture}
}
}

\caption{Example 1: Computed errors in the bulk and in the fractures as a function of the inverse of the mesh size (loglog scale). Case (a) on the left and case (b) on the right.}
 
\end{figure}

\subsection{Example 2: Y-shaped intersection}
In the second test case we take the bulk $\Omega= (-2,2)^2$ and the fracture network $\Gamma$ consisting of the fractures $\gamma_1=\{(x,y) \in \Omega: \;\; x=y,\; -2< y< 0\}$, $\gamma_2=\{(x,y) \in \Omega: \;\; x=y,\; 0< y< 2\}$ and $\gamma_3=\{(x,y) \in \Omega: \;\; x=0,\; 0< y< 2\}$, see Figure~\ref{fig: Y shaped domain}.
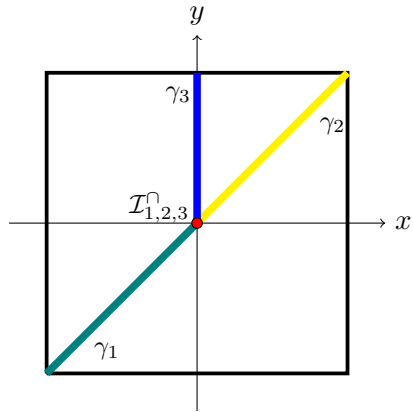
\begin{figure}[!htb]
\centering
\begin{tikzpicture}[scale=1] 
\draw[->] (-2.5,0)--(2.5,0) node[right]{$x$};
\draw[->] (0,-2.5)--(0,2.5) node[above]{$y$};
\draw[line width=0.5mm] (-2,-2)-- (2,-2)--(2,2)--(-2,2)--(-2,-2);
\draw[line width=1mm,teal] (-2,-2)--(0,0);
\draw[line width=1mm,yellow] (0,0)--(2,2);
\draw[line width=1mm,blue] (0,0)--(0,2);
\filldraw[red, draw= black] (0,0) circle(0.07);
\node at (-1.2,-1.7){\small $\gamma_1$};
\node at (1.8,1.3){\small $\gamma_2$};
\node at (-0.25,1.7){\small $\gamma_3$};
\node at (-0.5,0.2){\small $\mathcal{I}^\cap_{1,2,3}$};
\end{tikzpicture}
\caption{Example 2: computational domain.}
\label{fig: Y shaped domain}
\end{figure}
We choose the exact solution in the whole bulk as $p(x,y)=\cos(xy-x^2)$ and the permeability tensor $\boldsymbol{\nu}= \textbf{I}$. Note that, even if the bulk solution is continuous across the fractures, the first coupling condition in \eqref{CC} is satisfied because $\nabla p|_\Gamma=\textbf{0}$. In order for the second coupling condition to hold, we need to choose the solution in the fractures $p_{\Gamma}^k=p|_{\Gamma_k}$ for all $k=1,2,3$, that is $p_{\Gamma}^k=1$. 
Note also that this configuration satisfies the conditions at the intersection \eqref{eq intersezione} irrespective of the choice of the model coefficients.
Finally, the source terms are chosen accordingly as  
 $f =  \cos(xy-x^2)(y^2+5x^2-4xy)-2 \sin(xy-x^2)$ and $f_{\Gamma} =\textbf{0}$.
We impose Dirichlet boundary conditions on the whole $\partial \Omega$ and also on $\partial \Gamma$. In the numerical experiments we choose $\xi = 0.55$.

We perform two simulations, taking the physical parameters in the fracture network as in the previous example, that is for case (a) we choose the coefficients as in \eqref{coeff fratt vert A}, while for case (b) as in \eqref{coeff fratt vert B}.

Figures \ref{fig:Y errori caso a}-\ref{fig:Y errori caso b} show the computed errors (in loglog scale) $||p-p_h||_{b,DG}$ and $||p_{\Gamma}- p_{\Gamma,h}||_{\Gamma,DG}$ for the bulk and fracture problem, respectively (case (a) on the left and case (b) on the right). Also in this case the theoretical convergence rates are achieved and one order is gained for the $L^2$-norm.
\begin{figure}[!htb]\label{fig:Y error}
\subfigure[]{
\label{fig:Y errori caso a}
   \resizebox{0.45\linewidth}{!}{
\begin{tikzpicture}
\scriptsize
\begin{loglogaxis}[
legend entries={$||p-p_h||_{b,DG}$,  $||p-p_h||_{L^2(\Omega)}$, $||p_{\Gamma}- p_{\Gamma,h}||_{\Gamma, DG}$, $||p_{\Gamma}- p_{\Gamma,h}||_{L^2(\Gamma)}$},
legend pos=outer north east,
legend style={font=\tiny, at={(0.01,0.01)},anchor=south west},
legend cell align=left,
 legend style={fill=none,draw=none},
xlabel=$1/h$
]
\addplot [blue, mark=triangle*, line width=0.3mm] table {Y_shaped/caso_1/errH1.txt};
\addplot [magenta, mark=otimes*, line width=0.3mm] table {Y_shaped/caso_1/errL2.txt};
\addplot [teal, mark=diamond*, line width=0.3mm] table {Y_shaped/caso_1/errfracH1.txt};
\addplot [green, mark=square*, line width=0.3mm] table {Y_shaped/caso_1/errfracL2.txt};


\def \xA {1.3081475451951048e+01 };
\def \yA {5.8436815193572671e-03 *4.5};
\def \xB {1.1313708498984759e+01  };
\def \yB {  7.8125000000000017e-03*4.5};
\def \yC {0.5*\yA + 0.5*\yB};

\draw (\xA, \yA) -- (\xB, \yB)--(\xA, \yB)--cycle;
\draw (\xA, \yC) node [right] {2};

\def \xD {1.3081475451951048e+01 };
\def \yD { 4.4671425183049262e-04*0.5};
\def \xE {1.1313708498984759e+01 };
\def \yE {6.9053396600248797e-04*0.5};
\def \yF {0.5*\yD + 0.5*\yE};

\draw (\xD, \yD) -- (\xE, \yE)--(\xD, \yE)--cycle;
\draw (\xD, \yF) node [right] {3};

%

\def \xG {1.3081475451951048e+01 };
\def \yG {5.8436815193572671e-03 *0.0025};
\def \xH {1.1313708498984759e+01  };
\def \yH {  7.8125000000000017e-03*0.0025};
\def \yL {0.5*\yG + 0.5*\yH};

\draw (\xG, \yG) -- (\xH, \yH)--(\xG, \yH)--cycle;
\draw (\xG, \yL) node [right] {2};

\def \xM {1.3081475451951048e+01 };
\def \yM { 4.4671425183049262e-04*0.003};
\def \xN {1.1313708498984759e+01 };
\def \yN {6.9053396600248797e-04*0.003};
\def \yQ {0.5*\yM + 0.5*\yN};

\draw (\xM, \yM) -- (\xN, \yN)--(\xM, \yN)--cycle;
\draw (\xM, \yQ) node [right] {3};

\end{loglogaxis}
\end{tikzpicture}
}
}
\subfigure[]{
\label{fig:Y errori caso b}
  \resizebox{0.45\linewidth}{!}{
\begin{tikzpicture}
\scriptsize
\begin{loglogaxis}[
legend entries={$||p-p_h||_{b,DG}$,  $||p-p_h||_{L^2(\Omega)}$, $||p^{\Gamma}- p_{\Gamma,h}||_{\Gamma, DG}$, $||p^{\Gamma}- p_{\Gamma,h}||_{L^2(\Gamma)}$},
legend pos=outer north east,
legend style={font=\tiny, at={(0.01,0.01)},anchor=south west},
legend cell align=left,
 legend style={fill=none,draw=none},
xlabel=$1/h$
]
\addplot [blue, mark=triangle*, line width=0.3mm] table {Y_shaped/caso_3/errH1.txt};
\addplot [magenta, mark=otimes*, line width=0.3mm] table {Y_shaped/caso_3/errL2.txt};
\addplot [teal, mark=diamond*, line width=0.3mm] table {Y_shaped/caso_3/errfracH1.txt};
\addplot [green, mark=square*, line width=0.3mm] table {Y_shaped/caso_3/errfracL2.txt};

\def \xA {1.3081475451951048e+01 };
\def \yA {5.8436815193572671e-03 *4.2};
\def \xB {1.1313708498984759e+01  };
\def \yB {  7.8125000000000017e-03*4.2};
\def \yC {0.5*\yA + 0.5*\yB};

\draw (\xA, \yA) -- (\xB, \yB)--(\xA, \yB)--cycle;
\draw (\xA, \yC) node [right] {2};

\def \xD {1.3081475451951048e+01 };
\def \yD { 4.4671425183049262e-04*0.45};
\def \xE {1.1313708498984759e+01 };
\def \yE {6.9053396600248797e-04*0.45};
\def \yF {0.5*\yD + 0.5*\yE};

\draw (\xD, \yD) -- (\xE, \yE)--(\xD, \yE)--cycle;
\draw (\xD, \yF) node [right] {3};

%

\def \xG {1.3081475451951048e+01 };
\def \yG {5.8436815193572671e-03 *0.012};
\def \xH {1.1313708498984759e+01  };
\def \yH {  7.8125000000000017e-03*0.012};
\def \yL {0.5*\yG + 0.5*\yH};

\draw (\xG, \yG) -- (\xH, \yH)--(\xG, \yH)--cycle;
\draw (\xG, \yL) node [right] {2};

\def \xM {1.3081475451951048e+01 };
\def \yM { 4.4671425183049262e-04*0.009};
\def \xN {1.1313708498984759e+01 };
\def \yN {6.9053396600248797e-04*0.009};
\def \yQ {0.5*\yM + 0.5*\yN};

\draw (\xM, \yM) -- (\xN, \yN)--(\xM, \yN)--cycle;
\draw (\xM, \yQ) node [right] {3};

\end{loglogaxis}
\end{tikzpicture}
}
}

\caption{Example 2: Computed errors in the bulk and in the fractures as a function of the inverse of the mesh size (loglog scale). Case (a) on the left and case (b) on the right.}
 
\end{figure}
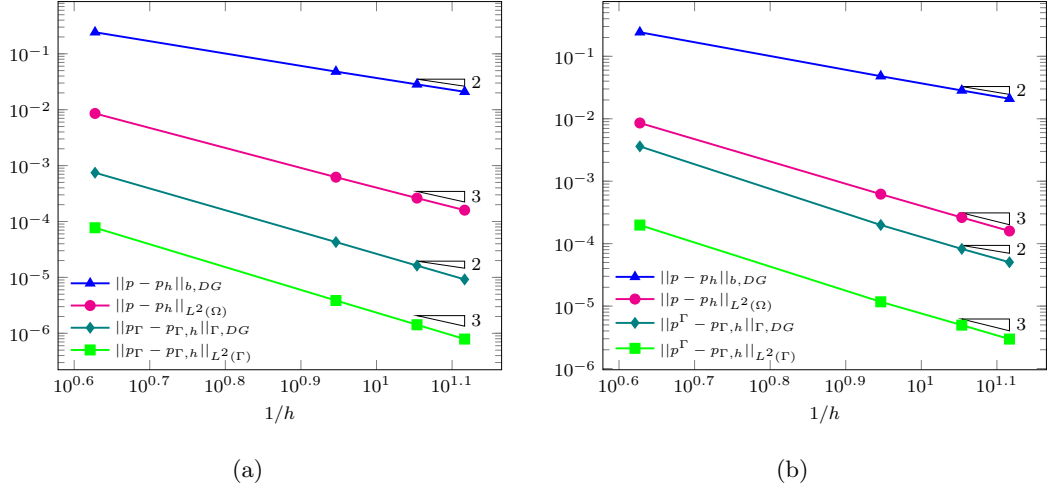

\subsection{Example 3: Checkerboard}
In the third test case we consider a cross-shaped network of fractures cutting the bulk $\Omega= (-1,1)^2$. The fractures are defined as $\gamma_1=\{(x,y) \in \Omega: \;\; y=0,\; -1 <x < -0.5\}$, $\gamma_2=\{(x,y) \in \Omega: \;\; y=0,\; -0.5 < x < 0\}$, $\gamma_3=\{(x,y) \in \Omega: \;\; x=0,\; -1 < y < 0\}$, $\gamma_4=\{(x,y) \in \Omega: \;\; y=0,\; 0< x< 1\}$ and $\gamma_5=\{(x,y) \in \Omega: \;\; x=0,\; 0< x< 1\}$, see Figure~\ref{fig: checkerboard domain}. Note that fracture $\gamma_2$ presents two intersection tips.

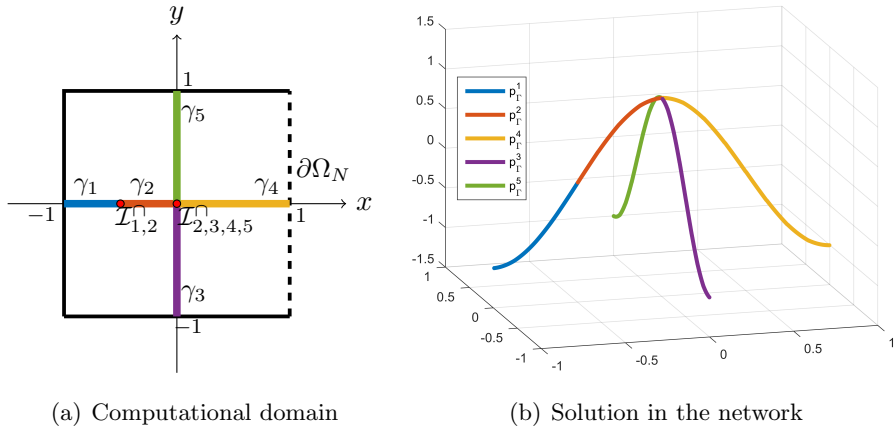
\begin{figure}[!htb] 
\centering
\subfigure[Computational domain]{
\label{fig: checkerboard domain}
\begin{tikzpicture} [scale=1.5] 
\draw[->] (-1.5,0)--(1.5,0) node[right]{$x$};
\draw[->] (0,-1.5)--(0,1.5) node[above]{$y$};
\draw[line width=0.5mm] (-1,-1)-- (1,-1);
\draw[line width=0.5mm, dashed] (1,-1)--(1,1);
\draw[line width=0.5mm] (1,1)--(-1,1)--(-1,-1);
\draw[line width=1mm,blueMatlab] (-1,0)--(-0.5,0);
\draw[line width=1mm,redMatlab] (-0.5,0)--(0,0);
\draw[line width=1mm,purpleMatlab] (0,-1)--(0,0);
\draw[line width=1mm,yellowMatlab] (0,0)--(1,0);
\draw[line width=1mm,greenMatlab] (0,0)--(0,1);
\filldraw[red, draw= black] (0,0) circle(0.033);
\filldraw[red, draw= black] (-0.5,0) circle(0.033);
\node at (-0.8,0.15){\small $\gamma_1$};
\node at (-0.3,0.15){\small $\gamma_2$};
\node at (0.15,-0.8){\small $\gamma_3$};
\node at (0.8,0.15){\small $\gamma_4$};
\node at (0.15,0.8){\small $\gamma_5$};
\node at (-0.35,-0.15){\small $\mathcal{I}^\cap_{1,2}$};
\node at (0.35,-0.15){\small $\mathcal{I}^\cap_{2,3,4,5}$};
\node at (1.3,0.3){\small $\partial \Omega_N$};
\node at (-1.2,-0.1){\scriptsize $-1$};
\node at (1.1,-0.1){\scriptsize $1$};
\node at (0.1,-1.1){\scriptsize $-1$};
\node at (0.1,1.1){\scriptsize $1$};
\end{tikzpicture}
}
\subfigure[Solution in the network]{
\label{fig: checkerboard solution}
\includegraphics[scale=0.5]{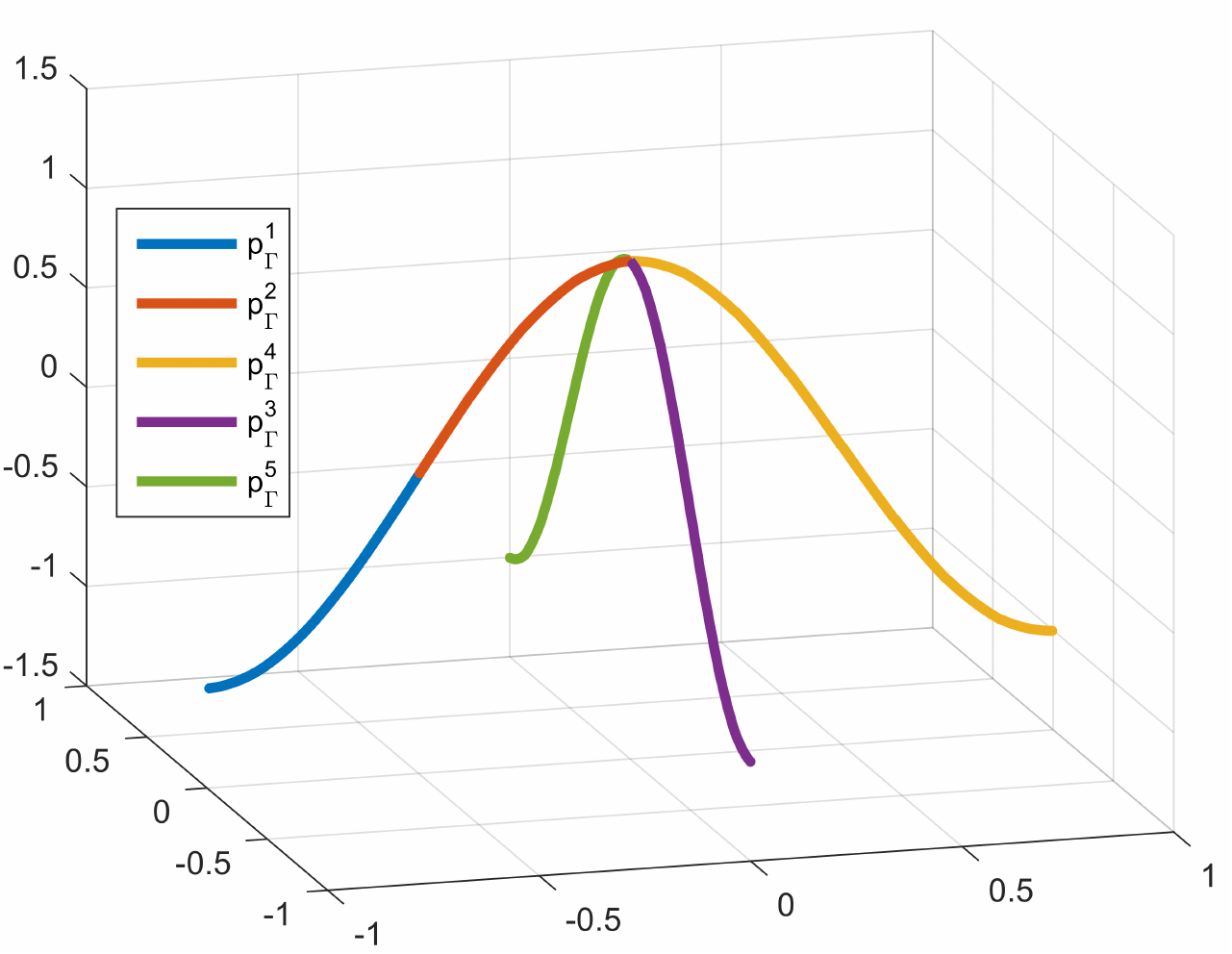}
}
\caption{Example 3: Computational domain (left) and computed fracture pressures plotted as $3d$-lines (right).}
\end{figure}
We choose again a solution in the bulk continuous across the fractures $p(x,y)=\cos(\pi x) \cos(\pi y)$ and the permeability tensor $\boldsymbol{\nu}= \textbf{I}$. 
In this case, the first coupling condition in \eqref{CC} is satisfied because $\nabla p|_\Gamma \cdot \textbf{n}_\Gamma=0$, where $\textbf{n}_{1}=\textbf{n}_{2}=\textbf{n}_{4} =(0,1)^T$ and $ \textbf{n}_{3} =\textbf{n}_{5}=(1,0)^T$.
The validity of the second coupling condition is satisfied if $p_{\Gamma}^k=p|_{\Gamma_k}$ for all $k=1,2,3,4,5$, that is $p_{\Gamma}^1=p_{\Gamma}^2=p_{\Gamma}^4=\cos(\pi x)$ and $p_{\Gamma}^3=p_{\Gamma}^5=\cos(\pi y)$. 

In the bulk, we impose Neumann boundary conditions on $\partial \Omega_N=\{(x,y) \in \Omega: \;\; x=1 \}$ and Dirichlet boundary conditions on the rest of the boundary. Accordingly, at the boundary tips of fractures $\gamma_1$, $\gamma_3$ and $\gamma_5$ we impose Dirichlet conditions, and at the boundary tip of $\gamma_4$ we impose Neumann conditions.
In the numerical experiments we choose $\xi = 0.55$.
Finally, the source term in the bulk is chosen accordingly as  
 $f =  2 \pi^2 \cos(\pi x) \cos(\pi y)$ and, given the physical coefficients $\bnu_{\Gamma_k}^\tau$ and $\ell_k$, for $k=1,2,3,4,5$, the source term for each fracture is $f_{\Gamma}^k = \pi^2 \cos(\pi x) \bnu_{\gamma_k}^\tau$. Note that, at intersection $\mathcal{I}^\cap_{1,2}$ flux conservation does not hold if the values of the coefficients vary from $\gamma_1$ to $\gamma_2$, while at intersection $\mathcal{I}^\cap_{2,3,4,5}$ flux conservation is satisfied for every choice, due to the fact that $\nabla p_\Gamma^k|_{\mathcal{I}^\cap_{2,3,4,5}}=0$, for $k=2,3,4,5$. \\
 
We perform two simulations:
\begin{itemize}
\item in case (a) we take $\ell_k=\bnu_{\gamma_k}^\tau= \nu_{\gamma_k}^n = k \cdot 10^k$, for $k=1,2,3,4,5$;
\item in case (b) we take $\ell_k=\bnu_{\gamma_k}^\tau= \nu_{\gamma_k}^n = k \cdot 10^{-k}$, for $k=1,2,3,4,5$.
\end{itemize}
In Figure~\ref{fig: checkerboard solution} we show the numerical solution for the fracture network problem computed with the coefficients of case (a). The values of the fracture pressures are displayed as lines in the $3d$ space, so that pressure continuity at the intersection points is evident.
The plots in Figures \ref{fig:checkerboard errori caso a}-\ref{fig:checkeborad errori caso b} show the computed errors in loglog scale for the bulk and network problems, together with the expected convergence rates. Test case (a) is on the left and test case (b) is on the right. Once again the results are in agreement with the theoretical estimates.

\begin{figure}[!htb]
\subfigure[]{
\label{fig:checkerboard errori caso a}
   \resizebox{0.45\linewidth}{!}{
\begin{tikzpicture}
\scriptsize
\begin{loglogaxis}[
legend entries={$||p-p_h||_{b,DG}$,  $||p-p_h||_{L^2(\Omega)}$, $||p_{\Gamma}- p_{\Gamma,h}||_{\Gamma, DG}$, $||p_{\Gamma}- p_{\Gamma,h}||_{L^2(\Gamma)}$},
legend pos=outer north east,
legend style={font=\tiny, at={(0.01,0.01)},anchor=south west},
legend cell align=left,
 legend style={fill=none,draw=none},
xlabel=$1/h$
]
\addplot [blue, mark=triangle*, line width=0.3mm] table {checkerboard/caso_1/errH1.txt};
\addplot [magenta, mark=otimes*, line width=0.3mm] table {checkerboard/caso_1/errL2.txt};
\addplot [teal, mark=diamond*, line width=0.3mm] table {checkerboard/caso_1/errfracH1.txt};
\addplot [green, mark=square*, line width=0.3mm] table {checkerboard/caso_1/errfracL2.txt};

\def \xA {2.2627416997969519e+01};
\def \yA {1.9531250000000004e-03*3};
\def \xB {1.1313708498984759e+01};
\def \yB {7.8125000000000017e-03*3};
\def \yC {0.5*\yA + 0.5*\yB};

\draw (\xA, \yA) -- (\xB, \yB)--(\xA, \yB)--cycle;
\draw (\xA, \yC) node [right] {2};

\def \xD {2.2627416997969519e+01 };
\def \yD {8.6316745750310997e-05*0.4};
\def \xE {1.1313708498984759e+01};
\def \yE {6.9053396600248797e-04*0.4};
\def \yF {0.5*\yD + 0.5*\yE};

\draw (\xD, \yD) -- (\xE, \yE)--(\xD, \yE)--cycle;
\draw (\xD, \yF) node [right] {3};

%

\def \xG {2.2627416997969519e+01};
\def \yG {1.9531250000000004e-03*0.7};
\def \xH {1.1313708498984759e+01};
\def \yH {7.8125000000000017e-03*0.7};
\def \yL {0.5*\yG + 0.5*\yH};

\draw (\xG, \yG) -- (\xH, \yH)--(\xH, \yG)--cycle;
\draw (\xH, \yL) node [left] {2};

\def \xM {2.2627416997969519e+01 };
\def \yM {8.6316745750310997e-05*0.05};
\def \xN {1.1313708498984759e+01};
\def \yN {6.9053396600248797e-04*0.05};
\def \yQ {0.5*\yM + 0.5*\yN};

\draw (\xM, \yM) -- (\xN, \yN)--(\xN, \yM)--cycle;
\draw (\xN, \yQ) node [left] {3};

\end{loglogaxis}
\end{tikzpicture}
}
}
\subfigure[]{
\label{fig:checkeborad errori caso b}
  \resizebox{0.45\linewidth}{!}{
\begin{tikzpicture}
\scriptsize
\begin{loglogaxis}[
legend entries={$||p-p_h||_{b,DG}$,  $||p-p_h||_{L^2(\Omega)}$, $||p^{\Gamma}- p_{\Gamma,h}||_{\Gamma, DG}$, $||p^{\Gamma}- p_{\Gamma,h}||_{L^2(\Gamma)}$},
legend pos=outer north east,
legend style={font=\tiny, at={(0.01,0.01)},anchor=south west},
legend cell align=left,
 legend style={fill=none,draw=none},
xlabel=$1/h$
]
\addplot [blue, mark=triangle*, line width=0.3mm] table {checkerboard/caso_2/errH1.txt};
\addplot [magenta, mark=otimes*, line width=0.3mm] table {checkerboard/caso_2/errL2.txt};
\addplot [teal, mark=diamond*, line width=0.3mm] table {checkerboard/caso_2/errfracH1.txt};
\addplot [green, mark=square*, line width=0.3mm] table {checkerboard/caso_2/errfracL2.txt};

\def \xA {2.2627416997969519e+01};
\def \yA {1.9531250000000004e-03*3.2};
\def \xB {1.1313708498984759e+01};
\def \yB {7.8125000000000017e-03*3.2};
\def \yC {0.5*\yA + 0.5*\yB};

\draw (\xA, \yA) -- (\xB, \yB)--(\xA, \yB)--cycle;
\draw (\xA, \yC) node [right] {2};

\def \xD {2.2627416997969519e+01 };
\def \yD {8.6316745750310997e-05*0.4};
\def \xE {1.1313708498984759e+01};
\def \yE {6.9053396600248797e-04*0.4};
\def \yF {0.5*\yD + 0.5*\yE};

\draw (\xD, \yD) -- (\xE, \yE)--(\xD, \yE)--cycle;
\draw (\xD, \yF) node [right] {3};

\end{loglogaxis}
\end{tikzpicture}
}
}

\caption{Example 3: Computed errors in the bulk and in the fractures as a function of the inverse of the mesh size (loglog scale). Case (a) on the left and case (b) on the right.}
 
\end{figure}
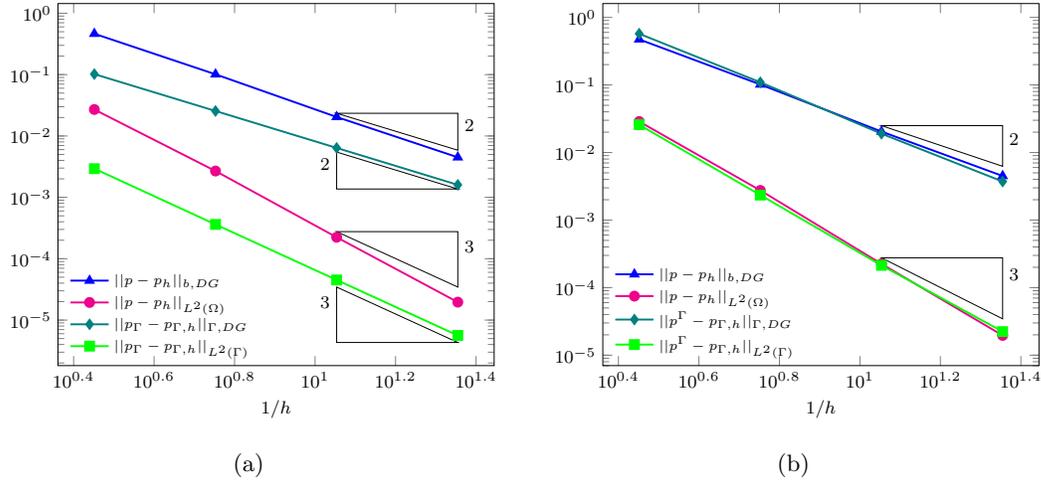

\subsection{Example 4: Cross-shaped intersection}
We consider the domain $\Omega= (0,1)^2$ cut by a cross-shaped network made up of the four fractures $\gamma_1=\{(x,y) \in \Omega: \;\; y=0.5,\; 0 < x< 0.5\}$, $\gamma_2=\{(x,y) \in \Omega: \;\; x=0.5,\; 0< y< 0.5\}$, $\gamma_3=\{(x,y) \in \Omega: \;\; y=0.5,\; 0.5< x< 1\}$ and $\gamma_4=\{(x,y) \in \Omega: \;\; x=0.5,\; 0.5 <y< 1\}$. The bulk domain $\Omega$ is then subdivided into the sets
\begin{align*}
\Omega_A &=\{(x,y) \in \Omega: \;\; 0< x< 0.5, \; 0<y<0.5\}, \\
\Omega_B&=\{(x,y) \in \Omega: \;\; 0.5< x< 1, \; 0<y<0.5\},\\
\Omega_C&=\{(x,y) \in \Omega: \;\; 0.5< x< 1, \; 0.5<y<1\}, \\
\Omega_D&=\{(x,y) \in \Omega: \;\; 0< x< 0.5, \; 0.5<y<1\},
\end{align*}
as shown in Figure~\ref{fig: cross domain}.

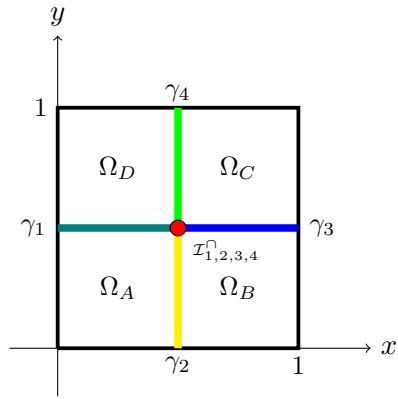
\begin{figure}[!htb]
\centering
\begin{tikzpicture}[scale=3.2] 
\draw[->] (-0.2,0)--(1.3,0) node[right]{$x$};
\draw[->] (0,-0.2)--(0,1.3) node[above]{$y$};
\draw[line width=0.5mm] (0,0)-- (1,0)--(1,1)--(0,1)--(0,0);
\draw[line width=1mm,teal] (0,0.5)--(0.5,0.5);
\draw[line width=1mm,yellow] (0.5,0)--(0.5,0.5);
\draw[line width=1mm,blue] (0.5,0.5)--(1,0.5);
\draw[line width=1mm,green] (0.5,0.5)--(0.5,1);
\filldraw[red, draw= black] (0.5,0.5) circle(0.033);
\node at (-0.1,0.5){\small $\gamma_1$};
\node at (0.5,-0.07){\small $\gamma_2$};
\node at (1.1,0.5){\small $\gamma_3$};
\node at (0.5,1.06){\small $\gamma_4$};
\node at (0.25,0.25){\small $\Omega_A$};
\node at (0.75,0.25){\small $\Omega_B$};
\node at (0.75,0.75){\small $\Omega_C$};
\node at (0.25,0.75){\small $\Omega_D$};
\node at (0.7,0.4){\tiny $\mathcal{I}^\cap_{1,2,3,4}$};
\node at (1,-0.07){\small $1$};
\node at (-0.07,1){\small $1$};
\end{tikzpicture}
\caption{Example 4: computational domain.}
\label{fig: cross domain}
\end{figure}

\noindent In order to define the exact solution for the bulk problem, we introduce the functions
\begin{align*}
p_l&= \sin(\frac{\pi}{2}x) \cos(2 \pi y),\\
p_r&= \cos(\frac{\pi}{2}x) \cos(2 \pi y),\\
p_u&= \cos(\frac{\pi}{2}y) \cos(2 \pi x),\\
p_d&= \sin(\frac{\pi}{2}y) \cos(2 \pi x),
\end{align*}
where the subscript is related to the position (left, right, up, down). The bulk pressure is then defined in each subdomain of $\Omega$ as
 \begin{equation*}
 p(x,y)  =\begin{cases}  p_l + p_d& \mbox{in} \, \Omega_A, \\
  p_r +p_d & \mbox{in} \, \Omega_B, \\
 p_r + p_u& \mbox{in} \, \Omega_C, \\
 p_l + p_u& \mbox{in} \, \Omega_D. 
 \end{cases} 
 \end{equation*}
If we choose the permeability tensor $\boldsymbol{\nu}= \textbf{I}$, the bulk source term will have the following expression $f(x,y)= \frac{17}{4} \pi^2 p(x,y)$.
Simple calculations show that $p(x,y)$ satisfies the first coupling condition in \eqref{CC} provided that for $ k=1,2,3,4$, we choose $\beta_{\gamma_k}=\frac{\pi}{4}$, that is $\bnu_{\gamma_k}^n = \frac{\pi}{2} \ell_k$. From the second coupling condition we deduce the following expressions for the solutions in the fractures
 \begin{align*}
  p_\Gamma^1 &= \xi \sqrt{2} \cos(2 \pi x) -\sin(\frac{\pi}{2}x),\hspace{2cm}
   &p_\Gamma^2 &= \xi \sqrt{2} \cos(2 \pi y) -\sin(\frac{\pi}{2}y),\\
   p_\Gamma^3 &= \xi \sqrt{2} \cos(2 \pi x) -\cos(\frac{\pi}{2}x),\hspace{2cm}
    & p_\Gamma^4 &= \xi \sqrt{2} \cos(2 \pi y) -\cos(\frac{\pi}{2}y).
 \end{align*}
Note that, with this choice, pressure continuity at the intersection point \eqref{eq intersezione} is ensured by the fact that $\cos(\frac{\pi}{4})=\sin(\frac{\pi}{4})$. However, flux conservation does not hold, so that we need to modify the right-hand-side of our formulation as in \eqref{flux at intersection}. 
Finally the source terms for the fracture problems are chosen accordingly as
 \begin{align*}
  f_\Gamma^1 &=  \cos(2 \pi x)[ \frac{\sqrt{2} \pi}{2 \ell_1} + 4 \pi^2 \xi \sqrt{2}  \bnu_{\gamma_1}^{\tau}] - \bnu_{\gamma_1}^{\tau} \frac{\pi^2}{4}\sin(\frac{\pi}{2}x),\\
   f_\Gamma^2 &=\cos(2 \pi y)[ \frac{\sqrt{2} \pi}{2 \ell_2} + 4 \pi^2 \xi \sqrt{2}  \bnu_{\gamma_2}^{\tau}] - \bnu_{\gamma_2}^{\tau} \frac{\pi^2}{4}\sin(\frac{\pi}{2}y),\\
    f_\Gamma^3 &= \cos(2 \pi x)[ \frac{\sqrt{2} \pi}{2 \ell_3} + 4 \pi^2 \xi \sqrt{2}  \bnu_{\gamma_3}^{\tau}] - \bnu_{\gamma_3}^{\tau} \frac{\pi^2}{4}\cos(\frac{\pi}{2}x),\\
    f_\Gamma^4 &= \cos(2 \pi y)[ \frac{\sqrt{2} \pi}{2 \ell_4} + 4 \pi^2 \xi \sqrt{2}  \bnu_{\gamma_4}^{\tau}] - \bnu_{\gamma_4}^{\tau} \frac{\pi^2}{4}\cos(\frac{\pi}{2}y).
 \end{align*}
We perform two simulations choosing the values of the physical coefficients as:
\begin{itemize}
\item in case (a) we take $\bnu_{\gamma_k}^\tau= \nu_{\gamma_k}^n = k \cdot 10^k$ and $\ell_k=\frac{2}{\pi}\nu_{\gamma_k}^n$, for $k=1,2,3,4$;
\item in case (b) we take $\bnu_{\gamma_k}^\tau= k \cdot 10^k$, $\ell_k = k \cdot 10^{-k}$  and $\nu_{\gamma_k}^n=\frac{\pi}{2} \ell_k$, for $k=1,2,3,4$.
\end{itemize}
In Figure~\ref{fig: cross solution bulk} we show the computed numerical solution for the problem in the bulk, with the  coefficients as in case (a). In Figure~\ref{fig: cross solution frac} we show the values of the fracture pressures as lines in the $3d$ space. Pressure continuity at the intersection point is clearly observed.
\begin{figure}[!htb] 
\subfigure[Solution in the bulk]{
\label{fig: cross solution bulk}
\includegraphics[scale=0.55]{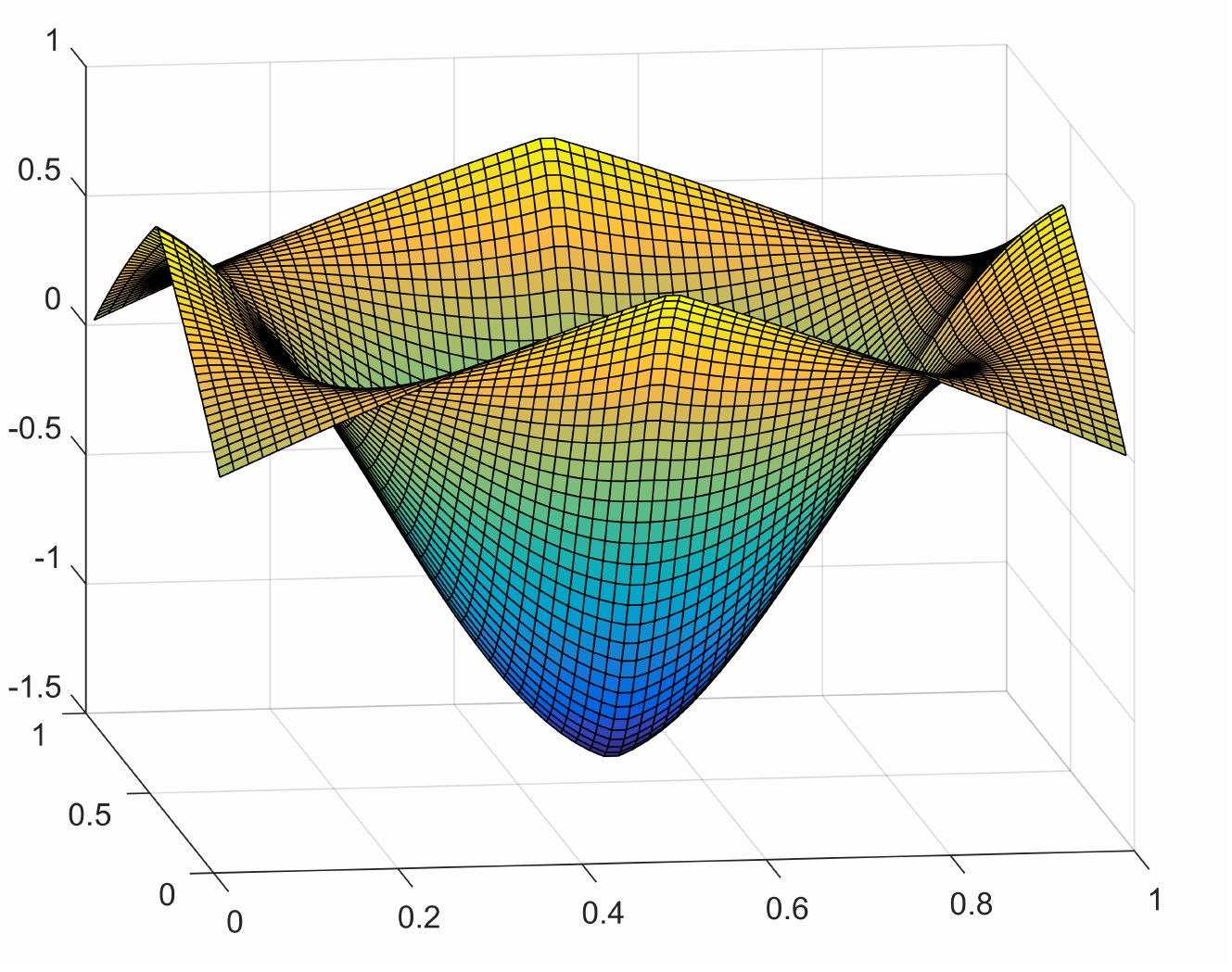}
}
\subfigure[Solution in the fractures]{
\label{fig: cross solution frac}
\includegraphics[scale=0.55]{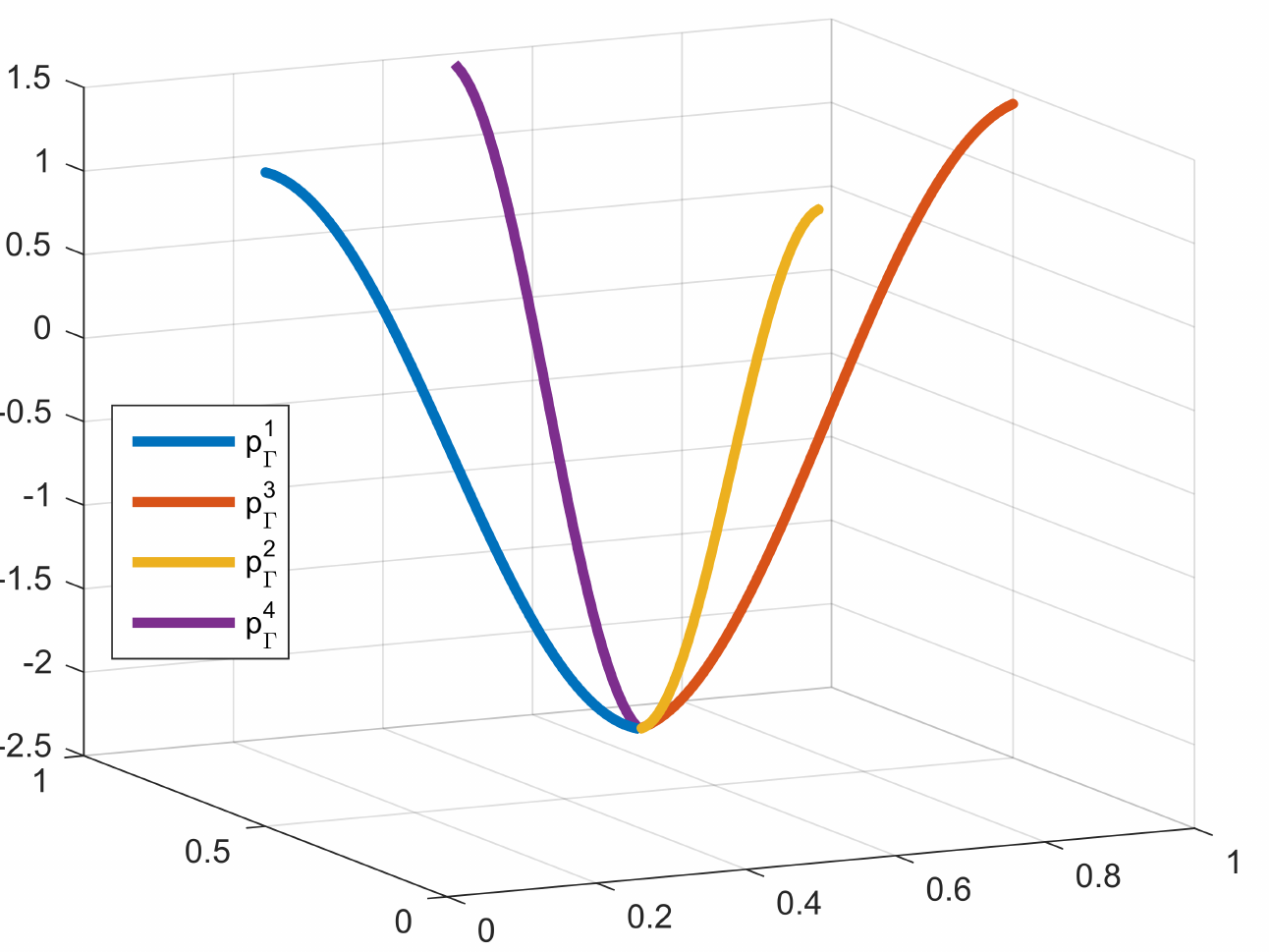}
}
\caption{Example 4: Computed bulk pressure (left) and computed fracture pressures plotted as $3d$-lines (right).}
\end{figure}

In Figures \ref{fig:vertical errori caso a}-\ref{fig:vertical errori caso b} we report the computed errors $||p-p_h||_{b,DG}$ and $||p_{\Gamma}- p_{\Gamma,h}||_{\Gamma,DG}$ in loglog scale for the bulk and fracture problems, respectively, as a function of the inverse of the mesh size $1/h$. On the left we show the results obtained for test case (a) and on the right for the case (b). Again, a convergence of order 2 is observed for both $||p-p_h||_{b,DG}$ and  $||p_{\Gamma}- p_{\Gamma,h}||_{\Gamma,DG}$, while a convergence of order 3 is observed for the error in the $L^2$-norm.

\begin{figure}[!htb]
\subfigure[]{
\label{fig:cross errori caso a}
   \resizebox{0.45\linewidth}{!}{
\begin{tikzpicture}
\scriptsize
\begin{loglogaxis}[
legend entries={$||p-p_h||_{b,DG}$,  $||p-p_h||_{L^2(\Omega)}$, $||p_{\Gamma}- p_{\Gamma,h}||_{\Gamma, DG}$, $||p_{\Gamma}- p_{\Gamma,h}||_{L^2(\Gamma)}$},
legend pos=outer north east,
legend style={font=\tiny, at={(0.01,0.01)},anchor=south west},
legend cell align=left,
 legend style={fill=none,draw=none},
xlabel=$1/h$
]
\addplot [blue, mark=triangle*, line width=0.3mm] table {cross/caso_1/errH1.txt};
\addplot [magenta, mark=otimes*, line width=0.3mm] table {cross/caso_1/errL2.txt};
\addplot [teal, mark=diamond*, line width=0.3mm] table {cross/caso_1/errfracH1.txt};
\addplot [green, mark=square*, line width=0.3mm] table {cross/caso_1/errfracL2.txt};

\def \xA {4.5254833995939038e+01};
\def \yA {4.8828125000000011e-04*8};
\def \xB {3.5355339059327342e+01 };
\def \yB {8.0000000000000145e-04*8};
\def \yC {0.5*\yA + 0.5*\yB};

\draw (\xA, \yA) -- (\xB, \yB)--(\xA, \yB)--cycle;
\draw (\xA, \yC) node [right] {2};

\def \xD {4.5254833995939038e+01 };
\def \yD {1.0789593218788875e-05*0.3};
\def \xE {3.5355339059327342e+01  };
\def \yE {2.2627416997969580e-05*0.3};
\def \yF {0.5*\yD + 0.5*\yE};

\draw (\xD, \yD) -- (\xE, \yE)--(\xE, \yD)--cycle;
\draw (\xE, \yF) node [left] {3};

%

\def \xG {4.5254833995939038e+01};
\def \yG {4.8828125000000011e-04*2.6};
\def \xH {3.5355339059327342e+01 };
\def \yH {8.0000000000000145e-04*2.6};
\def \yL {0.5*\yG + 0.5*\yH};

\draw (\xG, \yG) -- (\xH, \yH)--(\xH, \yG)--cycle;
\draw (\xH, \yL) node [left] {2};

\def \xM {4.5254833995939038e+01 };
\def \yM {1.0789593218788875e-05*0.7};
\def \xN {3.5355339059327342e+01  };
\def \yN {2.2627416997969580e-05*0.7};
\def \yQ {0.5*\yM + 0.5*\yN};

\draw (\xM, \yM) -- (\xN, \yN)--(\xM, \yN)--cycle;
\draw (\xM, \yQ) node [right] {3};

\end{loglogaxis}
\end{tikzpicture}
}
}
\subfigure[]{
\label{fig:cross errori caso b}
  \resizebox{0.45\linewidth}{!}{
\begin{tikzpicture}
\scriptsize
\begin{loglogaxis}[
legend entries={$||p-p_h||_{b,DG}$,  $||p-p_h||_{L^2(\Omega)}$, $||p^{\Gamma}- p_{\Gamma,h}||_{\Gamma, DG}$, $||p^{\Gamma}- p_{\Gamma,h}||_{L^2(\Gamma)}$},
legend pos=outer north east,
legend style={font=\tiny, at={(0.01,0.01)},anchor=south west},
legend cell align=left,
 legend style={fill=none,draw=none},
xlabel=$1/h$
]
\addplot [blue, mark=triangle*, line width=0.3mm] table {cross/caso_3/errH1.txt};
\addplot [magenta, mark=otimes*, line width=0.3mm] table {cross/caso_3/errL2.txt};
\addplot [teal, mark=diamond*, line width=0.3mm] table {cross/caso_3/errfracH1.txt};
\addplot [green, mark=square*, line width=0.3mm] table {cross/caso_3/errfracL2.txt};

\def \xA {4.5254833995939038e+01};
\def \yA {4.8828125000000011e-04*3.1};
\def \xB {3.5355339059327342e+01 };
\def \yB {8.0000000000000145e-04*3.1};
\def \yC {0.5*\yA + 0.5*\yB};

\draw (\xA, \yA) -- (\xB, \yB)--(\xB, \yA)--cycle;
\draw (\xB, \yC) node [left] {2};

\def \xD {4.5254833995939038e+01 };
\def \yD {1.0789593218788875e-05*0.3};
\def \xE {3.5355339059327342e+01  };
\def \yE {2.2627416997969580e-05*0.3};
\def \yF {0.5*\yD + 0.5*\yE};

\draw (\xD, \yD) -- (\xE, \yE)--(\xE, \yD)--cycle;
\draw (\xE, \yF) node [left] {3};

%

\def \xG {4.5254833995939038e+01};
\def \yG {4.8828125000000011e-04*7.5};
\def \xH {3.5355339059327342e+01 };
\def \yH {8.0000000000000145e-04*7.5};
\def \yL {0.5*\yG + 0.5*\yH};

\draw (\xG, \yG) -- (\xH, \yH)--(\xG, \yH)--cycle;
\draw (\xG, \yL) node [right] {2};

\def \xM {4.5254833995939038e+01 };
\def \yM {1.0789593218788875e-05*0.6};
\def \xN {3.5355339059327342e+01  };
\def \yN {2.2627416997969580e-05*0.6};
\def \yQ {0.5*\yM + 0.5*\yN};

\draw (\xM, \yM) -- (\xN, \yN)--(\xM, \yN)--cycle;
\draw (\xM, \yQ) node [right] {3};

\end{loglogaxis}
\end{tikzpicture}
}
}

\caption{Example 4: Computed errors in the bulk and in the fractures as a function of the inverse of the mesh size (loglog scale). Case (a) on the left and case (b) on the right.} 
\end{figure}
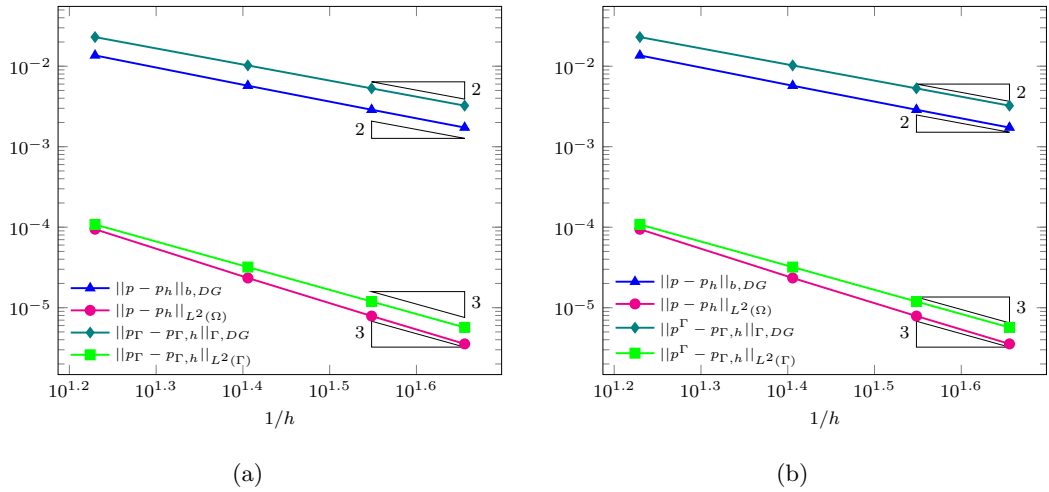

\bibliographystyle{abbrv}
\bibliography{biblio}{}

\begin{thebibliography}{10}

\bibitem{Alboin2}
C.~Alboin, J.~Jaffr{\'e}, J.~E. Roberts, and C.~Serres.
\newblock Modeling fractures as interfaces for flow and transport in porous
  media.
\newblock In {\em Fluid flow and transport in porous media: mathematical and
  numerical treatment ({S}outh {H}adley, {MA}, 2001)}, volume 295 of {\em
  Contemp. Math.}, pages 13--24. Amer. Math. Soc., Providence, RI, 2002.

\bibitem{Alboin}
C.~Alboin, J.~Jaffr{\'e}, J.~E. Roberts, X.~Wang, and C.~Serres.
\newblock Domain decomposition for some transmission problems in flow in porous
  media.
\newblock In {\em Numerical treatment of multiphase flows in porous media},
  volume 552 of {\em Lecture Notes in Phys.}, pages 22--34. Springer, Berlin,
  2000.

\bibitem{Angot}
P.~Angot, F.~Boyer, and F.~Hubert.
\newblock Asymptotic and numerical modelling of flows in fractured porous
  media.
\newblock {\em M2AN Math. Model. Numer. Anal.}, 43(2):239--275, 2009.

\bibitem{antonietti2015review}
P.~F. Antonietti, A.~Cangiani, J.~Collis, Z.~Dong, E.~H. Georgoulis, S.~Giani,
  and P.~Houston.
\newblock {\em Review of discontinuous {G}alerkin {F}inite {E}lement {M}ethods
  for partial differential equations on complicated domains}, volume 114 of
  {\em {L}ecture {N}otes in {C}omputational {S}cience and {E}ngineering},
  chapter~8, pages 281 -- 310.
\newblock Springer, 1st edition, 2016.

\bibitem{DGcurvihighorder}
P.~F. Antonietti, A.~Dedner, P.~Madhavan, S.~Stangalino, B.~Stinner, and
  M.~Verani.
\newblock High order discontinuous {G}alerkin methods for elliptic problems on
  surfaces.
\newblock {\em SIAM J. Numer. Anal.}, 53(2):1145--1171, 2015.

\bibitem{AntoniettiFacciolaHoustonMazzieriPennesiVerani_2020}
P.~F. Antonietti, C.~Facciol\`a, P.~Houston, I.~Mazzieri, G.~Pennesi, and
  M.~Verani.
\newblock {High--order discontinuous {G}alerkin methods on polyhedral grids for
  geophysical applications: seismic wave propagation and fractured reservoir
  simulations}.
\newblock In D. Di Pietro, L. Formaggia, and R. Masson (eds.),
  \textit{Polyhedral Methods in Geosciences}, SEMA-SIMAI Springer series, to
  appear, 2020.

\bibitem{mioUnafrattura}
P.~F. Antonietti, C.~Facciol\`a, A.~Russo, and M.~Verani.
\newblock Discontinuous {G}alerkin {A}pproximation of {F}lows in {F}ractured
  {P}orous {M}edia on {P}olytopic {G}rids.
\newblock {\em SIAM J. Sci. Comput.}, 41(1):A109--A138, 2019.

\bibitem{mioUnified}
P.~F. Antonietti, C.~Facciol\`a, and M.~Verani.
\newblock Unified analysis of {D}iscontinuous {G}alerkin approximations of
  flows in fractured porous media on polygonal and polyhedral grids.
\newblock {\em Mathematics in Engineering}, 2(2):340--385, 2020.

\bibitem{AnFoScVeNi2016}
P.~F. Antonietti, L.~Formaggia, A.~Scotti, M.~Verani, and N.~Verzotti.
\newblock {M}imetic {F}inite {D}ifference approximation of flows in fractured
  porous media.
\newblock {\em ESAIM Math. Model. Numer. Anal.}, 50(3):809--832, 2016.

\bibitem{antonietti2013hp}
P.~F. Antonietti, S.~Giani, and P.~Houston.
\newblock $hp$-version {C}omposite {D}iscontinuous {G}alerkin methods for
  elliptic problems on complicated domains.
\newblock {\em SIAM J. Sci. Comput}, 35(3):A1417--A1439, 2013.

\bibitem{antonietti2014domain}
P.~F. Antonietti, S.~Giani, and P.~Houston.
\newblock Domain decomposition preconditioners for discontinuous {G}alerkin
  methods for elliptic problems on complicated domains.
\newblock {\em J. Sci. Comput.}, 60(1):203--227, 2014.

\bibitem{Arnold82}
D.~N. Arnold.
\newblock An interior penalty finite element method with discontinuous
  elements.
\newblock {\em SIAM J. Numer. Anal.}, 19(4):742--760, 1982.

\bibitem{Arnoldbrezzicockburnmarini2002}
D.~N. Arnold, F.~Brezzi, B.~Cockburn, and L.~D. Marini.
\newblock Unified analysis of discontinuous {G}alerkin methods for elliptic
  problems.
\newblock {\em SIAM J. Numer. Anal.}, 39(5):1749--1779, 2001/02.

\bibitem{BerroneVEM1}
M.~F. Benedetto, S.~Berrone, S.~Pieraccini, and S.~Scial{\`o}.
\newblock The {V}irtual {E}lement {M}ethod for discrete fracture network
  simulations.
\newblock {\em Comput. Methods Appl. Mech. Engrg.}, 280:135--156, 2014.

\bibitem{BerroneVEM2}
M.~F. Benedetto, S.~Berrone, and S.~Scial{\`o}.
\newblock A globally conforming method for solving flow in discrete fracture
  networks using the {V}irtual {E}lement {M}ethod.
\newblock {\em Finite Elem. Anal. Des.}, 109:23--36, 2016.

\bibitem{boon2018robust}
W.~M. Boon, J.~M. Nordbotten, and I.~Yotov.
\newblock Robust discretization of flow in fractured porous media.
\newblock {\em SIAM J. Numer. Anal.}, 56(4):2203--2233, 2018.

\bibitem{brennerhennicker}
K.~Brenner, J.~Hennicker, R.~Masson, and P.~Samier.
\newblock Gradient discretization of hybrid-dimensional {D}arcy flow in
  fractured porous media with discontinuous pressures at matrix--fracture
  interfaces.
\newblock {\em IMA J. Numer. Anal.}, 37(3):1551--1585, 2016.

\bibitem{cutFem}
E.~Burman, P.~Hansbo, M.~G. Larson, and K.~Larsson.
\newblock Cut finite elements for convection in fractured domains.
\newblock {\em Computers \& Fluids}, 179:726--734, 2019.

\bibitem{cangiani2016hp}
A.~Cangiani, Z.~Dong, and E.~H. Georgoulis.
\newblock $hp$-version space-time discontinuous {G}alerkin methods for
  parabolic problems on prismatic meshes.
\newblock {\em SIAM J. Sci. Comput}, 39(4):A1251--A1279, 2017.

\bibitem{poligoni2}
A.~Cangiani, Z.~Dong, E.~H. Georgoulis, and P.~Houston.
\newblock $hp$-version discontinuous {G}alerkin methods for
  advection-diffusion-reaction problems on polytopic meshes.
\newblock {\em ESAIM Math. Model. Numer. Anal.}, 50(3):699--725, 2016.

\bibitem{libropoligoni}
A.~Cangiani, Z.~Dong, E.~H. Georgoulis, and P.~Houston.
\newblock {\em $hp$-version discontinuous {G}alerkin methods on polytopic
  meshes}.
\newblock SpringerBriefs in Mathematics. Springer International Publishing,
  2017.

\bibitem{poligoni1}
A.~Cangiani, E.~H. Georgoulis, and P.~Houston.
\newblock {$hp$}-version discontinuous {G}alerkin methods on polygonal and
  polyhedral meshes.
\newblock {\em Math. Models Methods Appl. Sci.}, 24(10):2009--2041, 2014.

\bibitem{florent}
F.~A. Chave, D.~Di~Pietro, and L.~Formaggia.
\newblock {A Hybrid High-Order method for Darcy flows in fractured porous
  media}.
\newblock {\em SIAM J. Sci. Comput.}, 40(2):A1063--A1094, 2018.

\bibitem{chernyshenko2019unfitted}
A.~Y. Chernyshenko and M.~A. Olshanskii.
\newblock An unfitted finite element method for the darcy problem in a fracture
  network.
\newblock {\em arXiv preprint arXiv:1903.06351}, 2019.

\bibitem{dangeloscotti}
C.~D'{A}ngelo and A.~Scotti.
\newblock A mixed finite element method for darcy flow in fractured porous
  media with non-matching grids.
\newblock {\em ESAIM Math. Model. Numer. Anal.}, 46(02):465--489, 2012.

\bibitem{DGcurvilineari}
A.~Dedner, P.~Madhavan, and B.~Stinner.
\newblock Analysis of the discontinuous {G}alerkin method for elliptic problems
  on surfaces.
\newblock {\em IMA J. Numer. Anal.}, 33(3):952--973, 2013.

\bibitem{reviewXFEMfratture}
B.~Flemisch, A.~Fumagalli, and A.~Scotti.
\newblock A review of the {XFEM}-based approximation of flow in fractured
  porous media.
\newblock In {\em Advances in Discretization Methods}, pages 47--76. Springer,
  2016.

\bibitem{ruffo}
L.~Formaggia, A.~Fumagalli, A.~Scotti, and P.~Ruffo.
\newblock A reduced model for darcy's problem in networks of fractures.
\newblock {\em ESAIM Math. Model. Numer. Anal.}, 48(4):1089--1116, 2014.

\bibitem{formaggiascottisottocasa}
L.~Formaggia, A.~Scotti, and F.~Sottocasa.
\newblock Analysis of a mimetic finite difference approximation of flows in
  fractured porous media.
\newblock {\em ESAIM Math. Model. Numer. Anal.}, 52(2):595--630, 2018.

\bibitem{Frih}
N.~Frih, J.~E. Roberts, and A.~Saada.
\newblock Modeling fractures as interfaces: a model for {F}orchheimer
  fractures.
\newblock {\em Comput. Geosci.}, 12(1):91--104, 2008.

\bibitem{fumagallikeiscialo}
A.~Fumagalli, E.~Keilegavlen, and S.~Scial{\`o}.
\newblock Conforming, non-conforming and non-matching discretization couplings
  in discrete fracture network simulations.
\newblock {\em J. Comput. Phys.}, 376:694--712, 2019.

\bibitem{fumagalliscotti}
A.~Fumagalli and A.~Scotti.
\newblock A numerical method for two-phase flow in fractured porous media with
  non-matching grids.
\newblock {\em Adv. Water Resour.}, 62, Part C:454--464, 2013.

\bibitem{martinjaffreroberts}
V.~Martin, J.~Jaffr{\'e}, and J.~E. Roberts.
\newblock Modeling fractures and barriers as interfaces for flow in porous
  media.
\newblock {\em SIAM J. Sci. Comput.}, 26(5):1667--1691, 2005.

\bibitem{flemischWohlmuth}
N.~Schwenck, B.~Flemisch, R.~Helmig, and B.~I. Wohlmuth.
\newblock Dimensionally reduced flow models in fractured porous media:
  crossings and boundaries.
\newblock {\em Comput. Geosci.}, 19(6):1219--1230, 2015.

\bibitem{stein}
E.~M. Stein.
\newblock {\em Singular integrals and differentiability properties of
  functions}, volume~2.
\newblock Princeton university press, 1970.

\bibitem{strang}
G.~Strang and G.~J. Fix.
\newblock {\em An analysis of the finite element method}, volume 212.
\newblock Prentice-Hall Englewood Cliffs, NJ, 1973.

\bibitem{Wheeler78}
M.~F. Wheeler.
\newblock An elliptic collocation-finite element method with interior
  penalties.
\newblock {\em SIAM J. Numer. Anal.}, 15(1):152--161, 1978.

\end{thebibliography}

\end{document}